\documentclass[12pt]{nmsuth01}
\usepackage{amsmath}
\usepackage[square, numbers, comma, sort&compress]{natbib}
\usepackage{graphicx,psfrag,dsfont,empheq,hyperref,cleveref}
\hypersetup{urlcolor=blue, colorlinks=true}
\usepackage[dvips]{epsfig}
\usepackage{epstopdf}
\usepackage{latexsym}

\usepackage{wallpaper}
\usepackage{mdframed}
\usepackage{amsthm,amsfonts,mathrsfs}
\usepackage[all]{xy}
\graphicspath{{./Figures/}} 
\usepackage{titlesec}

\titleformat{\paragraph}
{\normalfont\normalsize\bfseries}{\theparagraph}{1em}{}
\titlespacing*{\paragraph}
{0pt}{3.25ex plus 1ex minus .2ex}{1.5ex plus .2ex}

\newcommand{\SO}{\ensuremath{\mathrm{SO(3)}}}
\newcommand{\TSO}{\ensuremath{\mathrm{TSO(3)}}}
\newcommand{\Ta}{\ensuremath{\mathrm{T}}}
\newcommand{\T}{^{\mbox{\small T}}}
\newcommand{\so}{\ensuremath{\mathfrak{so}(3)}}
\newcommand{\SE}{\ensuremath{\mathrm{SE(3)}}}
\newcommand{\se}{\ensuremath{\mathfrak{se}(3)}}
\newcommand{\bR}{\ensuremath{\mathbb{R}}}
\newcommand{\bS}{\ensuremath{\mathbb{S}}}

\newcommand{\mrm}{\mathrm}

\newcommand{\diag}{\mbox{diag}}
\newcommand{\bbm}{\begin{bmatrix}}
\newcommand{\ebm}{\end{bmatrix}}
\newcommand{\matl}{\left[ \begin{array}}
\newcommand{\matr}{\end{array} \right]}
\newcommand{\be}{\begin{equation}}
\newcommand{\ee}{\end{equation}}
\newcommand{\bea}{\begin{eqnarray}}
\newcommand{\eea}{\end{eqnarray}}
\newcommand{\beas}{\begin{eqnarray*}}
\newcommand{\eeas}{\end{eqnarray*}}
\newcommand{\nn}{\nonumber}
\newcommand{\mC}{\mathcal{C}}
\newcommand{\cE}{\mathcal{E}}
\newcommand{\cL}{\mathcal{L}}

\newcommand{\cM}{\mathcal{M}}

\newcommand{\cT}{\mathcal{T}}
\newcommand{\cU}{\mathcal{U}}

\newcommand{\di}{\mathrm{d}}

\newcommand{\tr}{\mathrm{trace}}

\newcommand{\lan}{\langle}
\newcommand{\ran}{\rangle}
\renewcommand{\skew}[1]{\left(#1\right)^{\times}}
\newcommand{\cS}{\mathcal{S}}
\newcommand{\cI}{\mathcal{I}}

\newcommand{\ad}[1]{{\mathrm{ad}_{#1}}}          			
\newcommand{\adast}[1]{{\mathrm{ad}_{#1}^\ast}}  			
\newcommand{\Ad}[1]{{\mathrm{Ad}_{#1}}}  			
\newcommand{\s}[1] {{\scriptscriptstyle #1}}
\newcommand{\Iref}{{\{I\}}} 
\newcommand{\Bref}{{\{B\}}} 
\renewcommand{\skew}[1]{(#1)^{\times}}

\newcommand{\Imbb}{\mathbb{I}}

\newcommand{\bJ}{\mathbb{J}}

\newcommand{\mpz}{\mathpzc}
\newcommand{\msg}{\mathsf{g}}
\newcommand{\msh}{\mathsf{h}}

\newcommand{\sS}{\mathsf{S}}
\newcommand{\sO}{\mathsf{O}}
\newcommand{\bD}{\mathbb{D}}
\def\thesection{\arabic{section}}

\newcommand{\bi}{\begin{itemize}}
\newcommand{\ei}{\end{itemize}}
\newcommand{\sD}{\mathscr{D}}
\newcommand{\dsP}{\mathds{P}}

\DeclareMathAlphabet{\mathpzc}{OT1}{pzc}{m}{it}
\DeclareMathOperator{\logm}{{logm}}  			
\DeclareMathOperator{\expm}{{expm}}  			

\theoremstyle{plain}
\newtheorem{lemma}{Lemma}[section]
\newtheorem{assumption}{Assumption}[section]
\newtheorem{proposition}{Proposition}[section]
\newtheorem{theorem}{Theorem}[section]

\theoremstyle{definition}

\theoremstyle{remark}
\newtheorem*{remark}{Remark}
\newtheorem*{convention*}{Convention}

\newlength{\singlespace}
\newlength{\doublespace}
\begin{document}
\setlength{\baselineskip}{\doublespace}
%
%
\pagenumbering{roman}
\pagestyle{empty}
\renewcommand{\baselinestretch}{2}
\begin{center}
STABLE ESTIMATION OF RIGID BODY MOTION \\
USING GEOMETRIC MECHANICS\\
\vspace{0.1in}
BY\\
\vspace{0.1in}
MAZIAR IZADI, B.S., M.S.
\end{center}
\vspace{1.0in}
\begin{center}
A dissertation submitted to the Graduate School\\
\vspace{0.1in}
in partial fulfillment of the requirements\\
\vspace{0.1in}
for the degree \\
\vspace{0.1in}
Doctor of Philosophy, Engineering
\end{center}
\vfill
\begin{center}
Specialization in Mechanical Engineering
\end{center}
\vspace{1.0in}
\begin{center}
New Mexico State University\\
\vspace{0.1in}
Las Cruces, New Mexico\\
\vspace{0.1in}
December 2015
\end{center}

\pagestyle{plain}
\noindent
``Stable Estimation of Rigid Body Motion Using Geometric Mechanics,'' 
a dissertation prepared by
Maziar Izadi 
in partial fulfillment of the requirements for the degree, 
Doctor of Philosophy, Engineering,
has been approved and accepted by the following:

\vspace{0.1in}
\begin{flushleft}
\hrulefill
\newline
{\large \href{louireye@nmsu.edu}{Loui Reyes}}
\newline
Interim Dean of the Graduate School
\vspace{0.5in}

\hrulefill
\newline
{\large \href{aksanyal@syr.edu}{Amit K. Sanyal}}
\newline
Chair of the Examining Committee
\vspace{0.5in}

\hrulefill
\newline
Date
\vspace{0.2in}
\newline
Committee in charge:
\end{flushleft}

{\large \href{http://eng-cs.syr.edu/about-the-college/faculty-and-staff/sanyal}{Dr. Amit K. Sanyal}}, Chair

{\large \href{http://mae.nmsu.edu/~oma}{Dr. Ou Ma}}

{\large\href{http://ece.nmsu.edu/faculty-staff/rprasad}{Dr. Nadipuram Prasad}}

{\large\href{http://www.math.nmsu.edu/Vitas/barany.html}{Dr. Ernest Barany}}

%
\begin{center}
DEDICATION
\end{center}

I dedicate this work to my mother Mahnaz, my father Col. E. Izadi,
and my brothers Mohsen and Moein.

%
\begin{center}
ACKNOWLEDGMENTS
\end{center}

I am deeply appreciative of the many individuals who have supported my work and continually encouraged me through the writing of this dissertation. I am extremely grateful for everyone who gave me their time, encouragement, feedback, and patience. Without these, I would not have been able to see it through.

Above all, I would like to express my deepest gratitude to my advisor, Dr. Amit K. Sanyal, for his inspirational and timely advice and untiring encouragement. Personally, I would like to thank him
for sharing his knowledge which has enriched my study in
Mechanical/Aerospace Engineering area. I really appreciate the amount of time he spent to teach me many skills in research that maybe not every PhD student is blessed enough to learn from their advisor.

I would like to thank Dr. Nadipuram Prasad, Dr. Ou Ma, and Dr. Ernest Barany for their guidance, discussions, ideas, and feedback. I was so fortunate to take many classes with them on Mathematics and Controls, which gave me great insight and enabled me to accomplish this work.

My sincere thanks go to Drs. Ian Leslie and Gabe Garcia, former and current Department Heads of Mechanical and Aerospace Engineering at NMSU, for their support before and after my defense. Special gratitude also goes to the staff at NMSU, particularly Milen Bartnick and Margaret Vasquez, for their support with the paperwork ensuring the smooth completion of my doctoral studies. Moreover, I owe a debt of gratitude to wonderful people at ASNMSU and GSC for helping me attend four conferences over the past three years.

I would especially like to thank my amazing family for the love, support, and
encouragement I have gotten over the years. I undoubtedly
could not have done this without them. Thanks also to all the amazing fellow graduate students during my doctoral studies: both those involved in some research with me (Dr. S{\'e}rgio Br{\'a}s, Dr. Daero Lee, Dr. Morad Nazari, Dr. Sashi Prabhakaran, Dr. Ehsan Samiei, and Gaurav Misra), as well as other significant graduate student friends who accompanied me throughout my
journey. Thanks for your support, encouragement, and being there whenever I needed a friend.

Finally, but not least, I am grateful for absolutely invaluable advice from Dr. Mohammad (Behzad) Zamani, Dr. Paulo Oliveira, and Dr. Carlos Silvestre which greatly improved the quality of my research and papers. Furthermore, I gratefully acknowledge the support from the National Science Foundation through grant CMMI 1131643.

%
\begin{center}
            VITA
\end{center}
\begin{flushleft}
\begin{tabular}{ll}
April 1985 &  Born in Iran
\\
& \\
2003-2007        &  B.S., Amirkabir University of Technology (Tehran Polytechnic),
\\ &Tehran, Iran
\\
& \\
2007-2010        &  M.S., University of Tehran, 
\\               
                 &   Tehran, Iran
\\
 & \\
2012-2015        &  Research Assistant, Mechanical and Aerospace Engineering Department,
\\
                 &  New Mexico State University. 
\\
& \\
2014-2015        &  Teaching Assistant, Mechanical and Aerospace Engineering Department,
\\
                 &  New Mexico State University. 
\\
& \\
2012-2015        &  Ph.D., New Mexico State University.
\\
& \\
\end{tabular}
\end{flushleft}
\vspace{0.1in}
\begin{center}
PROFESSIONAL  AND HONORARY SOCIETIES
\end{center}
\begin{flushleft}
The Institute of Electrical and Electronics Engineers (IEEE)\\
The American Society of Mechanical Engineers (ASME)
\end{flushleft}
\vspace{0.1in}
\begin{center}PUBLICATIONS
\end{center}

\section*{Journal Papers}
\begin{enumerate}
\item {Izadi,~M., \& Sanyal,~A.~K. (2014).
\newblock Rigid body attitude estimation based on the {L}agrange-d'{A}lembert
  principle.
\newblock {\em Automatica}, 50(10), 2570--2577.}

\item {Izadi,~M., \& Sanyal,~A.~K. (2015).
\newblock Rigid body pose estimation based on the {L}agrange-d'{A}lembert principle.
\newblock To appear in {\em Automatica}.}

\item{Br{\'a}s,~S., Izadi,~M., Silvestre,~C., Sanyal,~A., \& Oliveira,~P. (2015).
\newblock Nonlinear Observer for 3{D} Rigid Body Motion Estimation Using Doppler Measurements. \newblock To appear in \emph{IEEE Transactions on Automatic Control.}}

\item {Misra,~G., Izadi,~M., Sanyal,~A.~K., \& Scheeres,~D.~J. (2015).
\newblock Coupled orbit-attitude dynamics and relative state estimation of
  spacecraft near small Solar System bodies.
\newblock To appear in {\em Advances in Space Research}.}

\item {Izadi,~M., Sanyal,~A., Silvestre,~C., \& Oliveira,~P. (2015).
\newblock The Variational Attitude Estimator in the Presence of Bias in Angular Velocity Measurements. \newblock Under preparation.}

\item{Viswanathan,~S.~P., Sanyal,~A.~K., \& Izadi,~M. (2015).
\newblock Smartphone-Based Spacecraft Attitude Determination and Control
System (ADCS) using Internal Momentum Exchange Actuators. \newblock Under preparation.}
\end{enumerate}

\section*{Conference Papers}
\begin{enumerate}
\item{Izadi,~M., Bohn,~J., Lee,~D., Sanyal,~A., Butcher,~E., \& Scheeres,~D. (2013).
\newblock A Nonlinear Observer Design for a Rigid Body in the Proximity of a Spherical Asteroid.
\newblock In {\em Proceedings of the ASME Dynamics Systems and Control Conference}. Palo Alto, CA, USA.}

\item{Br{\'a}s,~S., Izadi,~M., Silvestre,~C., Sanyal,~A., \& Oliveira,~P. (2013).
\newblock Nonlinear Observer for 3D Rigid Body Motion.
\newblock In {\em Proceedings of the $52^{nd}$ IEEE Conference on Decision and Control}. Florence, Italy.}

\item{Sanyal,~A., Izadi,~M., \& Bohn,~J. (2014).
\newblock An Observer for Rigid Body Motion with Almost Global Finite-time Convergence.
\newblock In {\em Proceedings of the ASME Dynamics Systems and Control Conference}. San Antonio, TX, USA.}

\item{Lee,~D., Izadi,~M., Sanyal,~A., Butcher,~E., \& Scheeres,~D. (2014).
\newblock Finite-Time Control for Body-Fixed Hovering of Rigid Spacecraft over an Asteroid.
\newblock In {\em Proceedings of the AAS/AIAA Space Flight Mechanics}. Santa Fe, NM, USA.}

\item {Sanyal,~A.~K., Izadi,~M., \& Butcher,~E.~A. (2014).
\newblock Determination of relative motion of a space object from simultaneous
  measurements of range and range rate.
\newblock In {\em Proceedings of the American Control Conference} (pp. 1607--1612). Portland, OR, USA.}

\item {Sanyal,~A.~K., Izadi,~M., Misra,~G., Samiei,~E., \& Scheeres,~D.~J. (2014).
\newblock Estimation of Dynamics of Space Objects from Visual Feedback During Proximity Operations.
\newblock In {\em Proceedings of the AIAA Space Conference}. San Diego, CA, USA.}

\item {Viswanathan,~S.~P., Sanyal,~A.~K., \& Izadi,~M. (2015).
\newblock Mechatronics Architecture of Smartphone Based Spacecraft ADCS using VSCMG Actuators.
\newblock In {\em Proceedings of the Indian Control Conference}. Chennai, India.}

\item{Izadi,~M., Samiei,~E., Sanyal,~A.~K., \& Kumar,~V. (2015).
\newblock Comparison of an attitude estimator based on the
  {L}agrange-d'{A}lembert principle with some state-of-the-art filters.
\newblock In {\em Proceedings of the IEEE International Conference on Robotics and
  Automation} (pp. 2848--2853). Seattle, WA, USA.}

\item{Samiei,~E., Izadi,~M., Viswanathan,~S.~P., Sanyal,~A.~K., \& Butcher,~E.~A. (2015).
\newblock Robust Stabilization of Rigid Body Attitude Motion in the
  Presence of a Stochastic Input Torque.
\newblock In {\em Proceedings of the IEEE International Conference on Robotics and
  Automation} (pp. 428--433). Seattle, WA, USA.}

\item{Samiei,~E., Izadi,~M., Viswanathan,~S.~P., Sanyal,~A.~K., \& Butcher,~E.~A. (2015).
\newblock Delayed Feedback Asymptotic Stabilization of Rigid Body Attitude Motion
  for Large Rotations.
\newblock In {\em Proceedings of the 12$^{th}$ IFAC Workshop on Time Delay Systems}. Ann Arbor, MI, USA.}

\item{Izadi,~M., Sanyal,~A.~K., Samiei,~E., \& Viswanathan,~S.~P. (2015).
\newblock Discrete-time rigid body attitude state estimation based on the
  discrete {L}agrange-d'{A}lembert principle.
\newblock In {\em Proceedings of the American Control Conference} (pp. 3392--3397). Chicago, IL, USA.}

\item{Izadi,~M., Sanyal,~A.~K., Barany,~E., \& Viswanathan,~S.~P. (2015).
\newblock Rigid Body Motion Estimation based on the Lagrange-d'Alembert Principle.
\newblock In {\em Proceedings of the $54^{nd}$ IEEE Conference on Decision and Control}. Osaka, Japan.}

\item{Izadi,~M., Sanyal,~A.~K., Beard,~R.~W., \& Bai,~H. (2015).
\newblock GPS-Denied Relative Motion Estimation for Fixed-Wing UAV using the Variational Pose Estimator.
\newblock In {\em Proceedings of the $54^{nd}$ IEEE Conference on Decision and Control}. Osaka, Japan.}

\item{Izadi,~M., Sanyal,~A.~K., Silvestre,~C., \& Oliveira,~P. (2016).
\newblock The Variational Attitude Estimator in the Presence of Bias in Angular Velocity Measurements.
\newblock Submitted to {\em the American Control Conference}. Boston, MA, USA.}

\end{enumerate}

\begin{center}
FIELD OF STUDY
\end{center}
\begin{flushleft}
Major Field: Mechanical Engineering
\end{flushleft}
%
\begin{center}
ABSTRACT
\end{center}
\vspace{0.3in}
\begin{center}
STABLE ESTIMATION OF RIGID BODY MOTION\\
USING GEOMETRIC MECHANICS
\\
BY
\\
MAZIAR IZADI, B.S., M.S.
\end{center}
\vspace{0.3in}
\begin{center}
Doctor of Philosophy, Engineering

New Mexico State University

Las Cruces, New Mexico, 2015

Dr. Amit K. Sanyal, Chair
\end{center}
\vspace{0.3in}
\hspace{\parindent}
In this work, asymptotically stable state estimation schemes are proposed for rigid body motion, using the framework of geometric mechanics. Rigorous stability analyses of the estimation schemes presented here guarantee the nonlinear stability of these schemes. The stability of these schemes does not depend on the characteristics of the sensor measurement noise or external disturbances. In addition, they are robust to initial errors in the state estimates and do not need to be re-tuned when sensor noise properties change. In the first part of this dissertation, estimation of rigid body states is considered, given the dynamics model of the rigid body. In the second part, an estimation scheme that does not require knowledge of the dynamics of the rigid body is derived, based on onboard sensor measurements obtained at an appropriate frequency. The frequency of such measurements must be suitably high to resolve the motion of the rigid body. These attitude and pose estimation schemes are obtained by applying the Lagrange-d'Alembert principle from variational mechanics, to a Lagrangian constructed from state estimation errors and a dissipative term linear in the velocity estimation errors.


\tableofcontents
\newpage
\listoftables
\addcontentsline{toc}{section}{LIST OF TABLES}
\newpage
\listoffigures
\addcontentsline{toc}{section}{LIST OF FIGURES}
\newpage
\pagenumbering{arabic}
%

\section{INTRODUCTION} 

\label{intro} 

\hspace{\parindent}

\hspace{.25in}Estimation of rigid body motion is a long-standing problem of interest for a wide variety of mechanical systems. Specifically, these systems include aerial and under-water vehicles, spacecraft, or any other moving objects in three dimensions. Motion estimation for rigid bodies is challenging primarily because this motion is described by 
nonlinear dynamics and the state space is nonlinear. This nonlinearity arises from the intrinsic nature of rigid body attitude, which is represented by the special orthogonal group, $\SO$. Throughout this dissertation, rigid body attitude is represented globally over the configuration space of rigid body attitude motion without using local coordinates or quaternions. Attitude estimators using unit quaternions for attitude representation may be {\em unstable 
in the sense of Lyapunov}, unless they identify antipodal quaternions with a single attitude. 
This is also the case for attitude control schemes based on continuous feedback of unit 
quaternions, as shown in~\cite{Bayadi2014almost,sanyal2009inertia,chaturvedi2011rigid}. One adverse 
consequence of these unstable estimation and control schemes is that they end up taking 
longer to converge compared with stable schemes under similar initial conditions and initial 
transient behavior. On the contrary, all the estimation schemes proposed here are stable in the sense of \textit{Lyapunov}.

\hspace{.25in}In the first phase of this work, which includes Chapters \ref{CH02_DSCC2013} and \ref{CH03_DSCC2014}, three instances of such estimation schemes are proposed for rigid body motion using knowledge of dynamics. This requires the knowledge of the physical properties of the rigid body, as well as all external forces and moments applied on it. In these chapters, exponential coordinates are used to represent rigid body configuration. 
An observer design for arbitrary rigid-body motion in the proximity of a spherical asteroid 
of unknown mass is considered in Chapter \ref{CH02_DSCC2013}. This observer exhibits almost global convergence of 
state estimates in the state space of rigid body rotations and translations. Continuous 
observers cannot be globally asymptotically stable in this state space, which is the tangent 
bundle of the Lie group $\SE$, due to topological obstructions arising from the fact that this 
state space is not contractible~\cite{bhat}. Most unmanned and manned vehicles can be accurately modeled as rigid bodies, and therefore this observer can be applied to such vehicles operating on air, underwater, 
and in space. In particular, such vehicles when operated in uncertain or poorly known 
environments, can be subject to unknown forces and moments. Therefore, estimation of 
parameters associated with such unknown forces and moments is also of value. 
Dynamical coupling between the rotational and translational dynamics, which occurs both 
due to the natural dynamics as well as control forces and torques, is treated directly in the 
geometric mechanics framework used for our observer design.

Relevant prior research on observer designs for rigid body dynamics in $\SE$ is briefly 
covered here. A nonlinear observer for integration of GNSS and IMU measurements in the presence of gyro bias was investigated in \cite{GNSS} by using inertial reading of acceleration 
and velocities, magnetometer measurements and satellite-based measurements.
Using landmark measurements and noisy velocity data, a nonlinear observer for pose 
estimation in $\SE$ is presented in \cite{Vas1}. Ideal inertial velocity readings decouples 
the position and attitude motions, whereas they are coupled in the presence of gyro rate bias. 
The work in \cite{RodImp} proposes an observer in the special Euclidean group $\SE$ and considers the conditions under which the estimated states converge to the real states exponentially fast. It is also shown that in the case there exist some measurement noise, the estimate converges to a neighborhood of the real state. A global exponential stable attitude observer is presented in \cite{GES}. Although this observer does not evolve on $\SO$, it yields estimates that converges asymptotically to $\SO$ and as a result, it does not have any topological limitations. A nonlinear observer using active vision and inertial measurements that estimates the attitude of a rigid body is verified experimentally in \cite{Sergio1,Sergio2}. An almost globally convergent orientation estimator is presented in \cite{6315375} when just a single body-fixed vector on the rotating rigid body is available. In \cite{Zarrouati}, with the knowledge of a camera dynamics and recalling a system of partial differential equations describing the invariant dynamics of brightness and depth smooth fields, an $\SO$-invariant variational method to directly estimate the depth field is investigated.
There are some novel methods to derive the nonlinear state observers designed directly on the Lie group structure of the Special Euclidean group $\SE$ called gradient-based observer design. A type of nonlinear state observers designed directly on the Special Euclidean group $\SE$ 
(a Lie group) are gradient-based observers on Lie groups. Using these methods and considering right invariant kinematics along with left invariant cost functions, \cite{Hua,Lageman} utilize position measurements to update the state estimates. A limitation of this approach is provided in \cite{Lageman} as well as a practical design methodology in the case where a non-invariant cost-function is considered. Dynamic attitude and angular velocity estimation for uncontrolled rigid bodies in gravity, using global representation of the equations of motion based on geometric mechanics, is reported 
in~\cite{san06,SCLpaper}. This estimation scheme is used in \cite{jgcd12} for feedback 
attitude tracking control.

In addition to estimating the states, in Chapter \ref{CH02_DSCC2013} the main gravitational parameter of an asteroid 
is also estimated using full state measurements, including pose and velocities of a spacecraft 
in an orbit around the asteroid. This parameter is a very important physical property of an 
asteroid, and is a critical piece of information in order to estimate the mass of the asteroid 
and predict the forces and moments applied to a mass particle in its gravity field. Estimation 
schemes for parameter estimation of asteroid based on measurements from exploring 
spacecraft have been developed in prior literature on this topic. Physical properties of the 
asteroid 433 Eros, describing its shape, spin rate and gravity field, were estimated in 
\cite{MilSch} using the data provided by the NEAR spacecraft in an orbit around Eros. Using 
LIDAR ranging instruments, mass and density of asteroid 25143 Itokawa have been 
estimated \cite{TakSch}. The gravitational acceleration of Itokawa turns out to be 18 times 
stronger than the acceleration as a result of solar radiation pressure at a distance of $1$ km from 
the asteroid's surface \cite{TakSch}. The strength of the gravity field of some small-bodies 
during a series of slow hyperbolic flybys around them were estimated in \cite{AbeSch}. The 
work in \cite{AbeSch} also analyzed how rapidly and precisely the gravitational parameter 
had been estimated for Itokawa, Eros and Didymos, and a new operational procedure called 
$\Delta V$ ranging was proposed.

The framework of geometric mechanics has not been used in the past for design of 
observers for the particular application of spacecraft exploring unknown or little known 
solar system bodies. This framework is beneficial for this application because the 
asteroid-spacecraft pair can execute large relative rotational motions. The use of 
homogeneous coordinates, which are not generalized coordinates and allow global 
representation of the configuration space $\SE$, make it possible to represent the 
motion of bodies that are executing large, non-periodic motions \cite{marat, blbook}. 
For the observer design here, the exponential coordinates in $\SE$ are also used. Since 
the exponential coordinates are not defined for rigid body orientations that correspond 
to $\pi$ radian rotations about a body-fixed axis, the convergence of the observer 
obtained is {\em almost global} over the state space. Prior work \cite{SCLpaper} has obtained attitude determination and filtering schemes from direction measurements with bounded attitude and angular velocity measurement errors, given a dynamics model, in the framework of geometric mechanics. Nevertheless, in Chapter \ref{CH02_DSCC2013}, no measurement model has been used and instead, the full state measurement has been exploited.

Finally, at the end of this chapter, another nonlinear observer that can accurately estimate the configuration 
and velocity states of a rigid body is presented. It is assumed that the rigid body has an onboard sensor suite providing measurements of configuration and velocities as well as forces and torques. Exponential convergence of the estimation
errors is shown and boundedness of the estimation error under bounded unmodeled torques and forces is established.
Since exponential coordinates can describe uniquely almost the entire group of rigid body motions, the resulting observer design is almost globally exponentially convergent.


In Chapter \ref{CH03_DSCC2014}, a finite-time convergent observer design for arbitrary rigid-body motion is derived and presented, using rigid body's pose and velocities measurements. This 
observer has an almost global domain of attraction of 
state estimates to actual states in the state space of rigid body rotations and translations. Since finite-time convergence is known 
to be more robust to noise in the dynamics model and noise in the measured states, this 
observer design has inherent stability properties to such noise. Besides, finite-time 
observers guarantee the time it takes for the system to converge to the actual 
states~\cite{bhatFT,Dorato2006overview,Haddad2008finite}.

Unlike discontinuous sliding-mode observers that also provide finite-time 
convergence~\cite{dfpu09,Yu20051957,dimassi2012continuously}, the observer given here uses continuous feedback. Although this 
observer is not Lipchitz continuous, it is H{\"o}lder continuous like the continuous attitude 
feedback stabilization scheme on $\TSO$ presented in \cite{CDC2013Sanyal}. The observer 
presented is derived explicitly using the exponential coordinate representation of $\SE$, 
which is almost global in its description of the motion, whereas \cite{CDC2013Sanyal} uses 
the coordinate-free representation of the attitude on the group of rigid body rotations in 
three-dimensional Euclidean space, $\SO$. A few related works that exploit exponential 
coordinates to design observers or controllers are~\cite{Lee2014AAS,Bras2013CDC}. 
The continuous observer proposed here is shown to provide finite-time convergence of 
state estimates, through a Lyapunov analysis using exponential coordinates. The proposed 
observer laws are shown to drive the estimation errors to the origin in a finite amount of time. 
Although the observer design is based on a given (known) dynamics model, robustness 
to noise in the dynamics and measurement process are shown through numerical simulations. 
These simulation results for the observer with noisy measurements and additive noise in the 
dynamics, show that the estimate errors remain bounded in the presence of noise.

In Chapter \ref{CH03_DSCC2014}, we give the rigid body 
dynamics model in $\Ta\SE$, the tangent bundle of $\SE$, along with the kinematics expressed 
in exponential coordinates on $\SE$. We also present the nonlinear 
observer design, analyze its convergence properties, and show its finite-time convergence to 
actual states of the rigid body system. Numerical simulation 
results are presented for the noise-free case, when there is no measurement noise and no noise in 
the dynamics model. This chapter also presents numerical simulation results when there is 
additive noise in angular and translational velocities measurements and disturbance inputs in the dynamics 
model.

Since the dynamics model of a mechanical system may not always be accurately known due to external disturbances, or as a result of motions of internal mechanisms, estimation schemes that do not require any knowledge of the dynamcis model are of great importance. Such schemes, instead of a known dynamics model, rely on rich measurements provided by sensors (nowadays at low costs) onboard the rigid body. The second phase of this treatise (including Chapters \ref{CH04_VE}-\ref{CH08_Automatica2}) focuses on such dynamics model-free estimation schemes.

The earliest solution to the attitude determination problem from two vector measurements is the 
so-called ``TRIAD algorithm", which dates from the early 1960s~\cite{TRIAD}. This 
was followed by developments in the problem of attitude determination from a set of three 
or more vector measurements, which was set up as an optimization problem called 
Wahba's problem~\cite{jo:wahba}. This problem has been solved by different methods in 
prior literature, a sample of which can be obtained in 
\cite{jo:solwahba,markley1988attitude,san06}.

Continuous-time attitude observers and filtering schemes on 
$\SO$ and $\SE$ have been reported in, e.g., \cite{Khosravian2015Recursive,mabeda04,silvest08,SCLpaper,markSO3,mahapf08,bonmaro09,
Vas1,Khosravian2015observers,rehbinder2003pose,Bonnabel2013,Bonnabel2015}. 
These estimators do not suffer from kinematic singularities \cite{Sommer,Barfoot} like estimators using coordinate 
descriptions of attitude, and they do not suffer from unwinding as they do not use 
unit quaternions. The maximum-likelihood (minimum energy) filtering method of 
Mortensen~\cite{Mortensen} was recently applied to attitude estimation, resulting in a nonlinear 
attitude estimation scheme that seeks to minimize the stored ``energy" in measurement 
errors~\cite{Hua,Zamani2012second,Zamani2013minimum,zatruma11,aguhes06}. 
This scheme is obtained by applying Hamilton-Jacobi-Bellman (HJB) theory~\cite{kirk} to 
the state space of attitude motion \cite{ZamPhD}. Since the HJB equation can only be 
approximately solved with increasingly unwieldy expressions for higher order 
approximations, the resulting filter is only ``near optimal" up to second order. Unlike filtering 
schemes that are based on approximate or ``near optimal" solutions of the HJB equation and 
do not have provable stability, the estimation scheme obtained here can be solved 
exactly, and is shown to be almost globally asymptotically stable. Moreover, unlike filters 
based on Kalman filtering, the estimator proposed here does not presume any knowledge 
of the statistics of the initial state estimate or the sensor noise. Indeed, for 
vector measurements using optical sensors with limited field-of-view, the probability 
distribution of measurement noise needs to have compact support, unlike standard Gaussian 
noise processes that are commonly used to describe such noisy measurements.

All the estimation schemes proposed in Chapter \ref{CH04_VE} and onwards are model-free, which means that they do not depend on any knowledge of the dynamics of rigid body. In Chapter \ref{CH04_VE}, the attitude determination 
problem from vector measurements is formulated on $\SO$. Wahba's cost function is 
generalized in two ways: by choosing a symmetric matrix of weights instead of scalar 
weight factor for individual vector measurements, and by making the resulting cost function 
an argument of a continuously differentiable increasing scalar-valued function. It is shown 
that this generalization of Wahba's function is a Morse function on $\SO$ under certain 
easily satisfiable conditions on the weight matrix, which can be chosen appropriately 
to satisfy these desirable conditions. This chapter formulates the attitude estimation problem 
for continuous-time measurements of direction vectors and angular velocity on the state 
space of rigid body attitude motion, using the formulation of variational mechanics. A 
Lagrangian is constructed from the measurement residuals (between measured and 
estimated states) for the angular velocity measurements and attitude estimates
obtained from the vector measurements. 
The Lagrange-d'Alembert principle applied to this Lagrangian, with a dissipative term linearly dependent 
on the angular velocity estimate error, leads to the state estimation scheme. This estimation 
scheme, when applied in the absence of measurement errors, is shown to 
provide almost global asymptotic stability of the actual attitude and angular velocity states, 
with a domain of attraction that is almost global over the state space. In fact, this domain 
of attraction is shown to be equivalent to that of the almost global asymptotic stabilization scheme 
for attitude dynamics in \cite{chaturvedi2011rigid}. In the development of the attitude and angular velocity estimation schemes presented here, it is assumed that measurements of direction vectors and angular velocity are 
available in continuous time, or at a suitably high sampling frequency. In such 
a measurement-rich estimation process, one need not use a dynamics model for 
propagation of state estimates between measurements.

In order to obtain attitude state estimation schemes from discrete-time vector and angular velocity 
measurements, we apply the discrete-time Lagrange-d'Alembert principle to an action functional 
of a Lagrangian constructed from the state estimate errors, with a dissipation term linear
in the angular velocity estimate error. It is assumed that these measurements 
are obtained in discrete-time at a sufficiently high but constant sample rate. In this chapter, we consider the state estimation problem for attitude and angular 
velocity of a rigid body, assuming that known inertial directions and angular velocity of the 
body are measured with body-fixed sensors. The number of direction vectors measured 
by the body may vary over time. For most of the theoretical developments in this chapter, it 
is assumed that at least two directions are measured at any given instant; this assumption 
ensures that the attitude can be uniquely determined from the measured directions at each 
instant. The state estimation schemes presented here have the following important 
properties: (1) the attitude is represented globally over the configuration space of rigid 
body attitude motion without using local coordinates or quaternions; (2) the schemes 
developed do not assume any statistics (Gaussian or otherwise) on the measurement 
noise; (3) no knowledge of the attitude dynamics model is assumed; and (4) the continuous 
and discrete-time filtering schemes presented here are obtained by applying the 
Lagrange-d'Alembert principle or its discretization~\cite{marswest} to a Lagrangian 
function that depends on the state estimate errors obtained from vector measurements 
for attitude and angular velocity measurements.

Three discrete-time versions of the filter introduced in \cite{Automatica} are obtained and 
compared in Chapter \ref{CH04_VE}. The three discrete-time filters are as follows: (1) a first-order implicit Lie group 
variational integrator that was presented in \cite{Automatica}; (2) a first-order {\em explicit} 
integrator that is the adjoint of the implicit integrator; and (3) a second-order time-symmetric 
integrator obtained by composing the flows of the first order integrators. A variational integrator 
works by discretizing the (continuous-time) variational mechanics principle that leads to the 
equations of motion, rather than discretizing the equations of motion directly. A good background 
on variational integrators is given in the excellent treatise~\cite{marswest}. As described in the 
book~\cite{haluwa}, symplectic integrators (for conservative systems) are a subset within the  
class of variational integrators. Lie group variational integrators are variational integrators for 
mechanical systems whose configuration spaces are Lie groups, like rigid body systems. 
In addition to maintaining properties arising from the variational principles of mechanics, 
like energy and momenta, \emph{Lie group 
variational integrator} (LGVI) schemes also maintain the geometry of the Lie group 
that is the configuration space of the system~\cite{mclele}.

A comparison of the variational estimator is made with some of the 
state-of-the-art attitude filters, namely the Geometric Approximate Minimum-Energy (GAME), the Multiplicative Extended Kalman Filter (MEKF) and the Constant Gain Observer (CGO), in the absence of bias in sensors readings in Chapter \ref{CH05_ICRA2015}. A new measurement model of the problem which can be used for all the filters is explained first. These three state-of-the-art filters on $\SO$ are presented in detail which are used to evaluate the performance of the LGVI, by comparing their principal angles of attitude estimate errors together. Such comparisons are carried out, and cases in which the variational estimator has advantages over other state-of-the-art filters are presented using numerical simulations. Numerical simulations show that the presented observer is robust and unlike the extended Kalman filter based schemes \cite{Zlotnik,Forbes2014Automatica}, its convergence does not depend on the gains values. Besides, the variational estimator is shown to be the most computationally efficient attitude observer.


Since the Variational Estimator requires gyro measurements and these data are usually corrupted by bias in angular velocities, another generalized version of this estimation scheme is presented in Chapter \ref{CH06_Bias}, considering a constant bias in gyro measurements in addition to 
measurement noise. The measurement model for measurements of 
inertially-known vectors and biased angular velocity measurements using body-fixed sensors 
is detailed first. The problem of variational attitude estimation from these 
measurements in the presence of rate gyro bias is formulated and solved on $\SO$. A Lyapunov stability proof of this estimator is given next, along with 
a proof of the almost global domain of convergence of the estimates in the case of perfect 
measurements. It is also shown that the bias estimate converges to the true bias in this case. 
This continuous estimation scheme is discretized in the form of an LGVI using the discrete 
Lagrange-d'Alembert principle. The LGVI gives a first-order approximation of the continuous-time 
estimator. Numerical simulations are carried out using this LGVI as the discrete-time variational 
attitude estimator with a fixed set of gains.

Chapter \ref{CH07_ICC2015} describes the details of experimental verification of the attitude estimator presented in Chapter \ref{CH04_VE}. This chapter utilizes the smartphone\rq{}s inbuilt accelerometer, magnetometer and gyroscope as an Inertial Measurement Unit (IMU) for attitude determination. The primary motivation for using an open source smartphone is to create a cost-effective, generic platform for spacecraft attitude determination and control (ADCS), while not sacrificing on performance and 
fidelity. The PhoneSat mission of NASA's Ames Research Center demonstrated the application of 
Commercial Off-The-Shelf (COTS) smartphones as the satellite's onboard computer with its sensors being 
used for attitude determination and its camera for Earth observation \cite{Phonesat1}. University of Surrey's Space Centre 
(SSC) and Surrey Satellite Technology Ltd (SSTL) developed STRaND-1, a 3U CubeSat containing a 
smartphone payload \cite{strand1,strand2}. Some advantages of using smartphones, on-board are:
 \begin{enumerate}
  \item compact form factor with powerful CPU, GPU etc.,
 \item integrated sensors and data communication options,
 \item long lasting batteries: reduces total mass budget, 
 \item cheap price and open source software development kit.
 \end{enumerate}
 
The attitude and angular velocity estimation 
scheme is based on inertial directions and angular velocity of the spacecraft measured by 
sensors in the body-fixed frame of the smartphone. The 
standalone mechatronics architecture performs the task of state sensing through embedded 
MEMS sensors, filtering, state estimation, to determine 
the cellphone's attitude, while maintaining active uplink/downlink with a remote ground control station.

An important generalization of the Variational Estimation scheme is to derive an estimator for the most general motion of rigid body in 3 dimensional space \cite{Haichao2015}, which is the special Euclidean group, $\SE$. Autonomous state estimation of a rigid body based on inertial vector measurement and visual feedback from stationary landmarks, in 
the absence of a dynamics model for the rigid body, is analyzed in Chapter \ref{CH08_Automatica2}. The estimation scheme proposed here can also be applied to {\em relative state} estimation with respect to moving 
objects \cite{Gaurav_ASR}. This estimation scheme can enhance the autonomy and reliability of 
unmanned vehicles in uncertain GPS-denied environments. Salient features of this estimation 
scheme are: (1) use of onboard optical and inertial sensors, with or without rate gyros, for 
autonomous navigation; (2) robustness to uncertainties and lack of knowledge of dynamics; 
(3) low computational complexity for easy implementation with onboard processors; (4) proven 
stability with large domain of attraction for state estimation errors;  
and (5) versatile enough to estimate motion with respect to stationary as well as moving objects. 
Robust state estimation of rigid bodies in the absence of complete knowledge of their dynamics, 
is required for their safe, reliable, and autonomous operations in poorly known conditions. In 
practice, the dynamics of a vehicle may not be perfectly known, especially when the vehicle is 
under the action of poorly known forces and moments. The scheme proposed here has a single, 
stable algorithm for the coupled translational and rotational motion of rigid bodies using 
onboard optical (which may include infra-red) and inertial sensors. This avoids the need for 
measurements from external sources, like GPS, which may not be available in indoor, underwater 
or cluttered environments \cite{leishman2014relative,miller2014tracking,amelin2014algorithm}.

Chapter \ref{CH08_Automatica2} applies the variational estimation framework to coupled rotational (attitude) and translational motion, as 
exhibited by maneuvering vehicles like UAVs. In such applications, designing separate state 
estimators for the translational and rotational motions may not be effective and may lead to poor 
navigation. For navigation and tracking the motion of such vehicles, the approach proposed here 
for robust and stable estimation of the coupled translational and rotational motion will be more 
effective than de-coupled estimation of translational and rotational motion states. Moreover, like 
other vision-inertial navigation schemes~\cite{shen2013vision,shen2013rotor}, the estimation 
scheme proposed here does not rely on GPS. However, unlike many other vision-inertial 
estimation schemes, the estimation scheme proposed here can
be implemented without any direct velocity measurements.
Since rate gyros are usually corrupted by high noise content and bias \cite{GoodPhD,Good2015ACC,Good2013ECC,Good2014ASME,Good2014ACC,Good2015IJCAS,Lotfi}, 
such a velocity measurement-free scheme can result in fault tolerance in the case of faults with rate 
gyros. Additionally, this estimation scheme can be extended to relative pose estimation between 
vehicles from optical measurements, without direct communications or measurements of relative velocities.

In this chapter, the problem of motion 
estimation of a rigid body using onboard optical and inertial sensors is introduced first. The 
measurement model is introduced and rigid body states are related to these measurements. Artificial energy terms are introduced next, representing the measurement residuals 
corresponding to the rigid body state estimates. The Lagrange-d'Alembert principle is applied to the Lagrangian 
constructed from these energy terms with a Rayleigh dissipation term linear in the velocity 
measurement residual, to give the continuous time state estimator. Particular versions of 
this estimation scheme are provided for the cases when direct velocity measurements are not 
available and when only angular velocity is directly measured. The 
stability of the resulting variational estimator is proved next. It is shown that, in the absence of measurement 
noise, state estimates converge to actual states asymptotically and the domain of 
attraction is an open dense subset of the state space. The variational pose
estimator is discretized as a Lie group variational integrator, by applying the discrete 
Lagrange-d'Alembert principle to discretizations of the Lagrangian and the dissipation term. This 
estimator is simulated numerically, for two cases: the case 
where at least three beacons are measured at each time instant; and the under-determined case, 
where occasionally less than three beacons are observed. For these simulations, true states 
of an aerial vehicle are generated using a given dynamics model. Optical/inertial measurements 
are generated, assuming bounded noise in sensor readings. Using these 
measurements, state estimates are shown to converge to a neighborhood of actual states, for 
both cases simulated. Finally, the contributions and possible future 
extensions of this chapter are listed.

\newpage

\section{MODEL-BASED OBSERVER DESIGN WITH ASYMPTOTIC CONVERGENCE} 

\label{CH02_DSCC2013} 

\hspace{\parindent}
\textit{This chapter is adapted from papers published in Proceedings of the 2013 ASME Dynamic Systems and Control (DSC) Conference \cite{Izadi2013DSCC} and the $52^{nd}$ IEEE Conference on Decision and Control \cite{Bras2013CDC}. The author gratefully acknowledges Dr. Amit Sanyal, Dr. Daero Lee, Dr. Eric Butcher, Dr. Daniel Scheeres, Jan Bohn, S{\'e}rgio Br{\'a}s, Dr. Paulo Oliveira and Dr. Carlos Silvestre for their participation.}
\\
\\
{\bf{Abstract}}~
We consider an observer design for a spacecraft modeled as a rigid body in the 
proximity of an asteroid. The nonlinear observer is constructed on the nonlinear state 
space of motion of a rigid body, which is the tangent bundle of the Lie group of rigid body 
positions and orientations in three-dimensional Euclidean space. The framework of geometric mechanics
is used for the observer design. States of motion of the spacecraft are estimated based on 
state measurements. In addition, the observer designed can also estimate the gravity of the 
asteroid, assuming the asteroid to have a spherically symmetric mass distribution. Almost 
global convergence of state estimates and gravity parameter estimate to their corresponding 
true values is demonstrated analytically, and verified numerically.

\subsection{Rigid Body Dynamics}
\hspace{\parindent}
Consider a body fixed reference frame in the center of mass of a rigid spacecraft denoted as $\Bref$ and an inertial fixed frame denoted as $\Iref$.
Let the rotation matrix from $\Bref$ to the inertial fixed frame $\Iref$ be given by $R$ and the coordinates of the origin of $\Bref$ with respect to $\Iref$ be denoted as $b$.
The set of rotation matrices which contains $R$ is denoted by
$\SO=\{R\in\mathbb{R}^{3\times 3}:R\T R=I,\det(R)=1\}$.
The rigid body kinematics are given by
\begin{align}
\dot R&= R\Omega^\times\\
\dot b&= R v,
\end{align}
the linear and angular velocities expressed in the body fixed frame $\{B\}$ are denoted by $v$ and $\Omega$, respectively, and the skew-symmetric operator $\skew{.}:\mathbb{R}^3\to\mathfrak{so}(3)$ satisfies
\begin{align}
\Omega^\times=
\begin{bmatrix}
0 & -\Omega_z & \Omega_y\\
\Omega_z & 0 & -\Omega_x\\
-\Omega_y & \Omega_x & 0
\end{bmatrix}.
\end{align}
Let $g$ be the spacecraft configuration such that
\begin{equation}
g=\begin{bmatrix}
R & b\\
0 & 1
\end{bmatrix}\in\SE,
\label{gDef}
\end{equation}
where the \emph{Special Euclidean Group} $\SE$ is the Lie group of rotations and translations whose matrix representation is given by the so-called homogeneous coordinates
\begin{align}
\SE=\left\{g\in \mathbb{R}^{4\times 4}, g=
\begin{bmatrix}
R & b\\
0 & 1
\end{bmatrix}:
R\in\SO,b\in\mathbb{R}^3\right\}.
\end{align}

The dynamics equations of the spacecraft in the compact form are
\begin{align}
\dot{g}&=g\xi^\vee \label{eq:Dyn_g}\\
\mathbb{I}\dot\xi=\adast{\xi}\,\mathbb{I}\xi+\varphi_G(g)&=\adast{\xi}\,\mathbb{I}\xi+\mu\psi_G(g) \label{eq:Dyn_xi},
\end{align}
where $\xi=[\Omega\T\ v\T]\T$, $\mathbb{I}=\begin{bmatrix}
J & 0\\
 0 & mI
\end{bmatrix}$, $m$ and $J$ denote mass and inertia matrix of the spacecraft respectively, 
$I$ is the $3\times 3$ identity matrix, $\adast{\xi}=(\ad{\xi})\T$
and $\ad{\xi}$ stands for the linear adjoint operator of the Lie algebra $\mathfrak{se(3)}$ 
associated with the Lie group $\SE$ such that
\begin{equation}
\ad{\xi}=
\begin{bmatrix}
\Omega^\times & 0\\
v^\times & \Omega^\times
\end{bmatrix}.
\end{equation}

Besides, $\varphi_G(g)=\left\{\begin{matrix}
M_G\\F_G
\end{matrix}\right\}$ is the vector of external moments and forces, $\mu$ is the unknown scalar gravity parameter and $\psi_G(g)$ is a $6\times 1$ vector which satisfies the equation $\varphi_G(g)=\mu\psi_G(g)$, where $M_G, F_G\in\mathbb{R}^3$ denote the gravity gradient moment and gravitational force applied on the spacecraft respectively, which are given by \cite{schaub2009analytical}:
\begin{align}
M_G&=\mu\Big\{\frac{3}{\|b\|^5}(p\times J p)\Big\}\\
F_G&=\mu\Big\{-\frac{m}{\|b\|^3}p-\frac{3}{\|b\|^5}\mathscr{J}p+\frac{15}{2}\frac{p\T J p}{\|b\|^7}p\Big\},
\end{align}
where $p=R\T b$ and $\mathscr{J}=\frac{1}{2}trace (J)I_{3\times 3}+J$.

\subsection{Observer Design for a Spherical Asteroid}
Consider $(\hat{g},\hat{\xi})$ to be the estimated values of the states $(g,\xi)$ of a rigid body's motion on $\SE\times\mathbb{R}^6$. Define
\begin{align}
h=\hat{g}^{-1}g\, \mbox{ and }\, \eta^\vee=\logm_{\SE}(h),
\end{align}
where $\logm_{\SE}(.):\SE\rightarrow \mathfrak{se}(3)$ denotes the logarithmic map on 
$\SE$ and $\expm_{\SE}$ is its inverse. Therefore, we obtain:
\begin{align}
\dot{h}=h\tilde{\xi}^\vee \, \mbox{ where }\, \tilde{\xi}=\xi-\Ad{h^{-1}}\hat{\xi},
\end{align}
and
\begin{equation}
\Ad{g}\zeta^\vee=
\left(\begin{bmatrix}
R & 0\\
b^\times R & R
\end{bmatrix}
\zeta\right)^\vee,\ \zeta\in\mathbb{R}^6,\ g\in\SE.
\end{equation}
If we define $\breve{\xi}=\Ad{h^{-1}}\hat{\xi}$, then $\tilde{\xi}=\xi-\breve{\xi}$. We express the exponential coordinate vector $\eta$ for the pose estimate error as 
\begin{align}
\eta= \begin{bmatrix} \Theta \\ \beta \end{bmatrix}\in\mathbb{R}^6\simeq \mathfrak{se}(3),
\label{C2ExpCor}
\end{align}
where $\Theta\in\mathbb{R}^3$ is the exponential coordinate vector (principal rotation vector) 
for the attitude estimation error and $\beta\in\mathbb{R}^3$ is the exponential coordinate vector 
for the position estimate error. The time derivative of the exponential coordinates of the 
configuration error is given by \cite{Bullo_ECC95}
\begin{align}\label{eq:ddt_eta}
\dot{\eta}=G(\eta)\tilde\xi,
\end{align}
where
\begin{align}
&~~~~~~~~~~~~~~~~
G(\eta)= \begin{bmatrix} A(\Theta) & 0\\ T(\Theta,\beta) & A(\Theta) \end{bmatrix},
\label{Geta}\\
&~~~~~~~~~~~~~~~~
A(\Theta)= I+\frac12 \Theta^\times + \left(\frac1{\theta^2} - \frac{1+\cos\theta}{2\theta\sin\theta}
\right) \big(\Theta^\times\big)^2, \label{C2ATheta}\\
T&(\Theta,\beta) = \frac12\big(S(\Theta)\beta\big)^\times A(\Theta)+\left(\frac1{\theta^2} - \frac{1+\cos\theta}{2\theta\sin\theta}\right)
\big[\Theta\beta\T + (\Theta\T\beta) A(\Theta)\big]\\ 
&-\frac{(1+\cos\theta)(\theta-\sin\theta)}{2\theta\sin^2\theta}S(\Theta)\beta\Theta\T +\left( \frac{(1+\cos\theta)(\theta+\sin\theta)}{2\theta^3\sin^2\theta}- \frac2{\theta^4}\right) 
\Theta\T\beta\Theta\Theta\T,\nn\\ 
&~~~~~~~~~~~~~~~~
S(\Theta)= I+  \frac{1-\cos\theta}{\theta^2}\Theta^\times+ \frac{\theta
-\sin\theta}{\theta^3}\big(\Theta^\times\big)^2.
\label{C2STheta}
\end{align} 
The time derivative of the exponential coordinate $\Theta$ for the rotational motion is 
obtained from Rodrigues' formula
\begin{align}
R(\Theta) &= I+ \frac{\sin\theta}{\theta}\Theta^\times+ \frac{1-\cos\theta}{\theta^2}
\big(\Theta^\times\big)^2 \mbox{with }\theta= \|\Theta\|,  \label{rodriform} 
\end{align} 
which is a well-known formula for the rotation matrix in terms of the exponential coordinates 
on $\SO$, the Lie group of special orthogonal 
matrices. In the context of equations \eqref{eq:ddt_eta}-\eqref{C2STheta}, the matrix $R(\Theta)=
\tilde R$, i.e., the attitude estimate error on $\SO$. We consider next a result that is 
important in obtaining the observer described later in this section.
\begin{lemma}
The matrix $G(\eta)$, which occurs in the kinematics equations \eqref{eq:ddt_eta}-\eqref{C2STheta} 
for the exponential coordinates on $\SE$, satisfies the relation
\begin{equation}
G(\eta) \eta= \eta. \label{eq:G_eta_eta}
\end{equation}
\end{lemma}
\begin{proof}
Beginning with the expression for $G(\eta)$ given by \eqref{Geta}, we evaluate
\[ G(\eta)\eta= \begin{bmatrix} A(\Theta)\Theta \\ T(\Theta,\beta)\Theta
+A(\Theta)\beta \end{bmatrix}. \]
From the expression for $A(\Theta)$, it is clear that
\[ A(\Theta)\Theta= \Theta. \]
On evaluation of the other component, after some algebra, we obtain
\[ T(\Theta,\beta)\Theta= \beta- A(\Theta)\beta. \]
Therefore, we obtain 
\[ T(\Theta,\beta)\Theta + A(\Theta)\beta= \beta, \]
which gives the desired result. 
\end{proof}

Further, define an auxiliary variable
\begin{align}
\ell=\tilde{\xi}+k_1\eta.
\end{align}

Let the candidate Lyapunov function be
\begin{align}
V=\frac{1}{2}k_2\eta\T\eta+\frac{1}{2}\ell\T\mathbb{I}\ell+\frac{1}{2}k_3\tilde{\mu}^2,
\end{align}
where $\tilde{\mu}=\mu-\hat{\mu}$ is the scalar estimation errors of the gravity parameter. Using this Lyapunov function we can show that the following observer design is stable. 

\begin{theorem}
The states observer given in the form
\begin{align}
\dot{\hat{g}}&=\hat{g}\hat{\xi}^\vee \label{eq:hat_g}\\
\mathbb{I}\dot{\breve{\xi}}&=\adast{\breve{\xi}}\mathbb{I}\xi+\hat{\mu}\psi_G(g)+s_\xi,
\label{eq:breve_xi}
\end{align}
where
\begin{align}
s_\xi=-\adast{k_1\eta}\mathbb{I}\xi+k_2G\T(\eta)\eta+k_1\mathbb{I}G(\eta)\tilde{\xi}+k_4\ell, \label{eq:s_xi}
\end{align}
along with the following equations for estimating the unknown gravity parameter $\mu$ ensures that the estimate errors converge to the origin:
\begin{align}
\dot{\hat{\mu}}=\dot{\mu}-\dot{\tilde{\mu}}&=\frac{1}{k_3}\ell\T\psi_G(g). \label{eq:hat_mu}
\end{align}
\end{theorem}
\begin{proof}
Using the equations \eqref{eq:ddt_eta} and \eqref{eq:G_eta_eta} in \cite{Bullo_ECC95}
and calculating the time derivative of estimation error in velocities and the gravity parameter as follows
\begin{align}
\mathbb{I}\dot{\tilde{\xi}}&=\adast{\ell}\mathbb{I}\xi+\tilde{\mu}\psi_G(g)-k_2G\T(\eta)\eta-k_1\mathbb{I}G(\eta)\tilde{\xi}-k_4\ell\\
\dot{\tilde{\mu}}&=-\frac{1}{k_3}\ell\T\psi_G(g),
\end{align}
and also the time derivative of the auxiliary parameter
\begin{align}
\mathbb{I}\dot{\ell}&=\mathbb{I}\dot{\tilde{\xi}}+k_1\mathbb{I}\dot{\eta}\nn\\
&=\adast{\ell}\mathbb{I}\xi+\tilde{\mu}\psi_G(g)-k_2G\T(\eta)\eta-k_1\mathbb{I}G(\eta)\tilde{\xi}-k_4\ell+k_1\mathbb{I}G(\eta)\tilde{\xi}\nn\\
&=\adast{\ell}\mathbb{I}\xi+\tilde{\mu}\psi_G(g)-k_2G\T(\eta)\eta-k_4\ell,
\end{align}
the first and second order time derivative of the proposed Lyapunov function can be written as
\begin{align}
\dot{V}=&k_2\eta\T \dot{\eta}+\ell\T\mathbb{I}\dot{\ell}+k_3\tilde{\mu}\dot{\tilde{\mu}}\nn\\
=&k_2\eta\T G(\eta)\tilde{\xi}+\ell\T\adast{\ell}\mathbb{I}\xi+\tilde{\mu}\ell\T\psi_G(g)-k_2\ell\T G\T(\eta)\eta-k_4\ell\T \ell-\tilde{\mu} \ell\T\psi_G(g)\nn\\
=&k_2\eta\T G(\eta)\big(\tilde{\xi}-\ell\big)-k_4\ell\T \ell=-k_1 k_2\eta\T G(\eta)\eta-k_4\ell\T \ell\nn\\
=&-k_1 k_2\|\eta\|^2-k_4\|\ell\|^2.
\end{align}

Thus, $\dot{V}$ is negative semi-definite. Besides,
\begin{align}
\ddot{V}=&-k_1k_2\eta\T\dot{\eta}-k_4\ell\T\dot{\ell}\\
=&-k_1k_2\eta\T G(\eta)\tilde{\xi}-k_4\ell\T\mathbb{I}^{-1}\adast{\ell}\mathbb{I}\xi+k_4\tilde{\mu}\ell\T\mathbb{I}^{-1}\psi_G(g)+k_2 k_4\ell\T\mathbb{I}^{-1}G\T(\eta)\eta\nn\\
&+k_4^2\ell\T\mathbb{I}^{-1}\ell,\nn
\end{align}
which means that $\ddot{V}$ is finite for any bounded pose and velocity vector. Using \textit{Barbalat's Lemma} one can conclude that $\dot{V}\rightarrow 0$ which gives $\|\eta\|\rightarrow 0$ and $\|\ell\|\rightarrow 0$, therefore $\|\tilde{\xi}\|\rightarrow 0$ in turn. Moreover, for initially bounded state estimate errors, $\|\tilde{\xi}\|\rightarrow 0$ leads to $\frac{d}{dt}\|\tilde{\xi}\|\rightarrow 0$ which implies that $|\tilde{\mu}|\rightarrow 0$.
\end{proof}

Note that this observer, which uses the exponential coordinate representation of the pose in 
$\SE$, is not defined when the exponential coordinate vector itself is not defined. This happens 
whenever the attitude corresponds to a principal rotation angle given by an odd multiple of 
$\pi$ radians.


\subsection{Numerical Simulations}
In order to verify the performance of the observer, we used a set of realistic data for all the 
states of the system. One can consider a spherical asteroid with the gravitational constant 
$\mu=1.729\times 10^{10}$ m$^3$/s$^{2}$ and integrate the dynamics to mimic the true 
states of the spacecraft in an orbit around this spherical asteroid. Mass and inertia matrix of the spacecraft is considered as $m=21$ kg and $J=diag([2.56\ 3.01\ 2.98]\T)$ kg m$^2$. The spacecraft is rotating in an elliptical orbit with semi-major axis $a=330$ km and the radius at periapsis equal to $r_p=310$ km. The initial configuration is given by
\begin{align}
R_0=\expm_{\SO}\skew{[0.4\ 0.2\ 0.1]\T},\
b_0=r_p\times[\frac{1}{3}\ \frac{-2}{3}\ \frac{2}{3}]\T\,
\end{align}
and the initial angular and linear velocities were set to
\begin{align}
\Omega_0&=10^{-3}\times [7\ -4\ 1]\T\ \text{rad/s},\\
\nu_0&=R_0\T\bigg(v_p\times\frac{b_0}{r_p}\times\skew{[0\  \frac{1}{\sqrt{2}}\ \frac{1}{\sqrt{2}}]\T}\bigg)=[241.4\ 16.7\ -24.5]\T\ \text{m/s},
\end{align}
where
\begin{align}
v_p=\sqrt{-\frac{\mu}{a}+2\frac{\mu}{r_p}}.
\end{align}

Using these numerical values, the simulated orbit is shown in the Fig.~\ref{C2fig:Dynamics}. As can be easily seen, the spacecraft's path around the asteroid is an elliptical periodic orbit with the radius at periapsis $r_p=310$ km and the semi-major axis $a=330$ km. Note that the orbital period of the spacecraft is
\begin{align}
T=2\pi\sqrt{\frac{a^3}{\mu}}=9058.4\text{ s}
\end{align}

\begin{figure}[t]
\centering
\includegraphics[width=0.9\columnwidth]{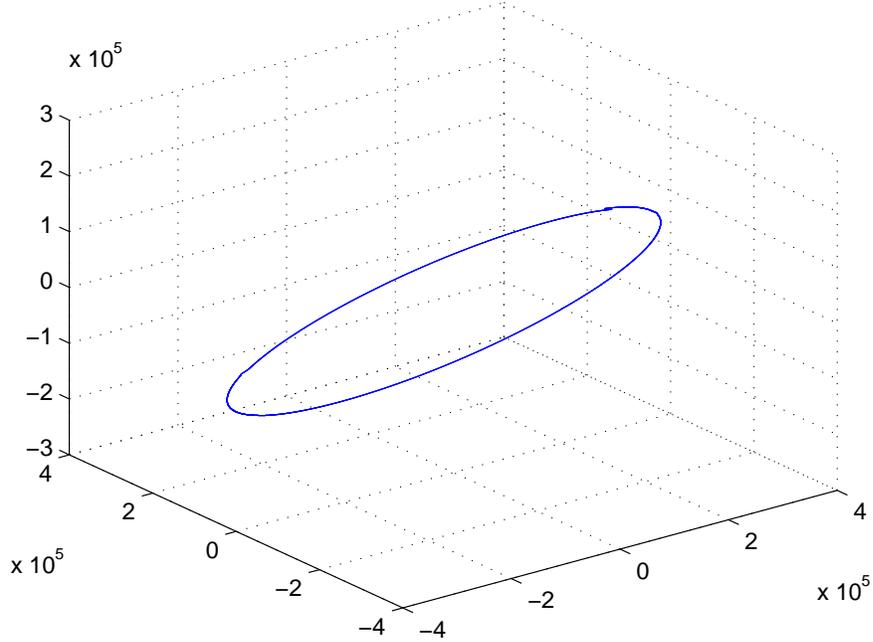}
\caption{The Generated Path for the Actual Orbit of the Spacecraft Rotating Around the Spherical Asteroid.}
\label{C2fig:Dynamics}
\end{figure}

Integrating the logarithmic map of equation \ref{eq:Dyn_g} along with \ref{eq:Dyn_xi}, the exponential coordinates of the spacecraft in the vicinity of the asteroid could be plotted numerically as in Fig.~\ref{C2fig:eta}. Note that the logarithmic map could be used to get the exponential coordinates of both absolute configuration and relative configuration and we have used the same notation for them and their components. In other words, $\eta$ denotes the logarithmic map of both pose ($g$) and relative pose ($h=\hat{g}^{-1}g$) and $\Theta$ and $\beta$ are it's angular and linear components.

\begin{figure}[thb]
\centering
\includegraphics[width=0.8\columnwidth]{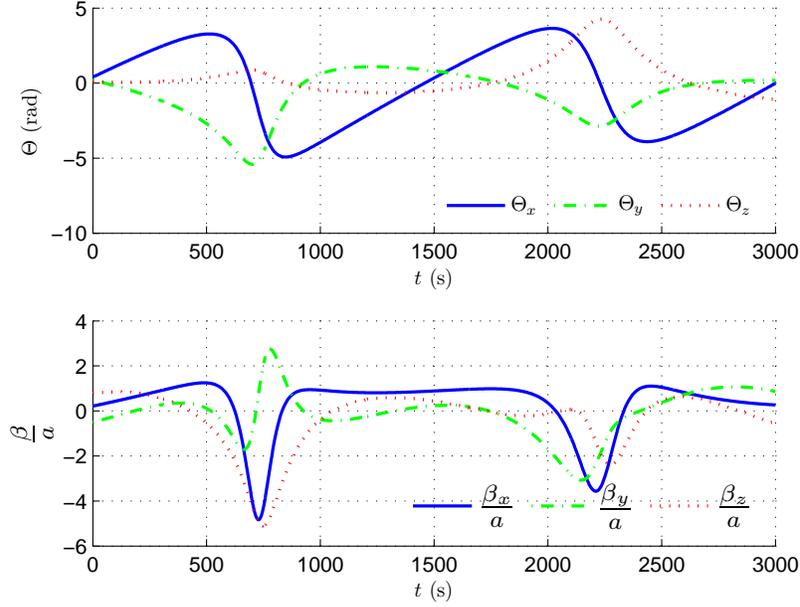}
\caption{Exponential Coordinates of the Spacecraft.}
\label{C2fig:eta}
\end{figure}

The angular and linear components of the spacecraft's velocity are also depicted in Fig.~\ref{C2fig:xi}. Both Fig.~\ref{C2fig:eta} and \ref{C2fig:xi} show somehow periodic motion in their components which was expected from the motion in the elliptical orbit of Fig.~\ref{C2fig:Dynamics}.
\begin{figure}[t]
\centering
\includegraphics[width=0.8\columnwidth]{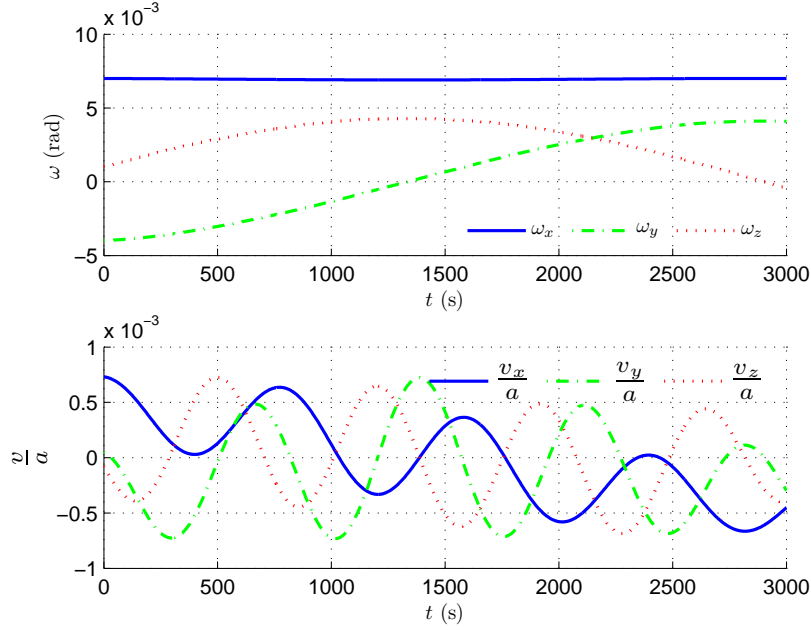}
\caption{Spacecraft's Angular and Linear Velocities.}
\label{C2fig:xi}
\end{figure}
After mimicking the actual dynamics, another code was used to numerically integrate the observer ODEs which are equations \eqref{eq:hat_g}-\eqref{eq:hat_mu}.

Since the numerical values for the translational quantities (displacements and velocities) depend on the unit by which they are described and specially in the case a relatively small unit like meters has been used the quantities will be of a much higher order compared with the angular quantities, we should normalize translational quantities to resolve numerical issues while dealing with the compact forms. The semi-major axis $a$ was used to make all linear quantities dimensionless. Note that the angular velocities are in radians and therefore dimensionless. In view of the fact that the dimension of the gravitational parameter is $L^3 T^{-2}$, it was divided by $a^3$.

In order to better agreement between angular and linear components, and as a result get better convergence behavior, we also could use a block diagonal form for some gain factors. This helps different components in the compact form converge at almost the same rate. The tuning parameters are set to be $k_1=\begin{bmatrix}
1.12\times I_{3\times 3} & 0_{3\times 3}\\
 0_{3\times 3} & I_{3\times 3}
\end{bmatrix}$, $k_2=1$, $k_3=0.2$ and $k_4=\begin{bmatrix}
1.2\times I_{3\times 3} & 0_{3\times 3}\\
 0_{3\times 3} & I_{3\times 3}
\end{bmatrix}$.

In this step, the value of the gravity parameter is completely unknown like all other states. The initial values for the estimated quantities are set as follows. The initial values of estimated attitude and position vector of the spacecraft are $\hat{R}_0=I_3\times 3$ and $\hat{b}_0=[103\ -206\ 206]\T\ \text{km}$. $\hat{\Omega}_0=10^{-3}\times[5\ -7\ 3]\T \text{rad/s}$ and $\hat{\nu}_0=[59\ -19\ -27]\T \text{m/s}$ were set as the initial estimates for the angular and linear velocities of the spacecraft.
The initial estimated value of gravitational parameter is set to $\hat{\mu}_0=1.6172\times 10^{14}$ m$^3$/s$^2$ which is almost 10,000 times bigger than the actual value and large enough to test 
the convergence behavior of the observer.


The estimation errors in the exponential coordinates obtained in this numerical simulation, 
are shown in Fig.~\ref{C2fig:etatilde}. These estimation errors are seen to decrease at a 
satisfactory rate. As can be seen, these errors become negligible after about $50 \text{(s)}$.

\begin{figure}[thb]
\centering
\includegraphics[width=0.8\columnwidth]{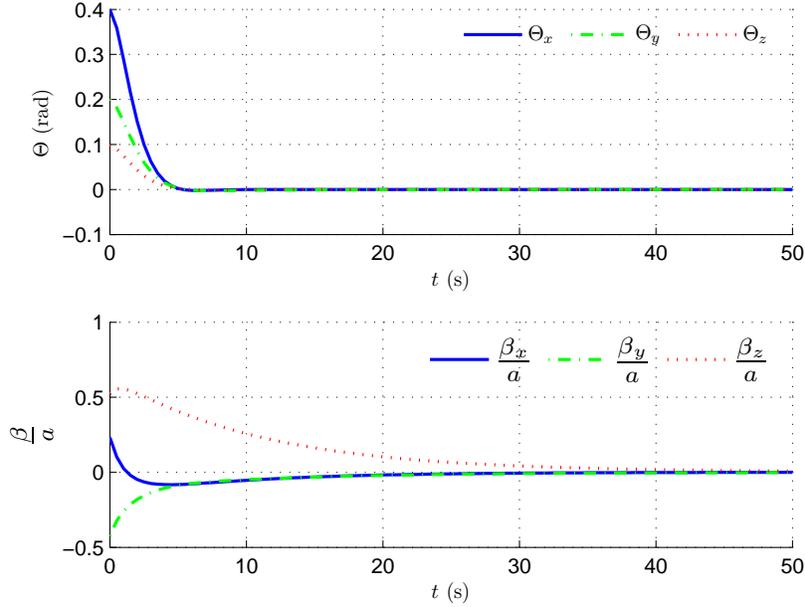}
\caption{Estimation Errors In Exponential Coordinates of the Spacecraft.}
\label{C2fig:etatilde}
\end{figure}

The errors in estimated angular and linear velocities are depicted in Fig.~\ref{C2fig:xitilde}. 
Convergence can be seen in all components, even though the first few seconds show 
increasing errors for some components. Here, the observer is seen to have desirable 
behavior for velocity estimation of the spacecraft.

\begin{figure}[t]
\centering
\includegraphics[width=0.8\columnwidth]{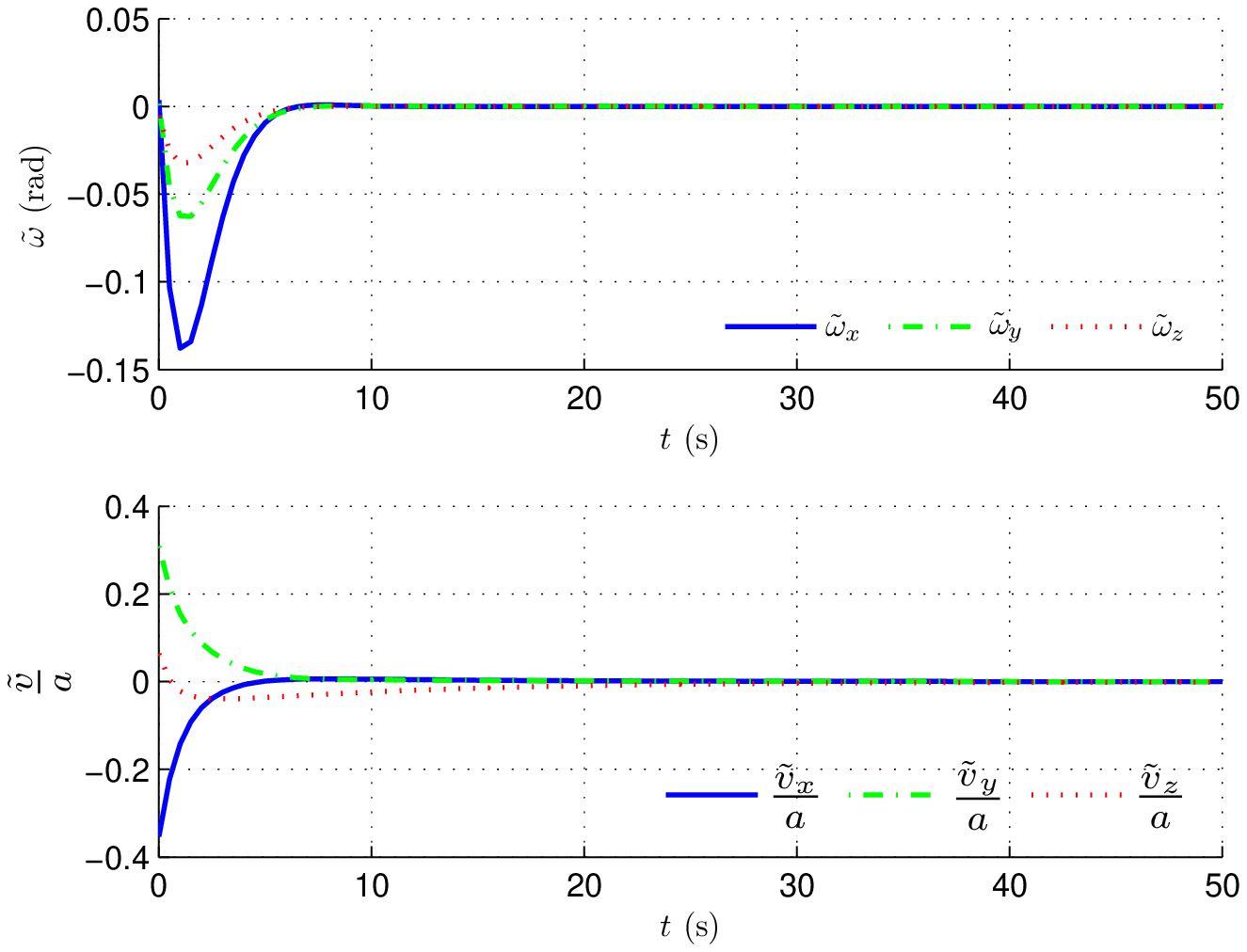}
\caption{Estimation Errors of the Spacecraft's Velocities.}
\label{C2fig:xitilde}
\end{figure}

The estimate error in the gravitational parameter is plotted in Fig.~\ref{C2fig:mutilde}, beginning  
with a large initial estimate error. Asymptotic convergence to the true value can be observed in 
this figure. Note that in Fig.~\ref{C2fig:eta}-\ref{C2fig:mutilde}, the length unit has been 
normalized to 1 unit$=$ 310 km, which is the value of the semi-major axis of the orbit of 
the spacecraft around the asteroid.

\begin{figure}[t]
\centering
\includegraphics[width=0.8\columnwidth]{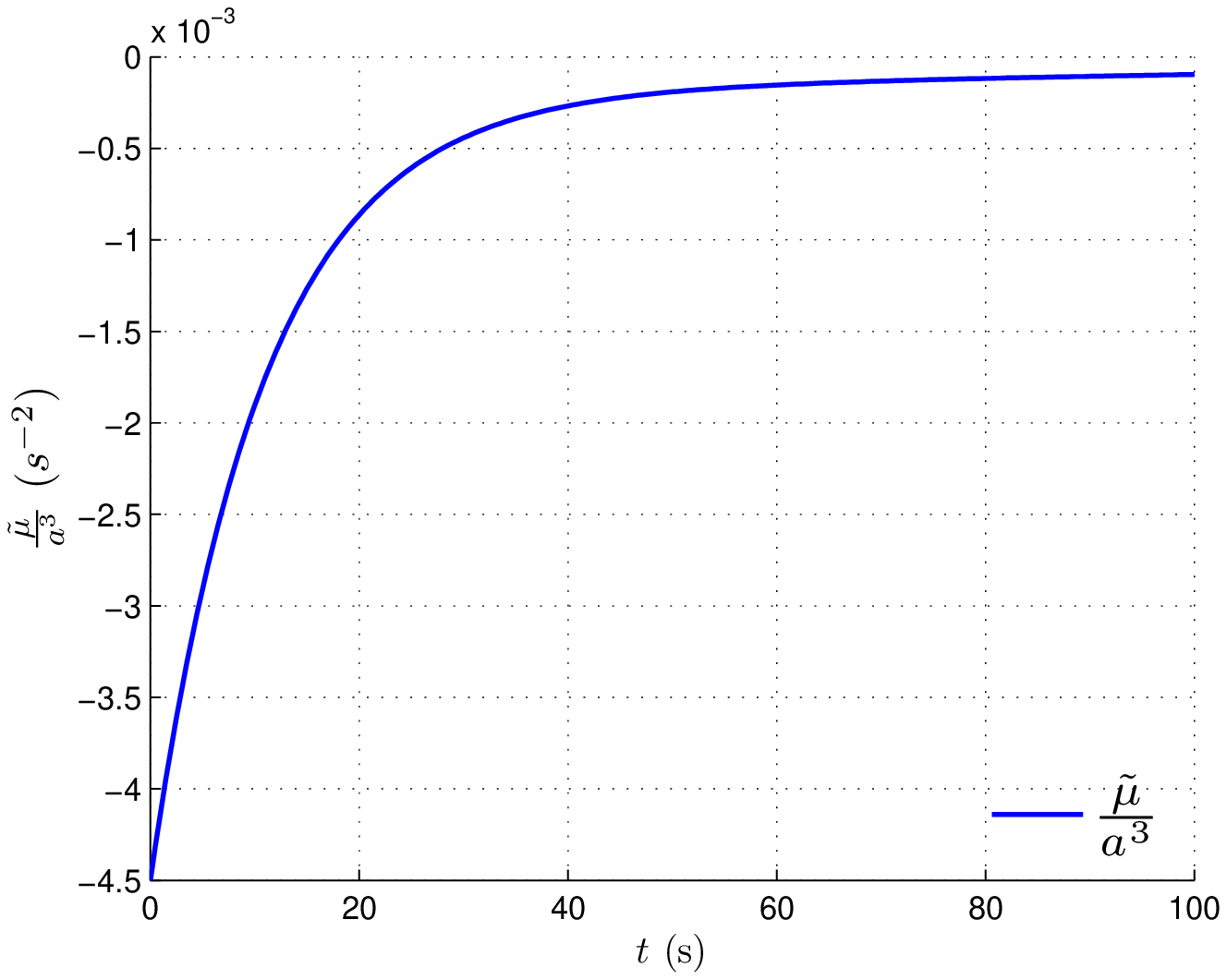}
\caption{Estimation Errors of the Gravitational Parameter.}
\label{C2fig:mutilde}
\end{figure}

\subsection{Nonlinear Observer Design when Force Measurements are available}
An observer design for pose and velocity estimation for three-dimensional rigid body motion, in the framework of geometric mechanics is presented here. Based on a Lyapunov analysis, a nonlinear observer on the Special Euclidean Group $\SE$ is derived. This observer is based on the exponential coordinates which are 
used to represent the group of rigid body motions.

\begin{assumption}\label{assump:full_state_meas}
The sensor suite available provides measurements about the configuration, velocity, forces and torques applied to the vehicle.
\end{assumption}
Note that, even with full state measurements, the existence of an observer is valuable for any navigation and control system as, like the EKF, it can mitigate the effects of sensor uncertainties such as noise and bias. The configuration observer takes the form 
\begin{align}
\dot{\hat{g}}&=\hat{g}\hat{\xi}^\vee,\label{eq:pose-est}\\
\mathbb{I}\dot{\breve\xi}&=\adast{-K(k_1\tilde\eta-\breve\xi)}\mathbb{I}\xi+\varphi+k_1\mathbb{I} G(\tilde\eta)\tilde\xi+G\T(K\tilde\eta)\tilde\eta+{k_3}u,\label{eq:dot_breve_xi}
\end{align}
where $K=\left[\begin{smallmatrix}I & 0\\0 & k_2I\end{smallmatrix}\right]$, $k_1,k_2,k_3>0$, $u=k_1\tilde\eta+\tilde\xi$ and
\begin{align}\label{eq:velest-err}
\breve{\xi}=\Ad{\tilde{g}^{-1}}\hat\xi, \qquad
\tilde{\xi}=\xi-\breve{\xi}. 
\end{align}

\begin{proof}
Consider the following Lyapunov function candidate
\begin{align}\label{eq:lyap_func}
V=\frac{1}{2}\tilde\eta\T K\tilde\eta+\frac{1}{2}(k_1\tilde\eta+\tilde\xi)\T K\mathbb{I}(k_1\tilde\eta+\tilde\xi),
\end{align}
which motivates the development of the velocity observer.
Letting $u=k_1\tilde\eta+\tilde\xi$ and taking the time derivative produces
\begin{align}
\dot{V}&=-k_1\tilde\eta\T K\tilde\eta+
u\T K(G\T(K\tilde\eta)\tilde\eta+
k_1\mathbb{I} G(\tilde\eta)\tilde\xi+\adast{\xi}\mathbb{I}\xi+\varphi-\mathbb{I}\dot{\breve\xi}),
\end{align}
where it is exploited the equality $K^{-1}G\T(\tilde\eta)K=G\T(K\tilde{\eta})$. 
Let
\begin{equation}\label{eq:dot_breve_xi}
\mathbb{I}\dot{\breve\xi}=\adast{K(k_1\tilde\eta-\breve\xi)}\mathbb{I}\xi+\varphi+k_1\mathbb{I} G(\tilde\eta)\tilde\xi+G\T(K\tilde\eta)\tilde\eta+{k_3}u.
\end{equation}

Then, resorting to some algebraic manipulations,
the time derivative of \eqref{eq:lyap_func} takes the negative definite form
\begin{equation}\label{eq:ddt_V}
\dot{V}=-k_1\tilde\eta\T K\tilde\eta-k_3(k_1\tilde\eta+\tilde\xi)\T K(k_1\tilde\eta+\tilde\xi).
\end{equation}
Thus, the point $(\tilde\eta,\tilde\xi)=(\mathbf{0},\mathbf{0})$ is asymptotically stable in sense of Lyapunov \cite{khal}. Topological limitations precludes global asymptotic stability of the origin \cite{bhat}. 
In fact, if $\theta=\pi$, the exponential coordinates of the configuration error $\tilde\eta$ cannot be computed without ambiguity. Sufficient conditions ensuring that for all $t>t_0$, $\theta(t)<\pi$ is provided in \cite{Bras2013CDC}.
\end{proof}

\subsection{Conclusion}

A nonlinear observer for rigid body motion in the presence of an unknown central gravity 
field due to a spherical asteroid was presented. In addition to estimating the states of an 
exploring spacecraft, modeled as a rigid body, in the proximity of a spherical asteroid, this 
observer also estimates the gravity parameter of this asteroid. Estimates obtained from this 
observer are shown to converge to true states and the true gravity parameter almost globally 
over the state space of motion of the rigid spacecraft. These convergence properties are 
verified by numerical simulation for a realistic scenario of a satellite in the proximity of an 
asteroid with spherical mass distribution. The following chapter presents another nonlinear observer for rigid body motion that has finite-time convergence.

\newpage

\section{MODEL-BASED OBSERVER DESIGN WITH FINITE-TIME CONVERGENCE} 

\label{CH03_DSCC2014} 

\hspace{\parindent}
\textit{This chapter is adapted from a paper published in Proceedings of the 2014 ASME Dynamic Systems and Control (DSC) Conference \cite{Izadi2014DSCC}. The author gratefully acknowledges Dr. Amit K. Sanyal and Jan Bohn for their participation.}
\\
\\
{\bf{Abstract}}~
An observer that obtains estimates of the translational and rotational motion states for a 
rigid body under the influence of known forces and moments is presented. This nonlinear 
observer exhibits almost global convergence of state estimates in finite time, based on state 
measurements of the rigid body's pose and velocities. It assumes a known dynamics model 
with known resultant force and resultant torque acting on the body, which may include feedback 
control force and control torque. The observer design based on this model uses the  
exponential coordinates to describe rigid body pose estimation errors on $\SE$, which 
provides an almost global description of the pose estimate error. Finite-time convergence of 
state estimates and the observer are shown using a Lyapunov analysis on the 
nonlinear state space of motion. Numerical simulation results confirm these analytically 
obtained convergence properties for the case that there is no measurement noise and no 
uncertainty (noise) in the dynamics. The robustness of this observer to measurement 
noise in body velocities and additive noise in the force and torque components is also 
shown through numerical simulation results.

\subsection{Rigid Body Dynamics Model}\label{sec: rbdyn}

Consider a body fixed reference frame in the center of mass of a rigid body denoted as 
$\Bref$ and an inertial fixed frame denoted as $\Iref$. Let the rotation matrix from $\Bref$ to 
the inertial fixed frame $\Iref$ be given by $R$ and the coordinates of the origin of $\Bref$ 
with respect to $\Iref$ be denoted as $b$.

\subsubsection{Rigid Body Dynamics}
The rigid body dynamics is given by
\begin{align}
J\dot{\Omega}&=J\Omega\times\Omega+{}^\s{B}\tau,\nn\\
m\dot{\nu}&=m\nu\times\Omega+{}^\s{B}\phi\label{eq:dyn},
\end{align}
where $m$ and $J$ denote the rigid body mass and inertia matrix, respectively, $^\s{B}\phi$ denotes the force applied to the rigid body and $^\s{B}\tau$ the external torque, both expressed in the body reference frame. The dynamics equations \eqref{eq:dyn} can be expressed 
in compact form as
\begin{equation}\label{eq:ddt_xi}
\Imbb\dot\xi=\adast{\xi}\Imbb\xi+\varphi,
\end{equation}
where $\varphi=[{}^\s{B}\tau\T\ {}^\s{B}\phi\T]\T$.

\subsubsection{Kinematics in Exponential Coordinates}
The exponential coordinate vector $\eta\in\bR^6$ for a given configuration $\msg\in\SE$ is given 
by 
\be \eta^\vee=  \logm(\msg)=  \bbm \Theta^\times & \beta \\ 0 & 0 \ebm, \,\mbox{ where }\, 
\eta= \bbm \Theta \\ \beta\ebm,
\label{matlog} \ee
and $\logm$ denotes the matrix logarithm, which is also the inverse of the exponential 
map $\expm: \se\to\SE$. We can obtain the exponential coordinate vector 
$\eta$ from $\msg$ as follows: 
\begin{align} 
\begin{split}
&\Theta^\times = \frac{\theta}{\sin(\theta)}\big( R- R\T\big),\, \mbox{ and }\, \beta = 
S^{-1}(\Theta) b, \\
&\mbox{where } S^{-1}(\Theta) = I-\frac12 \Theta^\times + \left(\frac1{\theta^2} - 
\frac{1+\cos\theta}{2\theta\sin\theta}\right) \big(\Theta^\times\big)^2, 
\end{split} \label{expcoord} 
\end{align}
and $\theta=\|\Theta\|$ is the principal angle of rotation corresponding to the rotation matrix 
$R$. Note that $\Theta$ in \eqref{expcoord} cannot be obtained when $\theta$ is an odd 
multiple of $\pi$ radians. Since all of $\SO$ can be represented by principal angle values 
in the range $\theta\in [0,\pi]$, we can therefore obtain an unique exponential coordinate 
vector for all $\msg\in\SE$ whose $\SO$ component has a principal angle less than $\pi$ 
radians, i.e., $\theta\in [0,\pi)$. 
Therefore, the exponential coordinates can represent almost all poses in $\SE$ excluding 
those with rotations of exactly $\pi$ radians about any axis.

The exponential coordinate vector $\eta\in\bR^6$ (corresponding to $\eta^\vee=\logm(\msg)
\in\se$), satisfies \eqref{eq:ddt_eta}. Note that $\theta=0$ is a removable singularity in equations \eqref{C2ATheta}-\eqref{C2STheta}, and corresponds to the identity orientation on $\SO$. An equivalent expression for $G(\eta)$ given in \cite{Bullo_ECC95} is as follows:
\be G(\eta) = I + \frac12 \ad{\eta} + \alpha(\theta)\ad{\eta}^2 + \beta(\theta)\ad{\eta}^4, 
\label{Gadexpr} \ee
where 
\begin{align} 
\begin{split}
\alpha(\theta) &= \frac2{\theta^2} -\frac{3}{4\theta}\cot(\theta/2) -\frac18\csc^2 (\theta/2), \\
\beta(\theta) &= \frac1{\theta^4} -\frac{1}{2\theta^3}\cot(\theta/2) -\frac{1}{8\theta^2}
\csc^2 (\theta/2). 
\end{split} \label{albeGad}
\end{align}
From the expression \eqref{albeGad}, it is clear that $G(\eta)\eta= \eta$ \cite{Izadi2013DSCC}, a fact used in 
the observer design. The exponential coordinates on $\SE$ were used for observer design 
recently in \cite{Bras2013CDC}. However, the observer design in \cite{Bras2013CDC} had 
asymptotic (exponential) convergence, unlike the observer designed here, which exhibits 
finite-time convergence. 

\subsection{Finite-Time Convergent Observer Design}\label{sec: obsdes}
We assume that a sensor suite onboard a rigid body vehicle provides 
information about the configuration and velocities of the vehicle. Our aim is to design a 
dynamic observer which exploits the sensors measurements (pose and velocities) to estimate the configuration 
(pose) and the velocities, such that the estimated states converge in finite-time to their 
true values in the absence of measurement errors. Robustness to bounded measurement 
errors and noisy inputs to the dynamics model is obtained consequently, and is shown 
through numerical simulation results. 

Consider ($\hat{\msg}$, $\hat{\xi}$) to be estimates of the states ($\msg$, $\xi$) of a rigid body's 
motion on $\SE\times\bR^6$. Define
\begin{align}
h=\hat{\msg}^{-1}\msg\, \mbox{ and }\, \tilde{\eta}^\vee=\logm(h).
\end{align}
Therefore, we obtain:
\begin{align}
\dot{h}=h\tilde{\xi}^\vee \, \mbox{ where }\, \tilde{\xi}=\xi-\Ad{h^{-1}}\hat{\xi}. \label{relkine}
\end{align}
If we define $\breve{\xi}=\Ad{h^{-1}}\hat{\xi}$, then $\tilde{\xi}=\xi-\breve{\xi}$. From 
\eqref{relkine} and the kinematics in exponential coordinates given in the previous section, 
we conclude that
\be \dot{\tilde\eta}= G(\tilde\eta)\tilde\xi. \label{relkinexp} \ee
Further, define
\begin{align}
u=\xi-\Ad{h^{-1}}\hat{\xi}+k\frac{\tilde{\eta}}{(\tilde{\eta}\T\tilde{\eta})^{1-\frac{1}{p}}}=\tilde{\xi}+
k\frac{\tilde{\eta}}{(\tilde{\eta}\T\tilde{\eta})^{1-\frac{1}{p}}},
\end{align}
where $k>0$ and $p\in (1,2)$ is a rational number (preferably a ratio of odd integers, to avoid sign mismatches when taking powers using a computer code). Let
\begin{align}
V=\frac{1}{2}\gamma\tilde{\eta}\T\tilde{\eta}+\frac{1}{2}u\T\mathbb{I}u \label{Lyapfunc}
\end{align}
be a candidate Lyapunov function, where $\gamma>0$, and $\mathbb{I}$ is the complete inertia 
matrix as given in eq. This Lyapunov function is used to show the finite-time convergence of 
the observer design that follows.
\begin{theorem}\label{filter}
The observer dynamics given by:
\begin{align}
\dot{\hat{\msg}} &=\hat{\msg}\hat{\xi}^\vee,\, \mbox{ OR }\, \dot{\tilde\eta}= G(\tilde\eta)\tilde\xi 
\mbox{ and }  \hat{\msg}= \msg\expm (-\tilde\eta^\vee), \label{kinobs} \\
\mathbb{I}\dot{\breve{\xi}} &=\adast{(\breve{\xi}-k(\tilde{\eta}\T\tilde{\eta})^{\frac{1}{p}-1}\tilde{\eta})}\mathbb{I}\xi+\varphi+k\mathbb{I}H(\tilde{\eta})G(\tilde\eta)\tilde{\xi}+\gamma G\T(\tilde{\eta})\tilde{\eta}+k\frac{\mathbb{I}u}{(u\T\mathbb{I}u)^{1-\frac{1}{p}}}, \label{dynobs}
\end{align}
where $H(\tilde{\eta})=\frac{1}{(\tilde{\eta}\T\tilde{\eta})^{1-\frac{1}{p}}}\bigg\{I-2(1-\frac{1}{p})
\frac{\tilde{\eta}\tilde{\eta}\T}{\tilde{\eta}\T\tilde{\eta}}\bigg\}$, and $\varphi$ is the resultant of 
forces and moments acting on the rigid body, ensures that the estimate errors converge to the 
origin in finite time. Thus, $(\tilde{\xi},\tilde{\eta})=(0,0)\in\mathbb{R}^6\times\mathbb{R}^6$ for 
all time $t\geq t_f$, where $t_f$ is finite.
\end{theorem}
\begin{proof}
The time derivative of the Lyapunov function given by \eqref{Lyapfunc} is:
\begin{align}
\dot{V}&=\gamma\tilde{\eta}\T\dot{\tilde{\eta}}+u\T\mathbb{I}\dot{u} \nn \\
&=\gamma\tilde{\eta}\T G(\tilde{\eta})\tilde{\xi}+u\T\mathbb{I}\dot{u} \label{C3dotV}
\end{align}
From \eqref{dynobs} and the dynamics
\begin{align}
\mathbb{I}\dot{\xi}=\adast{\xi}\mathbb{I}\xi+\varphi
\end{align}
of the rigid body, we obtain:
\begin{align}
\mathbb{I}\dot{u}&=\mathbb{I}\bigg\{\dot{\xi}-\dot{\breve{\xi}}+\frac{d}{dt}\big(k(\tilde{\eta}\T\tilde{\eta})^{\frac{1}{p}-1}\tilde{\eta}\big)\bigg\} \nn \\
&=\mathbb{I}\dot{\xi}-\mathbb{I}\dot{\breve{\xi}}+k\mathbb{I}H(\tilde{\eta})\dot{\tilde{\eta}} \nn \\
&=\adast{\xi}\mathbb{I}\xi+\varphi-\mathbb{I}\dot{\breve{\xi}}+k\mathbb{I}H(\tilde{\eta})G(\tilde{\eta})\tilde{\xi}, \label{udot0}
\end{align}
using the kinematics $\dot{\tilde{\eta}}=G(\tilde{\eta})\tilde{\xi}$ in \eqref{kinobs}, which holds for 
the configuration error expressed in exponential coordinates. Substituting for $\dot{\breve{\xi}}$ 
from \eqref{dynobs} into expression \eqref{udot0}, we obtain:
\begin{align}
\mathbb{I}\dot{u}=\adast{u}\mathbb{I}\xi-\gamma G\T(\tilde{\eta})\tilde{\eta}-k\frac{\mathbb{I}u}{(u\T\mathbb{I}u)^{1-\frac{1}{p}}} \label{udot}
\end{align}
for the feedback dynamics of the variable $u$. Now substituting for $\mathbb{I}\dot{u}$ from 
\eqref{udot} into the expression for $\dot{V}$ in \eqref{C3dotV}, we get:
\begin{align}
\dot{V}&=\gamma \tilde{\eta}\T G(\tilde{\eta})\tilde{\xi}-\gamma u\T G\T(\tilde{\eta})\tilde{\eta}-k(u\T\mathbb{I}u)^{\frac{1}{p}} \nn\\
&=\gamma\tilde{\eta}\T G(\tilde{\eta})(\tilde{\xi}-u)-k(u\T\mathbb{I}u)^{\frac{1}{p}}\nn\\
&=-k\gamma\frac{\tilde{\eta}\T G(\tilde{\eta})\tilde{\eta}}{(\tilde{\eta}\T\tilde{\eta})^{1-\frac{1}{p}}}-k(u\T\mathbb{I}u)^{\frac{1}{p}}\nn\\
&=-k\bigg[\gamma(\tilde{\eta}\T\tilde{\eta})^{\frac{1}{p}}-(u\T\mathbb{I}u)^{\frac{1}{p}}\bigg] 
\label{dV}
\end{align}
using the fact that $G(\tilde{\eta})\tilde{\eta}=\tilde{\eta}$. From \eqref{dV}, we note that
\begin{align}
\dot{V}=-2^{\frac{1}{p}}k\bigg[\gamma^{1-\frac{1}{p}}(\frac{\gamma}{2}\tilde{\eta}\T\tilde{\eta})^{\frac{1}{p}}+(\frac{1}{2}u\T\mathbb{I}u)^{\frac{1}{p}}\bigg]\leq-2^{\frac{1}{p}}kV^\frac{1}{p}, \label{Vdot}
\end{align}
for $\gamma\geq 1$, using the binomial expansion theorem. Thus, $V$ converges to zero in 
finite time, and hence the result.
\end{proof}

Note that since the exponential coordinates are defined almost globally on the configuration 
space $\SE$, the above observer can be used for all initial estimate errors $\big(h(0),\tilde\xi
(0)\big)\in\SE\times\bR^6$ such that the principal angle corresponding to the $\SO$ component 
of $h(0)$ is not $\pi$ radians (or 180$^\circ$). Therefore, the above observer is finite-time 
convergent, and its domain of convergence is almost global on the state space $\SE\times\bR^6$. 

\subsection{Numerical Simulations}\label{sec: simres}
In this section, this observer numerical simulated for a rigid body that represents a maneuverable 
aerial vehicle with known mass and inertia. The observer needs a set of measured states (pose and rigid body velocities) to 
estimate the exponential coordinates as well as the rigid body's velocities. Measurement data  
is generated by integrating the ``true" (known) dynamics of the rigid body offline with known 
models of external torques and forces, and then adding noise. The rigid body's mass is 
assumed to be $m=21$ kg and its inertia matrix is
\begin{equation}
J=\diag (2.56,3.01,2.98)\;\; \mbox{kg.m}^2.
\end{equation}

The rigid body is subjected to an external force and an external torque, which are expressed 
in the body-fixed frame as
\begin{align}
\phi_D^0=[0.4\;\ 0.5\;\ 0.768]\T \mbox{\;N} \mbox{ and } \tau_D^0=[0.07\;\ 0.0687\;\ 0.02]\T \mbox{\;N.m},
\end{align}
as well as a uniform gravity force directed towards the negative z-axis of the inertial frame. 
The initial 
configuration (pose) is given by
\begin{align}
R_0=\expm (0.1[4\ 2\ 1]^\times),\ 
b_0=[1\ 2\ 3]\T\ \mbox{m},
\end{align}
and the initial angular and linear velocities are 
$$\omega_0=[0.5\;\;\; -0.5\;\;\; 0.1]\T\ \text{rad/s},\;\;\; v_0=10^{-3}[-5\;\;\; 25\;\;\; 30]\T\ \text{m/s}.$$
A discrete-time numerical integration scheme with constant time stepsize is used to propagate 
the true states as well as the estimated ones. The discrete time period for this numerical 
integration scheme, $h$, is chosen to be $0.01$ s. The true states are propagated for 
$T=0\mbox{ to }10$ s. 

Using the above initial states, integration scheme, time interval and step size, the dynamics of the 
rigid body is integrated and its trajectory in three dimensional space is depicted in 
Fig. \ref{C3fig:Dynamics}. The body frame axes are also plotted on this path every 1 second to 
show the attitude motion.
\begin{figure}[thb]
\centering
\psfrag{t1}[c]{Error of exponential coordinates ($\tilde{\eta}$)}
\psfrag{xb1}[c]{$t\ \mathrm{(s)}$}
\psfrag{ya1}[c]{$\Theta$ (rad)}
\psfrag{yb1}[cb]{$\beta$ (m)}
\includegraphics[width=0.8\columnwidth]{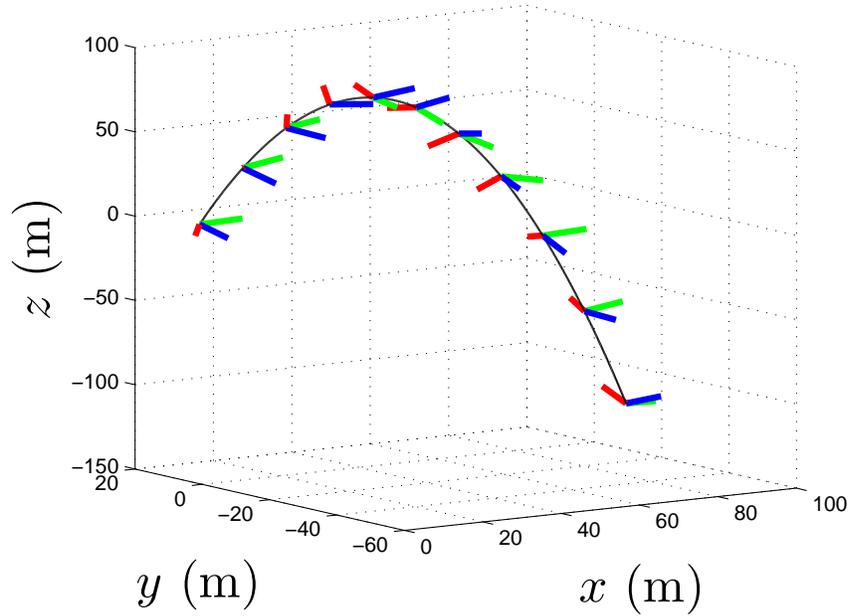}
\caption{Rigid body actual trajectory with its attitude.}
\label{C3fig:Dynamics}
\end{figure}

The initial estimated pose and velocity have been taken to be $\hat{g}_0=I_{4\times 4}$ and 
$\hat{\xi}_0=0_{6\times 1}$ repectively. In order to get fast convergence and smooth estimation 
error with a reasonable overshoot, the best design parameters of the observer arrived at 
using trial and error are:
\begin{align}
k=50,\ 
p=\frac{23}{21},\ 
\gamma=0.03.\ 
\end{align}
These parameter values were used in simulating the observer's performance with and 
without measurement noise and external disturbance. 

\subsubsection{In the absence of noise and disturbance}
Here, we assume that there is no measurement noise or external disturbance and the sensors 
are ideal. 
The measurement sampling period is taken to be $0.01$ s. This simulation runs for $1.5$ s, which is long enough for the estimation errors in exponential coordinates and velocities to 
converge to zero.
The principal angle of the attitude estimation error as well as the estimate errors in the Cartesian coordinates of the rigid body are depicted in Fig.~\ref{C3fig:Exp_coords_est_error_won}. All these components have been derived from the exponential coordinates estimation errors proposed in Theorem \ref{filter}.
The estimation errors in the angular and translational velocities of the rigid body during the simulation are shown in Fig.~\ref{C3fig:Vel_est_error_won}. 

\begin{figure}[thb]
\centering
\psfrag{t1}[c]{Error of exponential coordinates ($\tilde{\eta}$)}
\psfrag{xb1}[c]{$t\ \mathrm{(s)}$}
\psfrag{ya1}[c]{$\Theta$ (rad)}
\psfrag{yb1}[cb]{$\beta$ (m)}
\includegraphics[width=0.8\columnwidth]{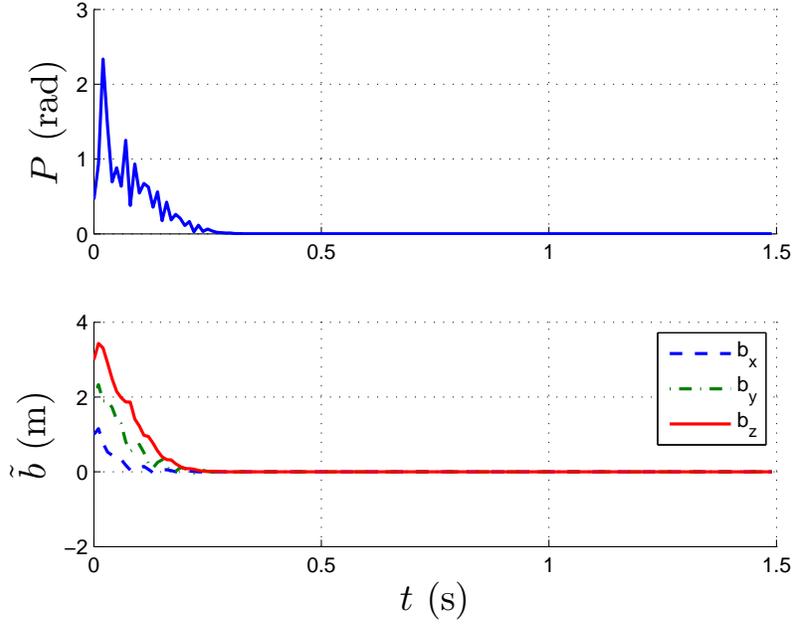}
\caption{Principal angle of attitude estimate error and rigid body's position estimate error in the absence of noise or disturbance.}
\label{C3fig:Exp_coords_est_error_won}
\end{figure}
\begin{figure}[thb]
\centering
\psfrag{t2}[c]{Error of velocities ($\tilde{\xi}$)}
\psfrag{xb2}[c]{$t\ \mathrm{(s)}$}
\psfrag{ya2}[c]{$\omega$ (rad/s)}
\psfrag{yb2}[cb]{$\nu$ (m/s)}
\includegraphics[width=0.8\columnwidth]{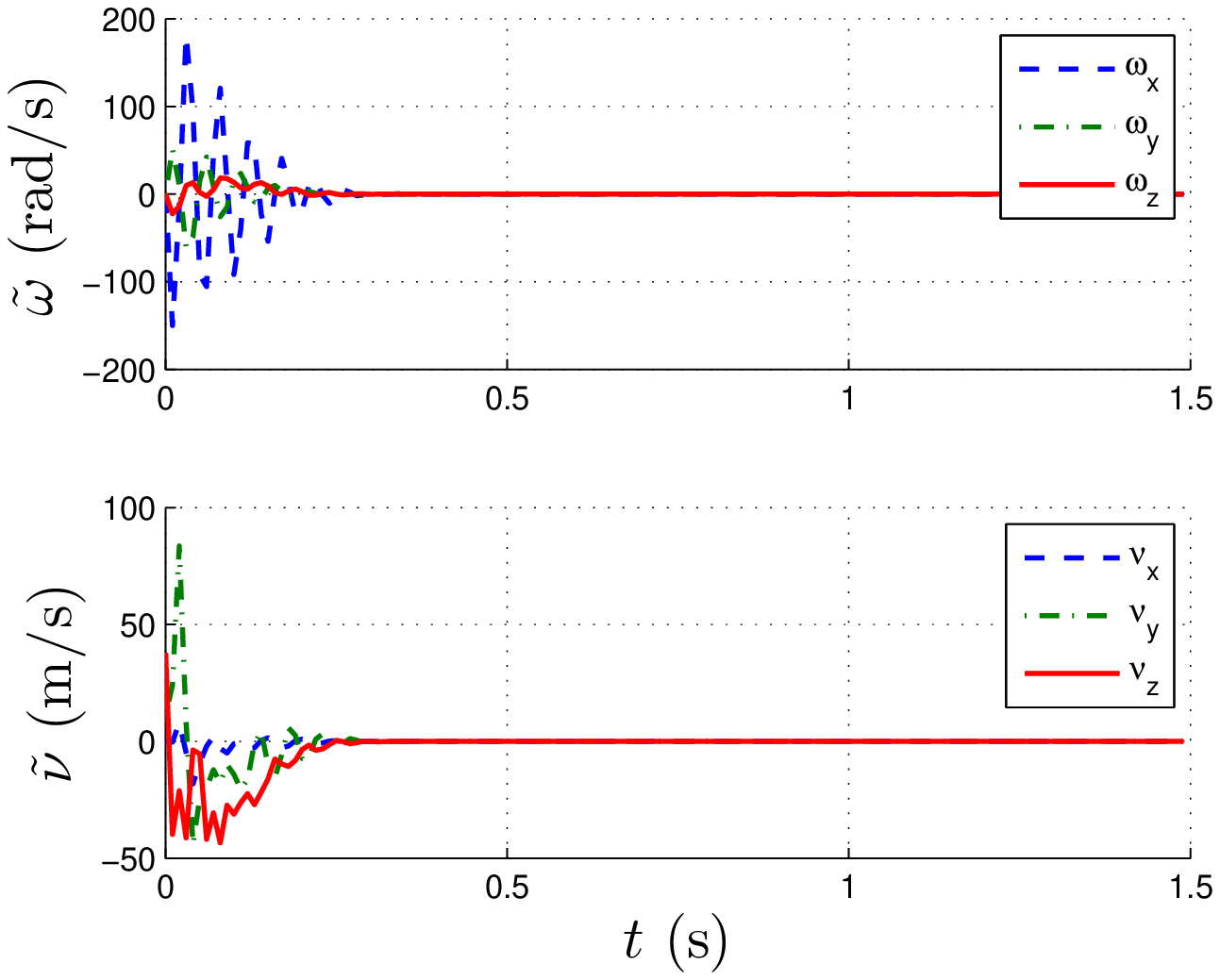}
\caption{Rigid body's angular and translational velocities estimate errors in the absence of noise or disturbance.}
\label{C3fig:Vel_est_error_won}
\end{figure}

These figures show that the estimation errors converge to zero in a very small and finite time, which is almost 0.3 s here. Just due to the numerical artifacts, the errors will not be exactly zero, but after the mentioned finite time all of their components are negligible. 

\subsubsection{In the presence of noise and disturbance}
In this section, we show that the observer is robust to measurement noises and external disturbances. Presence of measurement errors and external disturbances is always the case in reality, where the sensors data contain some certain levels of measurement noise. First of all, the dynamics of the system is mimicked to generate the pure states. Next, some noise signals with a realistic level of available rough sensors are added to each set of states. The noises in the position and attitude data are sinusoidal signals with the amplitudes of $10$ cm and $2^\circ$, respectively. Their frequencies also are both $100$ Hz. The same kind of signals are added to the pure velocities, but with different amplitudes, which are $1^\circ$/s and $1$ cm/s, respectively. Note that the observer proposed in Theorem \ref{filter} does not use the pose in all time steps. Therefore, just the noise in the initial value of position and attitude affect the estimated states. On the other hand, the noisy angular and translational velocities are used as a feedback in each step of the estimation. The external disturbances are assumed to be added to the previous external torques, which make the total torques applied to the rigid body equal to
\begin{align}
\tau_D&=
\tau_D^0+0.01\sin (0.1 t) \begin{bmatrix}
0.424\\
0.9 \\
0.1
\end{bmatrix}\mbox{\;N.m.}
\end{align}

The total external forces applied on the rigid body also are taken to have disturbance as
\begin{align}
\phi_D&=\phi_D^0
+0.1\sin (0.1 t) \begin{bmatrix}
0.1 \\
0.2 \\
0.975
\end{bmatrix}\mbox{\;N.}
\end{align}

In Fig.~\ref{C3fig:Exp_coords_est_error_wn} the principal angle corresponding to the rigid body attitude estimate error is plotted. This figure also depicts the rigid body position estimation errors by components. The estimation errors of the rigid body velocities are shown in
Fig.~\ref{C3fig:Vel_est_error_wn}. These two figures show that the estimation errors in all the pose and velocities converge to zero in a finite time almost as fast as the noise-free case. Thus, the proposed observer can estimate the real states even in the presence of additive noise in the dynamics model. This was expected, since the finite-time convergent systems have been shown to be robust to bounded external disturbances and measurement errors.

\begin{figure}[thb]
\centering
\psfrag{t3}[c]{Error of exponential coordinates ($\tilde{\eta}$)}
\psfrag{xb3}[c]{$t\ \mathrm{(s)}$}
\psfrag{ya3}[c]{$\Theta$ (rad)}
\psfrag{yb3}[cb]{$\beta$ (m)}
\includegraphics[width=0.75\columnwidth]{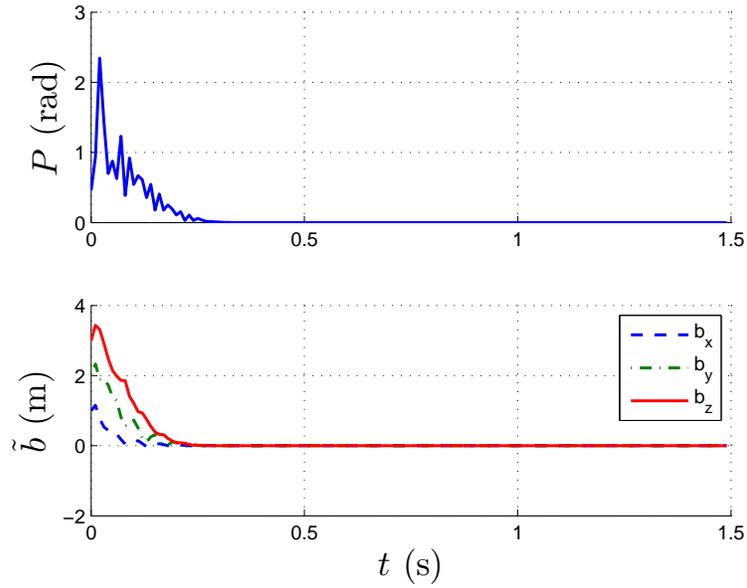}
\caption{Principal angle of attitude estimate error and rigid body's position estimate error in the presence of noise and disturbance.}
\label{C3fig:Exp_coords_est_error_wn}
\end{figure}
\begin{figure}[thb]
\centering
\psfrag{t4}[c]{Error of velocities ($\tilde{\xi}$)}
\psfrag{xb4}[c]{$t\ \mathrm{(s)}$}
\psfrag{ya4}[c]{$\omega$ (rad/s)}
\psfrag{yb4}[cb]{$\nu$ (m/s)}
\includegraphics[width=0.8\columnwidth]{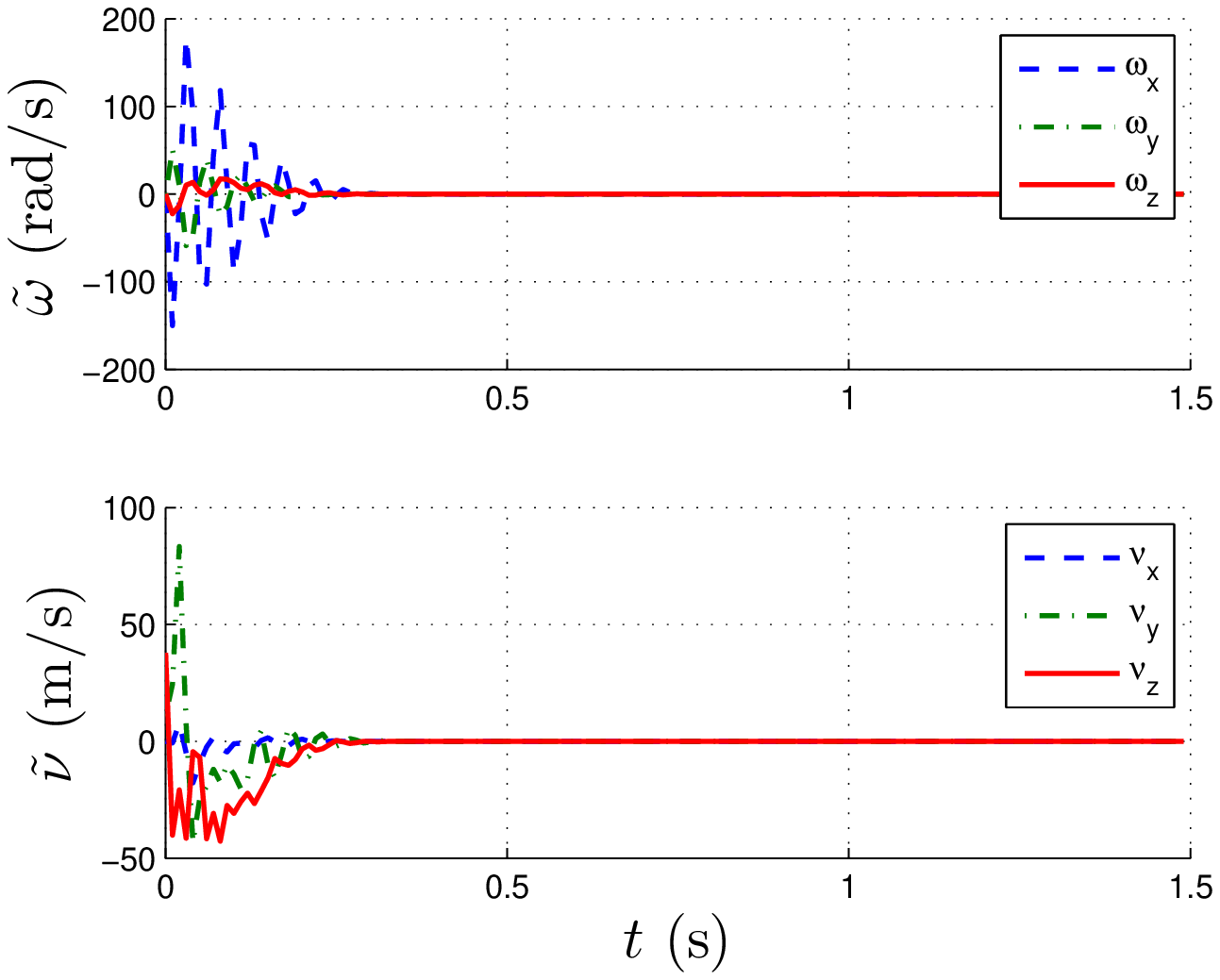}
\caption{Rigid body's angular and translational velocities estimate errors in the presence of noise and disturbance.}
\label{C3fig:Vel_est_error_wn}
\end{figure}

\subsection{Conclusion}\label{sec: conc}
An observer that exhibits finite-time convergence for arbitrary rigid body motion states in the 
tangent bundle of the special Euclidean group $\SE$ is presented. The observer design is 
based on use of the exponential coordinates, which are defined almost globally over 
the configuration space $\SE$. Therefore, the domain of convergence is almost 
global on this state space, and excludes only those initial pose errors whose rotation 
component has a principal angle of exactly $\pi$ radians. Since finite-time convergence has 
been shown to be more robust to noise in the dynamics or in measurements, this 
observer is expected to be robust to both measurement noise and process noise. Such 
robustness properties are indicated in simulation results obtained from a numerical 
implementation of this observer. The estimation errors in the absence of measurement 
noise and in the presence of measurement noise are seen to converge in a finite-time 
that depends on the observer design parameters.

\clearpage

\section{MODEL-FREE RIGID BODY ATTITUDE ESTIMATION BASED ON THE LAGRANGE-D'ALEMBERT PRINCIPLE} 

\label{CH04_VE} 

\hspace{\parindent}
\textit{This chapter is adapted from a paper published in Automatica \cite{Automatica} and a paper published in Proceedings of the 2015 American Control Conference \cite{ACC2015}. The author would like to thank Dr. Robert Mahony and Dr. Tarek Hamel for their comments during preparation of the initial manuscript. The author also gratefully acknowledges Dr. Amit K. Sanyal, Ehsan Samiei and Sasi P. Viswanathan for their participation.}
\\
\\
{\bf{Abstract}}~
Estimation of rigid body attitude and angular velocity without any knowledge of the attitude 
dynamics model, is treated using the Lagrange-d'Alembert principle from variational 
mechanics. It is shown that Wahba's cost function for attitude determination from two or 
more non-collinear vector measurements can be generalized and represented as a 
Morse function of the attitude estimation error on the Lie group of rigid body rotations. 
With body-fixed sensor measurements of direction vectors and angular velocity, a 
Lagrangian is obtained as the difference between a kinetic energy-like term that is 
quadratic in the angular velocity estimation error and an artificial potential obtained 
from Wahba's function. An additional dissipation term that depends on the angular velocity 
estimation error is introduced, and the Lagrange-d'Alembert principle is applied to the 
Lagrangian with this dissipation. A Lyapunov analysis shows that the state estimation 
scheme so obtained provides stable asymptotic convergence of state estimates to actual 
states in the absence of measurement noise, with an almost global domain of attraction. 
These estimation schemes are discretized for computer implementation using discrete 
variational mechanics. A first order implicit Lie group variational integrator is obtained as a 
discrete-time implementation and its adjoint flow yields an explicit first order LGVI. Composing these two first order flows, a symmetric second-order version of 
this discrete-time filtering scheme is also presented. In the 
presence of bounded measurement noise, numerical simulations show that the 
estimated states converge to a bounded neighborhood of the actual states. A comparison between the performances of the second-order filter and the 
first-order filter is also carried out.

\subsection{Attitude Determination from Vector Measurements}\label{C4S1}

Rigid body attitude is determined from $k\in\mathbb{N}$ known inertial vectors measured in 
a coordinate frame fixed to the rigid body. Let these vectors be denoted as $u_i^m$ for 
$i=1,2,\ldots,k$, in the body-fixed frame. The assumption that $k\ge 2$ is necessary for 
instantaneous three-dimensional attitude determination. When $k=2$, the cross product 
of the two measured vectors is considered as a third measurement for applying the attitude 
estimation scheme. Denote the corresponding known inertial vectors as seen from the rigid body as $e_i$, and 
let the true vectors in the body frame be denoted $u_i=R \T e_i$, where $R$ is the rotation 
matrix from the body frame to the inertial frame. This rotation matrix provides a coordinate-free,  
global and unique description of the attitude of the rigid body. Define the matrix composed 
of all $k$ measured vectors expressed in the body-fixed frame as column vectors,
\begin{align}
U^m&= [u_1^m\ u_2^m\ u_1^m\times u_2^m]  \mbox{ when }\, k=2,\, \mbox{ and }\nn\\ 
U^m&=[u_1^m\ u_2^m\ ... u_k^m]\in \mathbb{R}^{3\times k} \mbox{ when } k>2,\label{Umexp}
\end{align}
and the corresponding matrix of all these vectors expressed in the inertial frame as
\begin{align}
E&= [e_1\ e_2\ e_1\times e_2]  \mbox{ when }\, k=2,\, \mbox{ and }\nn\\ 
E&=[e_1\ e_2\ ... e_k]\in \mathbb{R}^{3\times k}  \mbox{ when } k>2.\label{Eexp}
\end{align}
Note that the matrix of the actual body vectors $u_i$ corresponding to the inertial vectors 
$e_i$, is given by
\begin{align}
U&= R \T E= [u_1\ u_2\ u_1\times u_2]  \mbox{ when }\, k=2,\, \mbox{ and }\nn\\ 
U&=R \T E=[u_1\ u_2\ ... u_k]\in \mathbb{R}^{3\times k} \mbox{ when } k>2.\nn
\end{align}

\subsubsection{Generalization of Wahba's Cost Function for Instantaneous Attitude 
Determination from Vector Measurements}
The optimal attitude determination problem for a set of vector measurements at a given time 
instant, is to find an estimated rotation matrix $\hat R\in\SO$ such that a weighted sum of the 
squared norms of the vector errors 
\begin{align}
s_i= e_i-\hat R u_i^m \label{dirmeaserrs}
\end{align}
are minimized. This attitude determination problem is known as Wahba's problem, and is 
the problem of minimizing the value of
\begin{align}
\cU^0 (\hat R, U^m)= \frac{1}{2}\sum_{i=1}^kw_i(e_i-\hat R u_i^m) \T(e_i-\hat R u_i^m)
\label{eq:J0}
\end{align}  %
with respect to $\hat R\in\SO$, where the weights $w_i>0$. Defining the trace inner product 
on $\mathbb{R}^{m\times n}$ as
\begin{align}
\lan A_1,A_2\ran=\tr(A_1 \T A_2),\label{eq:tr_def}
\end{align}
we can re-express equation \eqref{eq:J0} for Wahba's cost function as
\be 
\cU^0 (\hat R, U^m)= \frac{1}{2}\lan E-\hat R U^m,(E-\hat R U^m)W\ran, \label{U0def} \ee 
where $U^m$ is given by equation \eqref{Umexp}, $E$ is given by \eqref{Eexp}, and 
$W=\diag(w_i)$ is the positive diagonal matrix of the weight factors for the measured 
directions. 

From the expression \eqref{U0def}, note that $W$ may be generalized to be any positive 
definite matrix, not necessarily diagonal. Another 
generalization of Wahba's cost function is given by 
\begin{align}
\cU (\hat R,U^m)=\Phi \Big(\frac{1}{2}\lan E-\hat R U^m,(E-\hat R U^m)W\ran \Big), \label{attindex}
\end{align}
where $\Phi: [0,\infty)\mapsto[0,\infty)$ is a $C^2$ function that satisfies $\Phi(0)=0$ and 
$\Phi'(x)>0$ for all $x\in[0,\infty)$. Furthermore, $\Phi'(\cdot)\leq\alpha(\cdot)$ where 
$\alpha(\cdot)$ is a Class-$\mathcal{K}$ function. Note that these properties of 
$\Phi(\cdot)$ ensure that the indices $\cU^0 (\hat R, U^m)$ and $\cU (\hat R, U^m)$ have the 
same minimizer $\hat R\in\SO$.  In other words, minimizing the cost $\cU$, which is a 
generalization of the cost $\cU^0$, is equivalent to solving Wahba's problem. 
Here, $W$ is positive definite (not necessarily diagonal), and $E$ and $U^m$ are assumed to 
be of rank 3, which is true under the assumption that $k\ge2$ vectors are measured. The solution 
to Wahba's problem is given in \cite{san06} and \cite{markley1988attitude}.

\subsubsection{Properties of Wahba's Cost Function in the Absence of Measurement Errors}
In the absence of measurement errors, $U^m= U=R\T E$, and let $Q=R\hat{R}\T\in\SO$ denote
the attitude estimation error. The following lemmas give the structure of Wahba's cost 
function in this case.
\begin{lemma} \label{lem1}
Let $\mbox{rank}(E)=3$, where $E$ is as defined in \eqref{Eexp}.
Let the singular value decomposition of $E$ be given by
\begin{align}
E :&= U_E \Sigma_E V_E\T\, \mbox{ where }\,U_E\in\mathrm{O}(3),\ V_E\in\mathrm{O}(m),\nn\\
&\Sigma_E\in\mathrm{Diag}^+(3,m),\label{SVDE}
\end{align}
and $\mathrm{Diag}^+(n_1,n_2)$ is the vector space of $n_1\times n_2$ matrices with 
positive entries along the main diagonal and all other components zero. Let $\sigma_1, 
\sigma_2, \sigma_3$ denote the main diagonal entries of $\Sigma_E$. Further, let the positive 
definite weight matrix $W$ in the generalization of Wahba's cost function \eqref{attindex} 
be given by 
\be W= V_E W_0 V_E\T\, \mbox{ where }\, W_0\in\mathrm{Diag}^+(m,m) \label{Wdec} \ee
and the first three diagonal entries of $W_0$ are given by
\be w_1= \frac{d_1}{\sigma_1^2},\; w_2=\frac{d_2}{\sigma_2^2},\; w_3=\frac{d_3}
{\sigma_3^2}\, \mbox{ where }\, d_1,d_2,d_3>0. \label{w123} \ee
Then $K=EWE\T$ is positive definite and 
\be K= U_E\Delta U_E\T\, \mbox{ where }\, \Delta= 
\diag(d_1,d_2,d_3), \label{Kdec} \ee
is its eigendecomposition. Moreover, if 
$d_i\ne d_j\mbox{ for } i\ne j$ and $i,j\in\{1,2,3\}$, then $\lan I- Q,K\ran$ is a Morse function whose critical points are
\begin{align}
Q\in \big\{ I, Q_1, Q_2, Q_3\big\} \mbox{ where } Q_i= 2 U_E a_i a_i\T U_E\T - I, 
\label{critpts}
\end{align}
and $a_i$ is the $i$th column vector of the identity $I\in\SO$.
\end{lemma}
{\em Proof}:
It is straightforward to show that \eqref{Kdec} holds given \eqref{SVDE}-\eqref{w123}. It is shown here that $\lan I- Q,K\ran$ has the isolated non-degenerate critical points 
given by \eqref{critpts}. Consider a first differential in $Q$ given by  
\be \delta Q= Q\Sigma^\times, \label{Qvar} \ee
where $\Sigma\in\bR^3$.
The first variation of $\lan I- Q,K\ran$ with respect to $Q$ is given by
\begin{align}
&\partial_Q \lan I- Q,K\ran= \lan K, -\delta Q\ran= \tr\Big(\frac12 (Q\T K- KQ)\Sigma^\times\Big) \nn\\
&= \frac12\lan KQ- Q\T K,\Sigma^\times\ran= S_K\T (Q) \Sigma,
\label{gradQ}
\end{align}
where 
\be S_K (Q)= \mrm{vex}\big(KQ- Q\T K\big) \label{SKdef} \ee
and $\mrm{vex}(\cdot): \so\to\bR^3$ is the inverse of the $(\cdot)^\times$ map. The critical points 
of $\lan I- Q,K\ran$ on $\SO$ are therefore given by
\be S_K (Q)=0\, \Rightarrow\, KQ= Q\T K. \label{neccond} \ee
Substituting the eigendecomposition of $K$ given by \eqref{Kdec} in equation 
\eqref{neccond}, we obtain
\begin{align}
U_E\Delta U_E\T Q = Q\T U_E\Delta U_E\T
\Rightarrow  \Delta P &= P\T \Delta, \label{critcond}
\end{align}
where $P=U_E\T Q U_E\in\SO$. Given that $\Delta$ is a positive diagonal matrix with distinct 
diagonal entries, the solution set for $P$ that satisfies the condition \eqref{critcond} is 
\begin{align}
C_P&= \big\{ I, \diag (1,-1,-1), \diag (-1,1,-1), \diag(-1,-1,1)\big\}\nn\\
 &=\big\{ I, 2a_1 a_1\T-I, 2a_2 a_2\T -I, 2 a_3 a_3\T-I\big\}. \label{setCP} 
\end{align}
Thus, the set of critical points of $\lan I- Q,K\ran$ is given by 
\begin{align}
C_Q= U_E C_P U_E\T= \big\{I, Q_1, Q_2, Q_3\big\}, \label{setCQ}
\end{align}
where $Q_1$, $Q_2$ and $Q_3$ are as given by \eqref{critpts}. These critical points are 
clearly isolated. To show that they are non-degenerate, we evaluate the second variation 
of $\lan I- Q,K\ran$ at $Q\in C_Q\subset\SO$, as follows:
\begin{align*}
&\partial^2_Q \lan I- Q,K\ran= -\lan Q\T K, \delta\Sigma^\times\ran + \lan \Sigma^\times Q\T K,
\Sigma^\times\ran.
\end{align*}
Since $Q\T K$ is symmetric at the critical points according to \eqref{neccond}, and since 
$\delta\Sigma^\times$ is 
clearly skew-symmetric, the first term on the right-hand side of the above expression 
vanishes, as symmetric and skew-symmetric matrices are orthogonal under the trace inner 
product. Therefore the second variation of $\lan I- Q,K\ran$ evaluated at the critical points 
$Q\in C_Q$ is given by
\begin{align}
&\partial^2_Q \lan I- Q,K\ran= \lan \Sigma^\times Q\T K,\Sigma^\times\ran = -\lan Q\T K, 
(\Sigma^\times)^2\ran. \label{sufcond}
\end{align}
Since $(\Sigma^\times)^2$ is symmetric, the second variation vanishes for arbitrary 
non-zero $\Sigma^\times$ if and only if 
$Q\T K=0$ for $Q\in C_Q$. However, that possibility would contradict the positive definiteness 
of $K$, which we have already established. Therefore, the critical points of $\lan I- Q,K\ran$ 
are non-degenerate and isolated, which makes this a Morse function on $\SO$~\cite{bo:miln}.
\hfill\ensuremath{\square}

Note that this lemma specifies the weight matrix $W$ according to the SVD of the matrix  
$E$ and selected eigenvalues $d_1,d_2,d_3>0$ for the matrix $K=EWE\T$. As the following lemma shows, these eigenvalues play a special role in determining the overall 
properties of Wahba's cost function and its generalization.

Note that since $\lan I- Q, K\ran$ is a Morse function on $\SO$ by Lemma \ref{lem1}, by 
the properties of the function $\Phi$, one can conclude that $\Phi(\lan I- Q, K\ran):\SO\to\bR$ 
is also a Morse function with the same critical points as those of $\lan I- Q, K\ran$. The 
following result gives the characteristics of the critical points of $\Phi(\lan I- Q, K\ran)$. 
\begin{lemma} \label{lem2}
Let $K=EWE\T$ have the properties given by Lemma \ref{lem1}. Then the critical points of 
$\Phi(\lan I- Q, K\ran):\SO\to\bR$ given by \eqref{critpts}  
consist of a global minimum at the identity $I\in\SO$, a global maximum, and two hyperbolic 
saddle points whose indices depend on the distinct eigenvalues $d_1$, $d_2$, and $d_3$ 
of $K$.
\end{lemma}
{\em Proof}: 
The characteristics of these critical points are obtained from the second variation of 
$\Phi(\lan I- Q, K\ran)$ with respect to $Q\in C_Q$, which was obtained in \eqref{sufcond}. We 
express \eqref{sufcond} as follows:
\begin{align} 
\partial^2_Q \lan I- Q,K\ran= \lan \Sigma^\times Q\T K,\Sigma^\times\ran = F_K\T(Q,\Sigma)\Sigma,
\label{woHess}
\end{align}
where $F_K(Q,\Sigma)= \mrm{vex}\big(KQ\Sigma^\times+\Sigma^\times Q\T K\big)$.
To express $F_K(Q,\Sigma)$ as a vector in $\bR^3$, the following identity is useful:
\begin{align}
\mrm{vex}\big(A\T\Sigma^\times +\Sigma^\times A\big)= \big(\tr (A)I- A\big)\Sigma,\label{niceident}
\end{align}
for $A\in\bR^{3\times 3}$ and $\Sigma\in\bR^3$. Using identity \eqref{niceident} in the expression \eqref{woHess}, one obtains $F_K(Q,\Sigma)= H_K(Q)\Sigma$, where
\begin{align}
H_K(Q)= \tr(Q\T K) I - Q\T K. \label{rHess}
\end{align}
Note that $H_K(Q)$ corresponds to the Hessian matrix of $\lan I- Q, K\ran$ for $Q\in C_Q$. 
Moreover, at the critical points $Q_i$ ($i=1,2,3$) defined 
by \eqref{critpts}, $\Delta P_i= \Delta (2a_i a_i\T- I)$ is a diagonal matrix that is not positive 
definite. The Hessian at these critical points is therefore evaluated to be:
\begin{align}
&H_K(Q_i)= U_E\Lambda_i U_E\T,\, \Lambda_i =\tr (\Delta P_i)I- \Delta P_i,\,\;\; i=1,2,3,\, \mbox{such that} \nn\\
&\Lambda_1= \diag (-d_2-d_3, d_1-d_3, d_1-d_2), \,\nn\\
&\Lambda_2= \diag (d_2-d_3,-d_3-d_1, d_2-d_1),\nn\\
\mbox{and } &\Lambda_3= \diag (d_3-d_2, d_3-d_1, -d_1-d_2). 
\end{align}
Clearly, the indices of these critical points depend on the distinct eigenvalues $d_1$, $d_2$ 
and $d_3$. For example, if $d_1>d_2>d_3$, then the index of $Q_1$ is one, the index of 
$Q_2$ is two, and the index of $Q_3$ is three, which makes $Q_3$ the global maximum of 
$\lan I-Q, K\ran :\SO\to\bR$. Note that the identity $I\in\SO$ is the global minimum of this 
function since the Hessian evaluated at the identity is
\begin{align}
H_K(I) &= \tr (K)I- K= U_E\Lambda_0 U_E\T,\nn\\
\mbox{ where } \Lambda_0&= \diag (d_2+d_3, d_3+d_1, d_1+d_2),
\end{align}
and therefore the identity is a critical point with index zero. Finally, note that the second 
variation of $\Phi(\lan I- Q, K\ran):\SO\to\bR$ evaluated at its critical points is given by
\begin{align}
\partial^2_Q \Phi\big(\lan I- Q,K\ran\big)&= \Phi'\big(\lan I- Q,K\ran\big)\partial^2_Q 
\lan I- Q,K\ran\nn\\
&= \Phi'\big(\lan I- Q,K\ran\big)\Sigma\T H_K(Q)\Sigma\, \mbox{ for } Q\in C_Q. 
\label{HessPhi}
\end{align}
Since $\Phi$ is a Class-$\mathcal{K}$ function, the critical points 
and their indices are identical for $\Phi(\lan I- Q, K\ran)$ and $\lan I- Q, K\ran$.
\hfill\ensuremath{\square}


\subsection{Attitude State Estimation Based on the Lagrange-d'Alembert Principle}
Let $\Omega\in\mathbb{R}^3$ be the angular velocity of the rigid body expressed in the 
body-fixed frame. The attitude kinematics is given by Poisson's equation:
\begin{align}
\dot{R}=R\Omega^\times.
\label{Poisson}
\end{align}
In order to obtain attitude state estimation schemes from continuous-time 
vector and angular velocity measurements, we apply the Lagrange-d'Alembert principle 
to an action functional of a Lagrangian of the state estimate errors, with a dissipation term 
in the angular velocity estimate error. This section presents an estimation scheme obtained 
using this approach, as well as stability properties of this estimator. 

\subsubsection{Action Functional of the Lagrangian of State Estimate Errors}
The ``energy" contained in the errors between the estimated and the measured 
inertial vectors is given by $\cU (\hat R,U^m)$, where $\cU:\SO\times\bR^{3\times k}\to\bR$ is 
defined by \eqref{attindex} and depends on the attitude estimate. The ``energy" contained in 
the vector error between the estimated and the measured angular velocity is given by
\be \cT (\hat\Omega,\Omega^m)=\frac{m}{2}(\Omega^m-\hat\Omega) \T (\Omega^m
-\hat\Omega). \label{angvelindex} \ee
where $m$ is a positive scalar. One can consider the Lagrangian composed of these 
``energy" quantities, as follows:
\begin{align} 
\cL (\hat R,U^m,\hat\Omega,\Omega^m) &= \cT(\hat\Omega,\Omega^m)-\cU (\hat R,U^m) \label{cLag}\\
&=\frac{m}{2}(\Omega^m-\hat\Omega) \T (\Omega^m-\hat\Omega)-\Phi\Big( \frac12\lan E-\hat{R}U^m,(E-\hat{R}U^m)W\ran\Big). \nn
\end{align}
If the estimation process is started at time $t_0$, then the action functional 
of the Lagrangian \eqref{cLag} over the time duration $[t_0,T]$ is expressed as
\begin{align}
\cS (\cL(\hat R&,U^m,\hat\Omega,\Omega^m))= \int_{t_0}^T \big(\cT (\hat\Omega,
\Omega^m)- \cU (\hat R,U^m)\big)\di s \label{eq:J6}\\
=& \int_{t_0}^T \bigg\{ \frac{m}{2}(\Omega^m-\hat\Omega) \T (\Omega^m-
\hat\Omega)- \Phi\Big(\frac{1}{2}\lan E-\hat{R}U^m,(E-\hat{R}U^m)W\ran\Big) \bigg\} \di s.\nn
\end{align}

\subsubsection{Variational Filtering Scheme}
Consider attitude state estimation in continuous time in the presence of measurement noise 
and initial state estimate errors. Applying the Lagrange-d'Alembert principle to the action 
functional $\cS (\cL(\hat R,U^m,\hat\Omega,\Omega^m))$ given by \eqref{eq:J6}, in the 
presence of a dissipation term on $\omega:= \Omega^m-\hat\Omega$, leads to the 
following attitude and angular velocity filtering scheme. 
\begin{proposition} \label{filterN}
The filter equations for a rigid body with the attitude kinematics \eqref{Poisson} and with 
measurements of vectors and angular velocity in a body-fixed frame, are of the form
\begin{align}
\begin{cases}
&\dot{\hat{R}}=\hat{R}\hat{\Omega}^\times=\hat{R}(\Omega^m-\omega)^\times,\vspace{3mm}\\
&m\dot{\omega}= -m\hat{\Omega}\times \omega+\Phi'\big(\cU^0(\hat{R},U^m)\big)S_L(\hat{R})-D\omega,\vspace{3mm}\\
&\hat\Omega=\Omega^m-\omega,
\end{cases}
\label{eq:filterNoise}
\end{align}
where $D$ is a positive definite filter gain matrix, $\hat{R}(t_0)=\hat{R}_0$, $\omega(t_0)=\omega_0
=\Omega^m_0-\hat\Omega_0$, $S_L(\hat R)= \mrm{vex}\big(L\T \hat R - \hat R \T L\big)\in\bR^3$, 
$L=EW(U^m)\T$ and $W$ is chosen to satisfy the conditions in Lemma \ref{lem1}.
\label{filter1}
\end{proposition}
{\em Proof}: In order to find a filter equation which reduces the measurement noise in the estimated attitude, one may take the first variation of the action functional \eqref{eq:J6} with respect to $\hat R$ and $\hat\Omega$. Consider the potential term $\cU^0(\hat R,U^m)$ as 
defined by \eqref{U0def}.
Taking the first variation of this function with respect to $\hat{R}$ gives
\begin{align}
\delta\cU^0&=\lan -\delta\hat R U^m,(E-\hat R U^m)W \ran\nn\\
&=\frac{1}{2}\lan \Sigma^\times, U^m WE\T\hat R-\hat R\T EW(U^m)\T \ran,\nn\\
&=\frac{1}{2}\lan \Sigma^\times, L\T\hat R-\hat R\T L \ran=S\T_L(\hat R)\Sigma.
\end{align}
Now consider $\cU(\hat R,U^m)=\Phi\big(\cU^0(\hat R,U^m)\big)$. Then,
\begin{align}
\delta\cU=\Phi'\big(\cU^0(\hat R,U^m)\big)\delta\cU^0=\Phi'\big(\cU^0(\hat R,U^m)\big)S\T_L(\hat R)\Sigma.
\end{align}
Taking the first variation of the kinematic energy term associated with the artificial system \eqref{angvelindex} with respect to $\hat\Omega$ yields
\begin{align}
\delta\cT&=-m(\Omega^m-\hat\Omega)\T\delta\hat\Omega=-m(\Omega^m-\hat\Omega)\T(\dot\Sigma+\hat\Omega\times\Sigma)=-m\omega\T(\dot\Sigma+\hat\Omega\times\Sigma),
\end{align}
where $\omega=\Omega^m-\hat\Omega$.
Applying Lagrange-d'Alembert principle leads to
\begin{align}
&~~~~~\delta\cS+\int_{t_0}^T\tau_D\T\Sigma\di t=0\\
&\Rightarrow \int_{t_0}^T\Big\{-m\omega\T(\dot\Sigma+\hat\Omega\times\Sigma)-\Phi'\big(\cU^0(\hat R,U^m)\big)S\T_L(\hat R)\Sigma+\tau_D\T\Sigma\Big\}\di t=0 \Rightarrow\nn\\
& -m\omega\T\Sigma\big|_{t_0}^T+\int_{t_0}^T m\dot{\omega}\T\Sigma\di t=\int_{t_0}^T\Big\{m\omega\T\hat\Omega^\times+\Phi'\big(\cU^0(\hat R,U^m)\big)S\T_L(\hat R)-\tau_D\T\Big\}\Sigma\di t,\nn
\end{align}
where the first term in the left hand side vanishes, since $\Sigma(t_0)=\Sigma(T)=0$, and after replacing the dissipation term $\tau_D=D\omega$ gives the second equation in \eqref{eq:filterNoise}.
\hfill\ensuremath{\square}

\subsubsection{Stability of Filter}
Next consider the stability of the estimation scheme (filter) given by Proposition \ref{filterN}. 
The following result shows that this scheme is stable, with almost global convergence of 
the estimated states to the real states in the absence of measurement noise. 
\begin{theorem} \label{thmFilt0}
The filter presented in Proposition \ref{filterN}, with distinct positive eigenvalues for 
$K=EWE\T$, is asymptotically stable at the estimation error state 
$(Q,\omega)=(I,0)$ in the absence of measurement noise. Further, the domain of 
attraction of $(Q,\omega)=(I,0)$ is a dense open subset of $\SO\times\bR^3$. 
\end{theorem}
{\em Proof}: In the absence of measurement noise, $U^m=U=R\T E$ and therefore 
$\cU^0(\hat R,U^m)=\frac12\lan E-\hat{R}U,(E-\hat{R}U)W\ran=\lan I-Q, K\ran=\cU^0(Q)$ 
where $K=EWE\T$ and $Q=R\hat R\T$. Therefore, $\Phi(\lan I- Q, K\ran)$, is a Morse 
function on $\SO$. The stability of this filter can be shown using the 
following candidate Morse-Lyapunov function, which can be interpreted as the total energy 
function (equal in value to the Hamiltonian) corresponding to the Lagrangian \eqref{cLag}:
\begin{align}
V(\hat R,\omega,U)& =\Phi\Big(\frac{1}{2}\lan E-\hat{R}U,(E-\hat{R}U)W\ran\Big)+ \frac{m}{2}\omega \T\omega \nn \\
=& \Phi(\lan I-Q, K\ran)+ \frac{m}{2}\omega \T\omega= V(Q,\omega). \label{lyapf} 
\end{align}
Note that $V(Q,\omega)\geq 0$ and $V( Q,\omega)=0$ if and only if $(Q,
\omega)=(I,0)$. Therefore, $V(Q,\omega)$ is positive definite on 
$\SO\times\bR^3$. Using \eqref{Poisson} and \eqref{eq:filterNoise}
\begin{align}
\begin{split}
\frac{\di}{\di t}&\Phi\big(\lan I-Q, K\ran\big) =\frac{\di}{\di t}\Phi(\lan I-R\hat R\T, K\ran)\\
&= \Phi'\big(\lan I-Q, K\ran\big)\lan K, -R\Omega^\times\hat R\T+R\hat\Omega^\times\hat R\T \ran \\
&=\Phi'\big(\lan I-Q, K\ran\big)\Big(\frac12 \lan \hat R\T K R- R\T K\hat R,\omega^\times\ran\Big) \\
&= -\Phi'\big(\lan I-Q, K\ran\big) S_L\T (\hat R) \omega.
\end{split}\label{Phidot} 
\end{align}
Therefore, the time derivative of the candidate Morse-Lyapunov function is
\begin{align}
\dot{V}(Q,\omega)&=\frac{\di}{\di t}\Phi(\lan I- Q, K\ran)+m\omega\T\dot{\omega}\nn\\
&=\omega\T\bigg(-\Phi'\big(\cU^0(Q)\big) S_L (\hat R)-m\hat{\Omega}\times\omega+\Phi'\big(\cU^0(Q)\big)S_L(\hat{R})-D\omega\bigg).
\end{align}
Noting that $m\omega\T(\hat{\Omega}\times\omega)=0$, this yields
\begin{align}
\dot{V}(Q,\omega)=-\omega\T D\omega. \label{dlyapf}
\end{align}
Hence, the derivative of the Morse-Lyapunov function is negative semi-definite.

Note that the error dynamics for the attitude estimate error is given by 
\be \dot Q = Q \psi^\times \, \mbox{ where }\, \psi=\hat R\omega, \label{Qdot} \ee
while the error dynamics for the angular velocity estimate error $\omega$ is given by the 
second of equations \eqref{eq:filterNoise}. Therefore, the error dynamics for $(Q,\omega)$ is 
non-autonomous, since they depend explicitly on $(\hat R,\hat\Omega)$. Considering \eqref{lyapf} 
and \eqref{dlyapf} and applying Theorem 8.4 in \cite{khal}, one can conclude that $\omega\T D\omega
\rightarrow 0$ as $t\rightarrow \infty$, which consequently implies $\omega\rightarrow0$. 
Thus, the positive limit set for this system is contained in
\begin{align}
\cE = \dot{V}^{-1}(0)=\big\{(Q,\omega)\in\SO\times\so:\omega\equiv0\big\}. 
\label{dotV0}
\end{align}
Substituting $\omega\equiv 0$ in the filter equations \eqref{eq:filterNoise}, we obtain the 
positive limit set where $\dot V\equiv 0$ (or $\omega\equiv 0$) as the set
\begin{align}
\begin{split}
\mathscr{I} &= \big\{(Q,\omega)\in\SO\times\bR^3: S_K(Q)\equiv 0, 
\omega\equiv0\big\} \\
&= \big\{(Q,\omega)\in\SO\times\bR^3: Q \in C_Q,\ \omega\equiv0\big\}.
\end{split} \label{invset}
\end{align}
Therefore, in the absence of measurement errors, all the solutions of this filter converge 
asymptotically to the set $\mathscr{I}$. Thus, the attitude estimate error converges to the set 
of critical points of  $\lan I-Q, K\ran$ in this intersection. 
The unique global minimum of this function is at $(Q,\omega)=(I,0)$ (Lemma 
\ref{lem2}, see also \cite{san06,AKS2008GNC}), so this estimation error is 
asymptotically stable.

Now consider the set
\be \mathscr{C}= \mathscr{I}\setminus (I,0),  \label{othereqb} \ee
which consists of all stationary states that the estimation errors may converge to, besides 
the desired estimation error state $(I,0)$. Note that all states in the stable manifold of a 
stationary state in $\mathscr{C}$ will converge to this stationary state. 
From the properties of the critical points $Q_i\in C_Q\setminus (I)$ of $\Phi(\lan K, I- Q\ran)$ 
given in Lemma \ref{lem2}, we see that the stationary points in $\mathscr{I}\setminus (I,0)=
\big\{ (Q_i, 0) : Q_i\in C_Q\setminus (I)\big\}$ have stable manifolds whose dimensions
depend on the index of $Q_i$. Since the angular velocity estimate error $\omega$ 
converges globally to the zero vector, the dimension of the stable manifold $\cM^S_i$
of $(Q_i, 0)\in\SO\times\bR^3$ is 
\be \dim (\cM^S_i) = 3+(3-\,\mbox{index of } Q_i)= 6- \,\mbox{index of } Q_i. \label{dimStabM} \ee
Therefore, the stable manifolds of $(Q,\omega)=(Q_i,0)$ are three-dimensional, 
four-dimensional, or five-dimensional, depending on the index of $Q_i\in C_Q\setminus (I)$ 
according to \eqref{dimStabM}. Moreover, the value of the Lyapunov function $V(Q,
\omega)$ is non-decreasing (increasing when $(Q,\omega)\notin\mathscr{I}$) 
for trajectories on these manifolds when going backwards in time. This 
implies that the metric distance between error states $(Q,\omega)$ along these 
trajectories on the stable manifolds $\cM^S_i$ grows with the time separation between these 
states, and this property does not depend on the choice of the metric on $\SO\times\bR^3$. 
Therefore, these stable manifolds are embedded (closed) submanifolds of $\SO\times\bR^3$ 
and so is their union. Clearly, all states starting in the complement of this union, converge 
to the stable equilibrium $(Q,\omega)=(I,0)$; therefore the domain of attraction 
of this equilibrium is
\[ \mbox{DOA}\{(I,0)\} = \SO\times\bR^3\setminus\big\{\cup_{i=1}^3 \cM^S_i\big\}, \]
which is a dense open subset of $\SO\times\bR^3$. 
\hfill\ensuremath{\square}


\subsection{Discrete-Time Variational Estimator}\label{DiscVE}
\subsubsection{Measurement Model}\label{C4S3.1}
Consider an interval of time $[t_0, T]\in\bR^+$ separated into $N$ equal-length
subintervals $[t_i,t_{i+1}]$ for $i=0,1,\ldots,N$, with $t_N=T$ and $t_{i+1}-t_i=h$ 
is the time step size. Let $(\hat R_i,\hat\Omega_i)\in\SO\times\bR^3$ denote the discrete 
state estimate at time $t_i$, such that $(\hat R_i,\hat\Omega_i)\approx (\hat R(t_i),
\hat\Omega(t_i))$ where $(\hat R(t),\hat\Omega(t))$ is the exact solution of the 
continuous-time filter at time $t\in [t_0, T]$. Rigid body attitude is determined from $k\in\mathbb{N}$ known inertial vectors measured in 
a coordinate frame fixed to the rigid body. Let these vectors at time $t_i$ be denoted as $u_{j_i}^m$ for 
$j=1,2,\ldots,k$, in the body-fixed frame. The assumption that $k\ge 2$ is necessary for 
instantaneous three-dimensional attitude determination. When $k=2$, the cross product 
of the two measured vectors is considered as a third measurement for applying the attitude 
estimation scheme. Denote the corresponding known inertial vectors as seen from the rigid body at time $t_i$ as $e_{j_i}$, and 
let the true vectors in the body frame at the same time instance be denoted $u_{j_i}=R_i \T e_{j_i}$, where $R_i$ is the rotation 
matrix from the body frame to the inertial frame at time $t_i$. This rotation matrix provides a coordinate-free,  
global and unique description of the attitude of the rigid body. Define the matrix composed 
of all $k$ measured vectors expressed in the body-fixed frame at time $t_i$ as column vectors,
\begin{align}
U^m_i&= [u_{1_i}^m\ u_{2_i}^m\ u_{1_i}^m\times u_{2_i}^m]  \mbox{ when }\, k=2,\, \mbox{ and }\nn\\ 
U^m_i&=[u_{1_i}^m\ u_{2_i}^m\ ...\ u_{k_i}^m]\in \mathbb{R}^{3\times k} \mbox{ when } k>2,\label{Umexp}
\end{align}
and the corresponding matrix of all these vectors at the same time instance expressed in the inertial frame as
\begin{align}
E_i&=[e_{1_i}\ e_{2_i}\ e_{1_i}\times e_{2_i}]  \mbox{ when }\, k=2,\, \mbox{ and }\nn\\ 
E_i&=[e_{1_i}\ e_{2_i}\ ...\  e_{k_i}]\in \mathbb{R}^{3\times k}  \mbox{ when } k>2.\label{Eexp}
\end{align}

Note that the matrix of the actual body vectors $u_{j_i}$ corresponding to the inertial vectors 
$e_{j_i}$, is given by
\begin{align}
U_i&=R_i \T E_i= [u_{1_i}\ u_{2_i}\ u_{1_i}\times u_{2_i}]  \mbox{ when }\, k=2,\, \mbox{ and }\nn\\ 
U_i&=R_i \T E_i=[u_{1_i}\ u_{2_i}\ ...\ u_{k_i}]\in \mathbb{R}^{3\times k} \mbox{ when } k>2.\nn
\end{align}

\subsubsection{Discrete-Time Lagrangian}
Let $\Omega_i\in\mathbb{R}^3$ be the angular velocity of the rigid body at time $t_i$ expressed in the 
body-fixed frame. The attitude kinematics is given by Poisson's equation:
\begin{align}
\dot{R_i}=R_i\Omega_i^\times.
\label{Disc_Poisson}
\end{align}\par
The term encapsulating the ``energy" in the 
attitude estimate error is discretized as follows:
\begin{align}
\cU (\hat R_i,U^m_i)&=\Phi\Big(\cU^0(\hat R_i,U^m_i)\Big)\nn\\
&=\Phi \Big(\frac{1}{2}\lan E_i-\hat R_i U^m_i,(E_i-\hat R_i U^m_i)W_i\ran \Big), 
\label{dattindex}
\end{align}
where
\begin{align}
\lan A_1,A_2\ran=\tr(A_1 \T A_2)\label{eq:tr_def}
\end{align}
denotes the trace inner product on $\mathbb{R}^{m\times n}$, $W_i=\diag(w_{j_i})$ is the positive diagonal matrix of the weight factors for the measured 
directions at time $t_i$ which satisfies the eigendecomposition condition in \cite{Automatica}, and $\Phi: [0,\infty)\mapsto[0,\infty)$ is a $C^2$ function that satisfies $\Phi(0)=0$ and 
$\Phi'(x)>0$ for all $x\in[0,\infty)$. Furthermore, $\Phi'(\cdot)\leq\alpha(\cdot)$ where 
$\alpha(\cdot)$ is a Class-$\mathcal{K}$ function.\par
The term containing the ``energy" in the angular velocity estimate error is discretized as
\be \cT (\hat\Omega_i,\Omega^m_i)=\frac{m}{2}(\Omega^m_i-\hat\Omega_i) \T(\Omega^m_i-\hat\Omega_i), \label{dangvelindex} \ee 
where $m$ is a positive scalar. 

As with the continuous-time state estimation process in \cite{Automatica}, one can express 
these ``energy" terms in the state estimate errors for the case that perfect measurements 
(with no measurement noise) are available. In this case, these ``energy" terms can be 
expressed in terms of the state estimate errors $Q_i= R_i \hat R_i\T$ and 
$\omega_i= \Omega_i-\hat\Omega_i$ as follows:
\begin{align}
& \cU (Q_i)= \Phi \Big(\frac{1}{2}\lan E_i - Q_i\T E_i,(E_i - Q_i\T E_i)W_i\ran \Big)=
\Phi \big( \lan I-Q_i, K_i\ran \big)\,\nn\\
& \mbox{ where }\, K_i= E_i W_i E_i\T,  \mbox{  and }\,\cT (\omega_i)= \frac{m}{2} \omega_i\T\omega_i\, \mbox{ where }
m>0. 
\label{discUandT} 
\end{align}\par
The weights in $W_i$ can be chosen such that $K_i$ is always positive definite with distinct 
(perhaps constant) eigenvalues, as in the continuous-time filter. Using these ``energy" terms in 
the state estimate errors, the discrete-time Lagrangian can be expressed as:
\begin{align}
\cL (Q_i,\omega_i)&= \cT (\omega_i)- \cU ( Q_i)
= \frac{m}{2} \omega_i\T\omega_i- \Phi\big( \lan I-Q_i, K_i\ran \big).
\label{discLag}
\end{align}

In order to numerically implement the filtering scheme introduced in this paper, a
discrete-time version is obtained to estimate the attitude states from vector 
measurements and angular velocity measurements. It is assumed that these measurements 
are obtained in discrete-time at a sufficiently high but constant sample rate. In this section, 
a discrete-time version of the filter introduced in Proposition \ref{filterN} is obtained in the 
form of a Lie group variational integrator (LGVI). A variational integrator works by discretizing 
the (continuous-time) variational mechanics principle that leads to the equations of motion, 
rather than discretizing the equations of motion directly. A good background on variational 
integrators is given in the excellent treatise~\cite{marswest}. The correspondence between 
variational integrators and symplectic integrators (for conservative systems) is given in 
the book~\cite{haluwa}. Lie group variational integrators are variational integrators for 
mechanical systems whose configuration spaces are Lie groups, like rigid body systems. 
In addition to maintaining properties arising from the variational principles of mechanics, 
like energy and momenta, LGVI schemes also maintain the geometry of the Lie group 
that is the configuration space of the system~\cite{mclele}. 

\subsubsection{Discrete-Time Lagrangian}
As a first step to obtaining the LGVI that discretizes the filter in Proposition \ref{filterN}, a 
discrete-time counterpart of the (continuous-time) Lagrangian expressed in \eqref{cLag} 
is obtained. Consider an interval of time $[t_0, T]\in\bR^+$ separated into $N$ equal-length
subintervals $[t_i,t_{i+1}]$ for $i=0,1,\ldots,N$, with $t_N=T$ and $t_{i+1}-t_i=h$ 
is the time step size. Let $(\hat R_i,\hat\Omega_i)\in\SO\times\bR^3$ denote the discrete 
state estimate at time $t_i$, such that $(\hat R_i,\hat\Omega_i)\approx (\hat R(t_i),
\hat\Omega(t_i))$ where $(\hat R(t),\hat\Omega(t))$ is the exact solution of the 
continuous-time filter at time $t\in [t_0, T]$. 

It is assumed that $k\ge 2$ known inertial vectors are measured in the body frame, as in 
Proposition \ref{filterN}. The term encapsulating the ``energy" in the 
attitude estimate error, given by \eqref{attindex}, is discretized as follows:
\begin{align}
\cU (\hat R_i,U^m_i)=\Phi \Big(\frac{1}{2}\lan E_i-\hat R_i U^m_i,(E_i-\hat R_i U^m_i)W_i\ran \Big), 
\label{dattindex}
\end{align}
where  $E_i\in\bR^{3\times k}$ is the set of inertial vectors and $U^m_i\in\bR^{3\times k}$ is the 
corresponding set of measured body vectors observed at time $t_i$,  and $W_i$ 
is the corresponding diagonal matrix of weight factors.
The term containing the ``energy" in the angular velocity estimate error is discretized as
\be \cT (\hat\Omega_i,\Omega^m_i)=\frac{m}{2}(\Omega^m_i-\hat\Omega_i) \T(\Omega^m_i-\hat\Omega_i), \label{dangvelindex} \ee 
which is the discrete-time version of equation \eqref{angvelindex}. 

As with the continuous-time state estimation process in Sections 2 and 3, one can express 
these ``energy" terms in the state estimate errors for the case that perfect measurements 
(with no measurement noise) are available. In this case, these ``energy" terms can be 
expressed in terms of the state estimate errors $Q_i= R_i \hat R_i\T$ and 
$\omega_i= \Omega_i-\hat\Omega_i$ as follows:
\begin{align}
\begin{split}
& \cU (Q_i)= \Phi \Big(\frac{1}{2}\lan E_i - Q_i\T E_i,(E_i - Q_i\T E_i)W_i\ran \Big)=
\Phi \big( \lan I-Q_i, K_i\ran \big)\, \\
& \mbox{ where }\, K_i= E_i W_i E_i\T, \mbox{and }\,\cT (\omega_i)= \frac{m}{2} \omega_i\T\omega_i\, \mbox{ where }
m>0. 
\end{split}\label{discUandT} 
\end{align}
The weights in $W_i$ can be chosen such that $K_i$ is always positive definite with distinct 
(perhaps constant) eigenvalues, as in the continuous-time filter given by Proposition 
\ref{filterN}. Using these ``energy" terms in the state estimate errors, the 
discrete-time Lagrangian can be expressed as: 
\begin{align}
\cL (Q_i,\omega_i)= \cT (\omega_i)- \cU ( Q_i)
= \frac{m}{2} \omega_i\T\omega_i- \Phi\big( \lan I-Q_i, K_i\ran \big).
\label{discLag}
\end{align}

\subsubsection{Discrete-Time Attitude State Estimation Based on the Discrete 
Lagrange-d'Alembert Principle}

The following statement gives the discrete-time filter equations, in the form of a Lie group 
variational integrator, corresponding to the continuous-time filter given by Proposition 
\ref{filterN}.
\begin{proposition} \label{discfilter}
Let two or more vector measurements be available, along with angular velocity 
measurements in discrete-time, at time intervals of length $h$. Further, let the weight 
matrix $W_i$ for the set of vector measurements $E_i$ be chosen such that $K_i=E_i 
W_i E_i\T$ satisfies the eigendecomposition condition \eqref{Kdec} of Lemma \ref{lem1}. 
A discrete-time filter that approximates the continuous-time filter of Proposition \ref{filterN} 
to first order in $h$ is 
~~~~~~~~~~~~~~~~~~~~~~~~~~~~~~~~~~~~~~~~~~~~~~~~~~
\begin{align}
&\hat R_{i+1}=\hat R_i\exp(h\hat\Omega_i^\times)=\hat{R_i}\exp\big(h(\Omega_i^m-\omega_i)^\times\big),\vspace*{4mm}\label{1stDisFil_Rhat}\\
&m\omega_{i+1}=\exp(-h \hat\Omega_{i+1}^\times)\Big\{(m I_{3\times3}-hD)
\omega_i+h\Phi'\big(\cU^0(\hat R_{i+1},U^m_{i+1})\big)S_{L_{i+1}}(\hat R_{i+1})\Big\},\vspace*{4mm}\label{1stDisFil_omega}\\
&\hat\Omega_i=\Omega_i^m-\omega_i,
\label{1stDisFil_Omegahat}
\end{align}
where $S_{L_i}(\hat R_i)=\mrm{vex}(L_i\T\hat R_i-\hat R_i\T L_i)\in\bR^3$, $L_i=E_i W_i(U^m_i)\T\in\mathbb{R}^{3\times3}$ and $(\hat R_0,\hat\Omega_0)\in\SO\times\bR^3$ are initial estimated states.
\end{proposition}
{\em Proof}: The action functional in expression \eqref{eq:J6} is replaced by the 
discrete-time action sum as follows:
\begin{align}
&\cS_d (\cL (\hat R_i,U^m_i,\hat\Omega_i,\Omega^m_i))= h \sum_{i=0}^N \Big\{ 
\frac{m}{2}(\Omega^m_i-\hat\Omega_i)\T(\Omega^m_i-\hat\Omega_i)
- \Phi\big(\cU^0(\hat R_i,U^m_i)\big) \Big\}.
\end{align}
Discretize the kinematics of the attitude estimate as
\begin{align}
\hat R_{i+1}=\hat R_i \exp( h\hat\Omega_i^\times), \label{diskin}
\end{align}
and consider a first variation in the discrete attitude estimate, $R_i$, of the form
\be \delta \hat R_i=\hat R_i\Sigma_i^\times,  \label{disVar} \ee
where $\Sigma_i\in\bR^3$ gives a variation vector for the discrete attitude estimate.
For fixed end-point variations, we have $\Sigma_0=\Sigma_N=0$. 
Further, a first order approximation is to assume that $\hat\Omega_i^\times$ and 
$\delta\hat\Omega_i^\times$ commute. With this assumption, taking the first variation of the 
discrete kinematics \eqref{diskin} and substituting from \eqref{disVar} gives:
\begin{align}
\begin{split}
\delta \hat R_{i+1} =&\delta \hat R_i \exp(h\hat\Omega_i^\times)+
\hat R_i\delta\big(\exp(h\hat\Omega_i^\times)\big) \\
=&\hat R_i\Sigma_i^\times\exp(h\hat\Omega_i^\times)+ h \hat R_i \exp(h\hat\Omega_i^\times)
\delta\hat\Omega_i^\times=\hat R_{i+1}\Sigma_{i+1}^\times. 
\end{split}\label{disLagdAlm}
\end{align}
Equation \eqref{disLagdAlm} can be re-arranged to obtain: 
\begin{align}
h\delta\hat\Omega_i^\times&=\exp(-h\hat\Omega_i^\times)\hat R_i\T\big[\delta
\hat R_{i+1}-\hat R_i\Sigma_i^\times\exp(h\hat\Omega_i^\times)\big] \nn \\
&=\exp(-h\hat\Omega_i^\times)\hat R_i\T \hat R_{i+1}\Sigma_{i+1}^\times-
\Ad{\exp(-h\hat\Omega_i^\times)}\Sigma_i^\times \nn \\
&=\Sigma_{i+1}^\times-\Ad{\exp(-h\hat\Omega_i^\times)}\Sigma_i^\times.  \label{delR}
\end{align}
This in turn can be expressed as an equation on $\bR^3$ as follows:
\be h\delta\hat\Omega_i= \Sigma_{i+1} - \exp(-h\hat\Omega_i^\times)\Sigma_i, 
\label{delomega} \ee
since $\Ad{R}\Omega^\times=R\Omega^\times R\T=(R\Omega)^\times$. 

Applying the discrete Lagrange-d'Alembert principle~\cite{marswest}, one obtains
\begin{align}
&\delta\cS_d+ h\sum_{i=0}^{N-1}\tau_{D_i}\T\Sigma_i =0\nn\\
\Rightarrow &h\sum_{i=0}^{N-1} m(\hat\Omega_i-\Omega^m_i)\T\delta\hat\Omega_i-\Big\{\Phi'\big(\cU^0(\hat R_i,U^m_i)\big)S_{L_i}\T(\hat R_i)
-\tau_{D_i}\T\Big\}\Sigma_i=0. \label{disSd1}
\end{align} 
Substituting  \eqref{disVar} and \eqref{delomega} into equation \eqref{disSd1}, 
one obtains
\begin{align}
&\sum_{i=0}^{N-1}\Big\{m(\hat\Omega_i-\Omega^m_i)\T\big(\Sigma_{i+1}-\exp(-h\hat\Omega_i^\times)\Sigma_i\big)\nn\\
&-h\Phi'\big(\cU^0(\hat R_i,U^m_i)\big)S_{L_i}\T(\hat R_i)\Sigma_i+h\tau_{D_i}\T\Sigma_i\Big\}=0.
\label{disLagdAlObs}
\end{align}
For $0\le i<N$, the expression \eqref{disLagdAlObs} leads to the following one-step 
first-order LGVI for the discrete-time filter:
\begin{align}
&m(\Omega^m_{i+1}-\hat\Omega_{i+1})\T\exp(-h\hat\Omega_{i+1}^\times)+h\tau_{D_{i+1}}\T
\nn\\ &-h\Phi'\big(\cU^0(\hat R_{i+1},U^m_{i+1})\big)S_{L_{i+1}}\T(\hat R_{i+1})
+m(\hat\Omega_i-\Omega^m_i)\T=0\nn\\
&\Rightarrow m\exp(h\hat\Omega_{i+1}^\times) (\Omega^m_{i+1}-\hat\Omega_{i+1})=m(\Omega^m_i-\hat\Omega_i)\nn\\
&+h\Big(\Phi'\big(\cU^0(\hat R_{i+1},U^m_{i+1})\big)S_{L_{i+1}}(\hat R_{i+1})-\tau_{D_{i+1}}\Big),
\end{align}
which after substituting $\omega_i=\Omega_i^m-\hat\Omega_i$ and $\tau_{D_{i+1}}=D\omega_i$ gives the discrete-time filter presented in \eqref{1stDisFil_Rhat}-\eqref{1stDisFil_Omegahat}.
\hfill\ensuremath{\square}

Note that the filter equations \eqref{1stDisFil_Rhat}-\eqref{1stDisFil_Omegahat} given by the 
LGVI scheme are in the form 
of an implicit numerical integration scheme. The discrete kinematics equation, which is 
the equation \eqref{1stDisFil_Rhat}, is solved first. Then the angular velocity estimate 
error is updated by solving the implicit discrete dynamics equation, which is the equation \eqref{1stDisFil_omega}. The stability and convergence properties of this discrete-time 
filter are not shown here directly. Since this filter is a first-order discretization of the 
continuous-time filter in \cite{Automatica} which is almost globally asymptotically stable, its 
solution will be a first-order (in $h$) approximation to the continuous-time filter.

\subsubsection{Discrete-time First Order Butterworth Filter for Angular Velocity}
Since the proposed filter does not filter noise from angular velocity measurements, a symmetric 
linear filter in the form of a discrete first-order Butterworth filter is applied to these measurements. 
The filtered velocities are then used in place of the unfiltered $\Omega_i^m$ to enhance 
the nonlinear filter given by equations\eqref{1stDisFil_Rhat}-\eqref{1stDisFil_Omegahat}. This 
time-symmetric filter is of the form:
\begin{align}
(2+ h) \bar \Omega_{i+1}= (2-h)\bar\Omega_i+ h (\Omega^m_i + \Omega^m_{i+1}),
\end{align}
where $\bar\Omega_i$ is the filtered angular velocity at time $t_i$ for $i=0,1,...,N-1$, and 
$\bar\Omega_0=\Omega^m(t_0)$. $\bar\Omega_i$ is used in place of $\Omega^m_i$ in 
\eqref{1stDisFil_Rhat}-\eqref{1stDisFil_Omegahat}.

The stability and convergence properties of this discrete-time 
filter are not shown here directly. Since this filter is a first-order discretization of the 
continuous-time filter in Proposition \ref{filterN}, its solution will be a first-order (in $h$) 
approximation to the continuous-time filter.

\subsubsection{Explicit First-Order Estimator}
Note that the second equation in the first-order Lie group variational integrator, eq. 
\eqref{1stDisFil_omega}, is an implicit equation with respect to $\omega_{i+1}$. One needs to 
solve this equation using an iterative method like Newton-Raphson at every time step. This can 
considerably increase the computational load and runtime, making the estimator difficult to 
implement in applications requiring real-time estimation \cite{ICC2015}. In order to find a solution 
for this issue, one can use the adjoint of this LGVI, which provides an explicit first order as 
given in the following statement.
\begin{proposition}\label{ExplicitFil}
A discrete-time filter that gives an explicit first order numerical integrator for the filter presented 
in \cite{Automatica} is given by:
\begin{align}
&\hat R_{i+1}=\hat{R_i}\exp\big(h(\Omega_{i+1}^m-\omega_{i+1})^\times\big),\label{ExplicitDisFil_Rhat}\\
&\omega_{i+1}= \big(mI_{3\times3}+hD\big)^{-1}\bigg\{ \exp\big(-h
\hat\Omega_i^\times\big) m\omega_i+h\Phi'\big(\cU^0(\hat R_i,U^m_i)\big)S_{L_{i}}(\hat R_i) \bigg\},\label{ExplicitDisFil_omega}\\
&\hat\Omega_i=\Omega_i^m-\omega_i,
\label{ExplicitDisFil_Omegahat}
\end{align}
where $S_{L_i}(\hat R_i)$ is defined in Proposition \ref{discfilter} and $(\hat R_0,\hat\Omega_0)\in\SO\times\bR^3$ are initial estimated states.
\end{proposition}
\begin{proof} Let $\Xi_h$ denote the forward time map of the Lie group variational integrator given by 
equations \eqref{1stDisFil_Rhat}-\eqref{1stDisFil_Omegahat} of the filter in Proposition \ref{discfilter}. The adjoint of a 
numerical integration method whose one step forward time map is denoted $\Xi_h$, is 
defined as $\Xi_h^\star= \Xi_{-h}^{-1}$~\cite{haluwa} for the time interval $[t_i,t_{i+1}]$. In other words, the adjoint scheme is 
obtained by interchanging indices $i$ and $i+1$ and replacing $h$ with $-h$ in the original 
scheme. This adjoint flow can be 
constructed from \eqref{1stDisFil_Rhat}-\eqref{1stDisFil_Omegahat} as
\begin{align}
&\hat R_{i+1}=\hat{R}_i\exp\big(h(\Omega_{i+1}^m-\omega_{i+1})^\times\big),\label{adjntflow1_Rhat}\\
&(m I_{3\times3}+hD)\omega_{i+1}=\exp(-h\hat\Omega_i^\times)m\omega_i+h\Phi'(\cU^0(\hat R_i,U^m_i))S_{L_{i}}(\hat R_i).\label{adjntflow1_omega}
\end{align}\par
The filter equations \eqref{ExplicitDisFil_Rhat}-\eqref{ExplicitDisFil_Omegahat} are easily concluded from \eqref{adjntflow1_Rhat}-\eqref{adjntflow1_omega}.
\end{proof}

One can easily verify that the trajectories of both the implicit and the explicit first order filters are 
very similar using numerical simulations. However, the explicit filter is computationally simpler 
and faster, which makes it more suited for real-time implmentation. A second order discretization 
of this variational estimator is presented next. 

\subsubsection{Symmetric Numerical Integrator as Discrete-Time Filter}
A symmetric numerical scheme is presented here as a higher-order discretization of the 
continuous-time filter. Symmetric numerical integrators have discrete flows that are time 
reversible, i.e., the  composition of the one step forward time map with the one step 
backward time map is the identity map. Symmetric schemes have many useful 
properties (e.g., easier error analysis), as detailed in chapter 5 of \cite{haluwa}. 
The symmetric scheme presented here is obtained by composing the first-order LGVI scheme 
from Proposition \ref{discfilter}, with its adjoint defined in Proposition \ref{ExplicitFil}. Clearly, a 
symmetric integration scheme is self-adjoint by definition. The following 
statement gives the discrete-time attitude and angular velocity filter that is obtained in the 
form of a symmetric integrator using the above-mentioned composition.

\begin{proposition}\label{dis2ndfilter}
A discrete-time filter that gives a second order numerical integrator for the filter in 
Proposition \ref{discfilter} is given as follows:
\begin{align}
&\hat R_{i+1}=\hat{R_i}\exp\big(h(\Omega_{i+\frac12}^m-\omega_{i+\frac12})^\times\big),\label{2ndDisFil_Rhat}\\
&m\omega_{i+1}=\exp(-\frac{h}{2} \hat\Omega_{i+1}^\times)\Big\{(m I_{3\times3}-\frac{h}{2}D)
\omega_{i+\frac12}+\frac{h}{2}\Phi'\big(\cU^0(\hat R_{i+1},U^m_{i+1})\big)S_{L_{i+1}}(\hat R_{i+1})\Big\},
\label{2ndDisFil_omega}\\
&\hat\Omega_i=\Omega_i^m-\omega_i,\;\;\Omega_{i+\frac12}^m=\frac12(\Omega_i^m+\Omega_{i+1}^m),
\label{2ndDisFil_Omegahat}
\end{align}
where 
\begin{align}
\omega_{i+\frac12}= \big(&mI_{3\times3}+\frac{h}{2}D\big)^{-1}\bigg\{ \exp\big(-\frac{h}{2}
\hat\Omega_i^\times\big) m\omega_i+\frac{h}{2} \Phi'\big(\cU^0(\hat R_i,U^m_i)\big)S_{L_{i}}(\hat R_i) \bigg\} \label {iAV}
\end{align}
is a discrete-time approximation to the angular velocity at time $t_{i+\frac12}:=t_i+\frac{h}{2}$. 
\end{proposition}
\begin{proof}
As described in chapter 2 of \cite{haluwa}, a second-order symmetric integrator is obtained by 
composing the flow of this first-order LGVI with its adjoint to obtain the forward time map:
\be \Psi_h := \Xi_{h/2}\circ \Xi^\star_{h/2}. \label{comp2nd} \ee

This composition method is referred to as the {\em Strang splitting} in \cite{haluwa}. The 
flow $\Xi^\star_{h/2}= \Xi^{-1}_{-h/2}$, for the time interval $[t_i,t_i+\frac{h}{2}]$, can be 
constructed from \eqref{1stDisFil_Rhat}-\eqref{1stDisFil_Omegahat} as
\begin{align}
&\hat R_{i+\frac12}=\hat{R}_i\exp\big(\frac{h}{2}(\Omega_{i+\frac12}^m-\omega_{i+\frac12})^\times\big),\label{adjntflow_R}\\
&(m I_{3\times3}+\frac{h}{2}D)\omega_{i+\frac12}=\exp(-\frac{h}{2} \hat\Omega_i^\times)m\omega_i+\frac{h}{2} \Phi'(\cU^0(\hat R_i,U^m_i))S_{L_{i}}(\hat R_i).\label{adjntflow_omega}
\end{align}

Note that equations \eqref{adjntflow_R} and \eqref{adjntflow_omega} give an {\em explicit integrator}, where \eqref{adjntflow_omega} can be solved first to obtain $\omega_{i+\frac12}$, following which \eqref{adjntflow_R} can be used to solve for $\hat R_{i+\frac12}$. Besides, $\Omega_{i+\frac12}^m$ could be approximated as the average of $\Omega_i^m$ and $\Omega_{i+1}^m$.
The flow $\Xi_{h/2}$ is easily obtained from \eqref{1stDisFil_Rhat}-\eqref{1stDisFil_Omegahat} as follows:
\begin{align}
&\hat R_{i+1}=\hat{R}_{i+\frac12}\exp\big(\frac{h}{2}(\Omega_{i+\frac12}^m-\omega_{i+\frac12})^\times\big),\label{origFlow_Rhat}\\
&m\omega_{i+1}=\exp(-\frac{h}{2} \hat\Omega_{i+1}^\times)\Big\{(m I_{3\times3}-\frac{h}{2}D)
\omega_{i+\frac12}+\frac{h}{2}\Phi'\big(\cU^0(\hat R_{i+1},U^m_{i+1})\big)S_{L_{i+1}}(\hat R_{i+1})\Big\}.\label{origFlow_omega}
\end{align}\par
Composing the discrete-time flows given by \eqref{adjntflow_R}-\eqref{adjntflow_omega} and \eqref{origFlow_Rhat}-\eqref{origFlow_omega}, in the 
order specified by \eqref{comp2nd}, gives rise to the one-step forward time map given as in \eqref{2ndDisFil_Rhat}-\eqref{2ndDisFil_Omegahat} and \eqref{iAV}. The overal integration scheme given by \eqref{2ndDisFil_Rhat}-\eqref{iAV} is implicit because \eqref{origFlow_omega} is implicit.
\end{proof}


\subsection{Numerical Simulations}\label{Sec4}
\subsubsection{First-order variational integrator}\label{1stSim}
This section presents numerical simulation results of the discrete time estimator presented in 
Section \ref{DiscVE}, which is a first order Lie group variational integrator. The estimator is simulated 
over a time interval of $T=300 \mbox{ s}$, with a time stepsize of $h=0.01 \mbox{ s}$. The 
rigid body is assumed to have an initial attitude and angular velocity given by,
\[ R_0=\expm_{\SO}\bigg(\Big(\frac{\pi}{4}\times[\frac{3}{7}\ \frac{6}{7}\ \frac{2}{7}]\T\Big)^\times\bigg), \]
\[\mbox{and } \Omega_0=\frac{\pi}{60}\times[-2.1\ 1.2\ -1.1]\T\mbox{ rad/s}.  \] The inertia scalar gain is $m=100$ and the dissipation matrix is selected as the following positive definite matrix:
\[D=\diag\big([12\ 13\ 14]\T\big).\]
$\Phi(\cdot)$ could be any $C^2$ function with the properties described in Section 2, but is 
selected to be $\Phi(x)=x$ here. $W$ is selected based on the measured set of vectors $E$ at 
each instant, such that it satisfies the conditions in 
Lemma \ref{lem1}. The initial estimated states have the following initial estimation errors:
\begin{align}
Q_0&=\expm_{\SO}\bigg(\Big(\frac{\pi}{2.5}\times[\frac{3}{7}\ \frac{6}{7}\ \frac{2}{7}]\T\Big)^\times\bigg), \nn\\
\mbox{and } \omega_0&=[0.001\ 0.002\ -0.003]\T\mbox{ rad/s}.
\end{align}
We assume that there are at most 9 inertially known directions which are being measured 
by the sensors fixed to the rigid body at a constant sample rate. The number of observed directions is taken to be variable over different time intervals. The dynamics equations produce the true states of the rigid body, assuming a sinusoidal force is applied to it. These true states are used to simulate the observed directions in the body-fixed frame, as well as the comparison between true and estimated states. Bounded zero mean noises are considered to be added to the true quantities to generate each measured component. A summation of 
three sinusoidal matrix functions is added to the matrix $U=R\T E$, to generate a measured $U^m$ 
with measurement noise. The frequency of the noise signals are 1, 10 and 100 Hz, with different phases and amplitudes up to $2.4^\circ$, based on coarse attitude sensors like sun sensors and magnetometers. Similarly, two sinusoidal noise signals of 10 Hz and 200 Hz frequencies are added 
to $\Omega$ to form the measured $\Omega^m$. These signals also have different phases and their magnitude is up to $0.97^\circ/s$, which is close to the real noise levels for coarse rate gyros. In order to integrate the implicit set of equations in \eqref{1stDisFil_Rhat}-\eqref{1stDisFil_Omegahat} numerically, the first equation is solved at each sampling step, then the result for $R_{i+1}$ is substituted in the second one. Using the Newton-Raphson method, the resulting equation is solved with respect to $\omega_{i+1}$ iteratively. The root of this nonlinear equation with a specific accuracy along with the $\hat R_{i+1}$ is used for the next sampling time instant. This process is repeated to the end of the simulation time.
Using the aforementioned quantities and the integration method, the simulation is carried out. The principal angle $\phi$ corresponding to the rigid body's attitude estimation error $Q$ is depicted in 
Fig. \ref{fig1}. Components of the estimation error $\omega$ in the rigid body's angular velocity 
are shown in Fig. \ref{fig2}. All the estimation errors are seen to converge to a neighborhood 
of $(Q,\omega)=(I,0)$, where the size of this neighborhood depends on the bounds of the 
measurement noise.

\begin{figure}
\begin{center}
\includegraphics[height=3.2in]{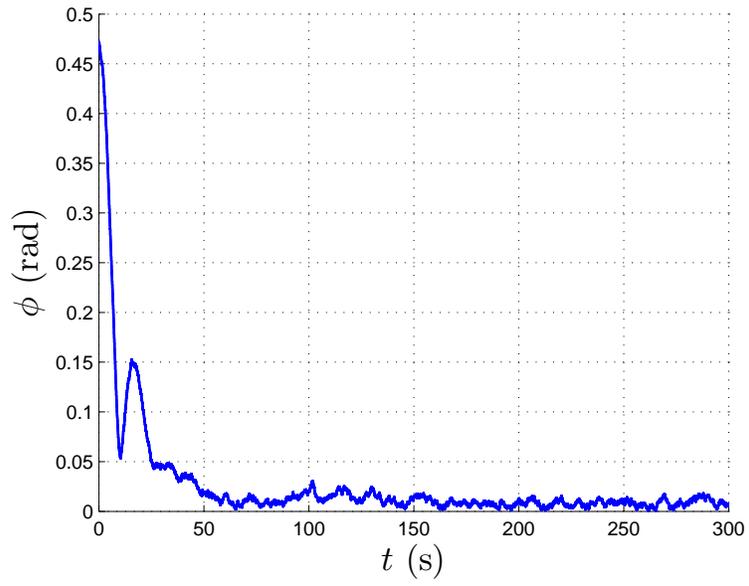}    
\caption{Principal Angle of the Attitude Estimation Error}  
\label{fig1}                                 
\end{center}                                 
\end{figure}

\begin{figure}
\begin{center}
\includegraphics[height=3.2in]{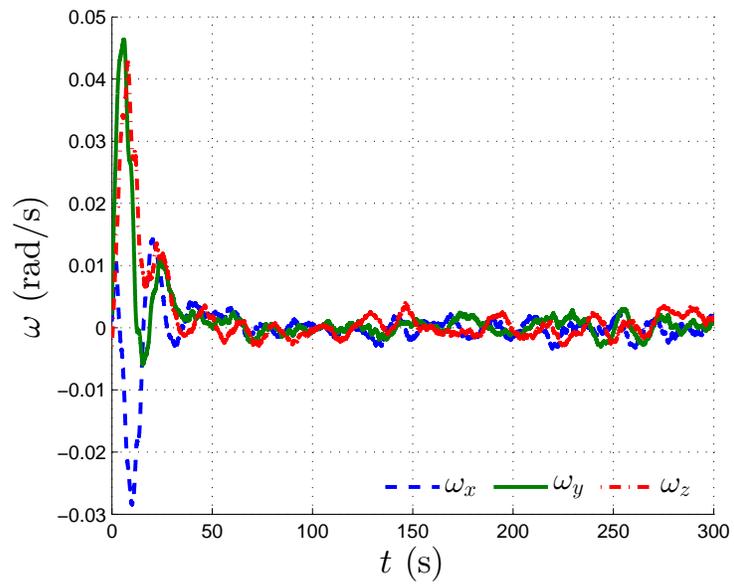}    
\caption{Angular Velocity Estimation Error}  
\label{fig2}                                 
\end{center}                                 
\end{figure}

\subsubsection{Comparison between the First-Order and Second-Order Filters}
A set of comparisons between performances of the first- and second-order filters are presented
next. The same initial conditions and parameters as in Subsection \ref{1stSim} are used for both first 
order and second order filters. The noise type and levels are also identical to those introduced 
in Subsection \ref{1stSim}. The simulations are carried out for a simulated duration of $T=300$s and 
for three different time stepsizes, namely $h=0.005$s, $h=0.01$s and $h=0.05$s. This shows the 
effect of the discretization time stepsize on each filter's convergence behavior. In order to integrate the implicit set of equations in \eqref{2ndDisFil_Rhat}-\eqref{2ndDisFil_Omegahat} numerically, $\omega_{i+\frac12}$ is substituted from \eqref{iAV}, the first 
equation is solved at each sampling step, then the result for $\hat R_{i+1}$ is substituted in the 
second one. Using the Newton-Raphson method, the resulting implicit equation is solved with 
respect to $\omega_{i+1}$ iteratively to a set tolerance. The root of this nonlinear equation 
along with $\hat R_{i+1}$ is used for the next sampling time instant. This process is repeated to 
the end of the simulated duration.\par
The principal 
angles of attitude estimate errors are depicted in Fig.~\ref{fig4}, Fig.~\ref{fig6} and Fig.~\ref{fig8}. 
In these figures, $\phi_1$ and $\phi_2$ denote the principal angle of attitude estimate error for 
first and second order observers, respectively. It can be observed that the transient response of 
the second order filter has less oscillations when compared with the first order integrator. Besides, 
the higher order estimator has a smoother behavior in the steady state.\par
In order to compare the convergence of angular velocity estimate errors of the two discrete filters, 
the norm of each vector is calculated as shown in Fig.~\ref{fig5}, Fig.~\ref{fig7} and Fig.~\ref{fig9} 
corresponding to $h=0.005$s, $h=0.01$s and $h=0.05$s, respectively. In these figures, the norm 
of the first and second order filter's angular velocity estimate errors are denoted by $\|\omega_1\|$ 
and $\|\omega_2\|$, respectively. The first order filter always has a higher overshoot and more 
oscillations in the transient and steady state phases. Moreover, the second order LGVI appears 
to converges faster to the steady state. This second order filter also shows a more robust behavior 
throughout the simulation.\par
Comparing these three pairs of figures (for different time stepsizes), one can notice that with 
increasing time stepsize $h$, the difference between the behavior of first order and second order 
integrators increases, as is expected.

\begin{figure}
\begin{center}
\includegraphics[height=3.2in]{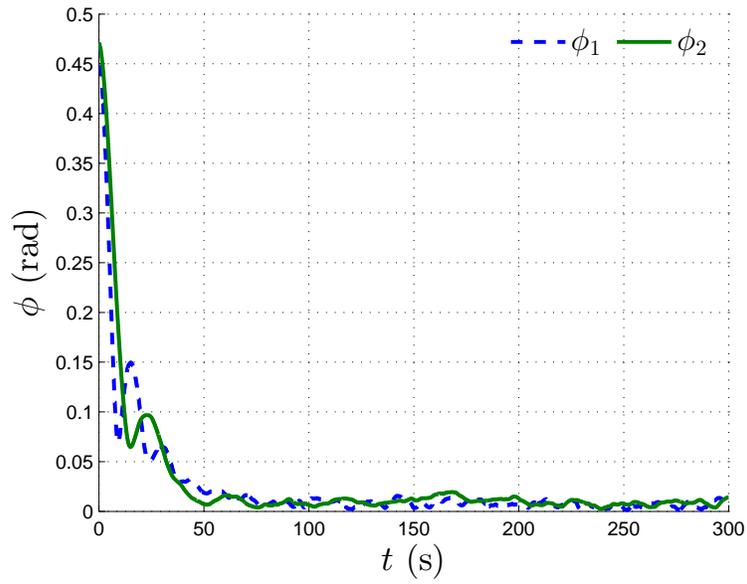}    
\caption{Principal Angles of the Attitude Estimate Errors for $h=0.005$s}  
\label{fig4}                                 
\end{center}                                 
\end{figure}

\begin{figure}
\begin{center}
\includegraphics[height=3.2in]{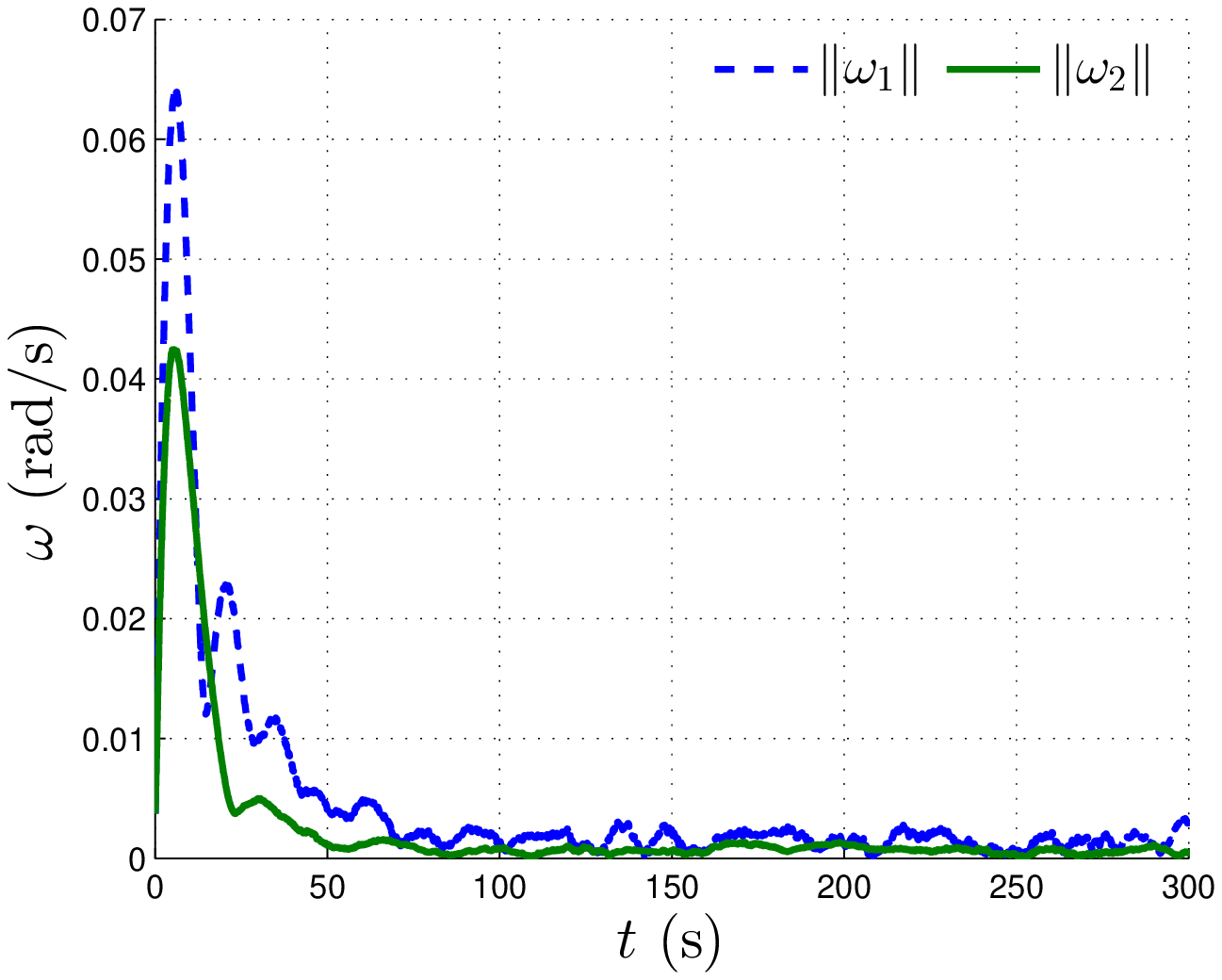}    
\caption{Norms of Angular Velocity Estimate Errors for $h=0.005$s}  
\label{fig5}                                 
\end{center}                                 
\end{figure}

\begin{figure}
\begin{center}
\includegraphics[height=3.2in]{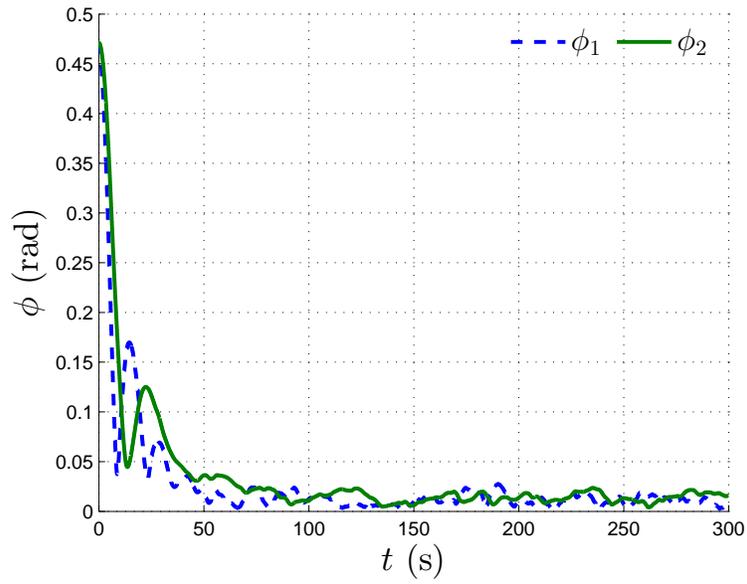}    
\caption{Principal Angles of the Attitude Estimate Errors for $h=0.01$s}  
\label{fig6}                                 
\end{center}                                 
\end{figure}

\begin{figure}
\begin{center}
\includegraphics[height=3.2in]{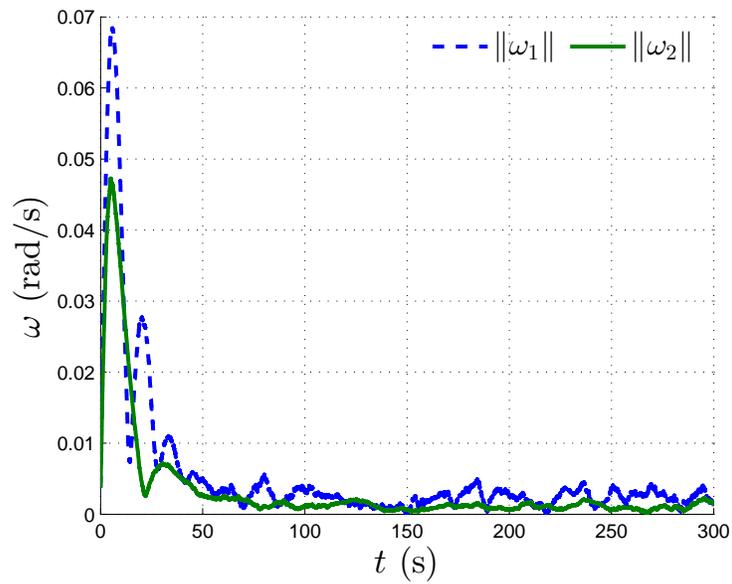}    
\caption{Norms of Angular Velocity Estimate Errors for $h=0.01$s}  
\label{fig7}                                 
\end{center}                                 
\end{figure}

\begin{figure}
\begin{center}
\includegraphics[height=3.2in]{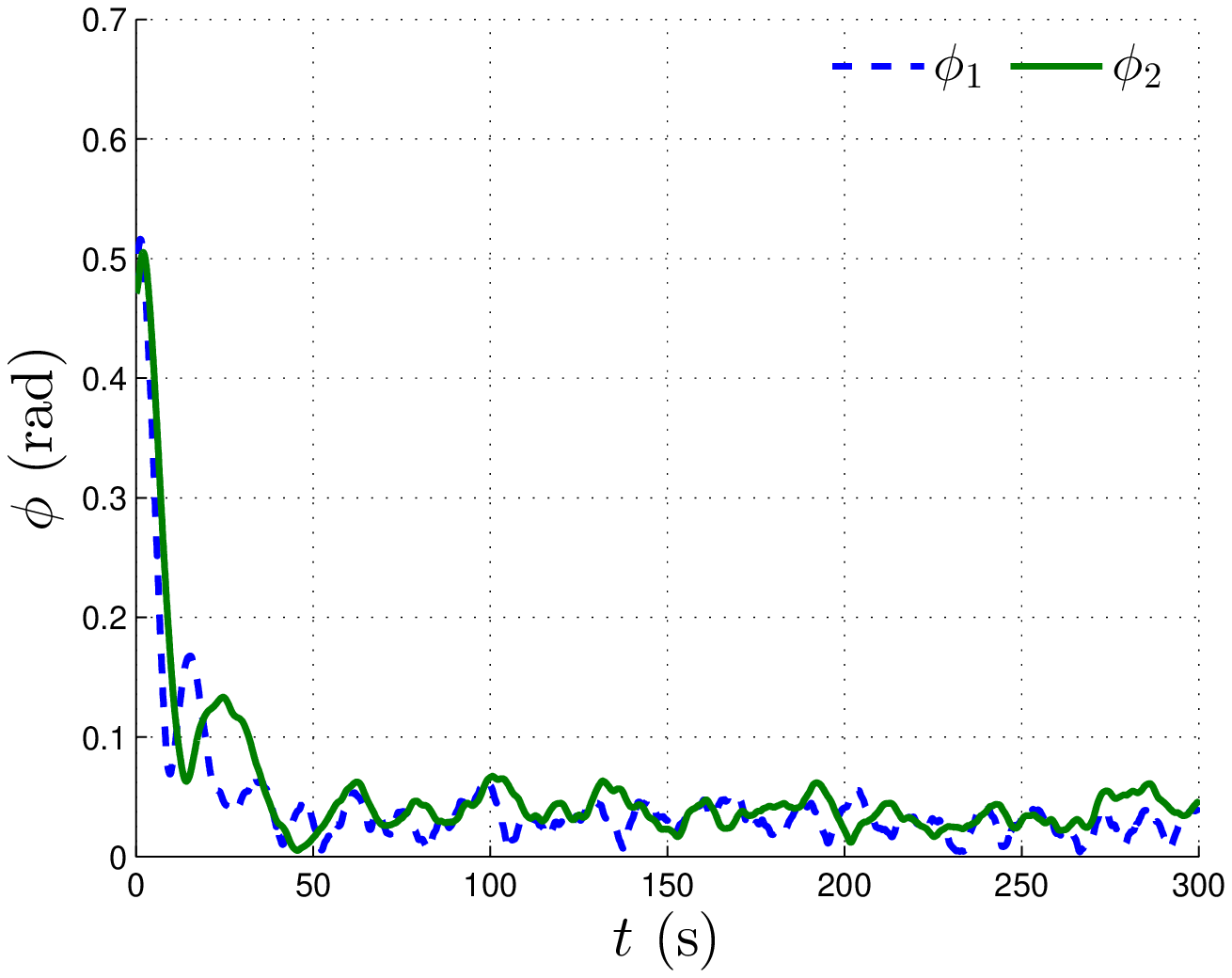}    
\caption{Principal Angles of the Attitude Estimate Errors for $h=0.05$s}  
\label{fig8}                                 
\end{center}                                 
\end{figure}

\begin{figure}[h!]
\begin{center}
\includegraphics[height=3.2in]{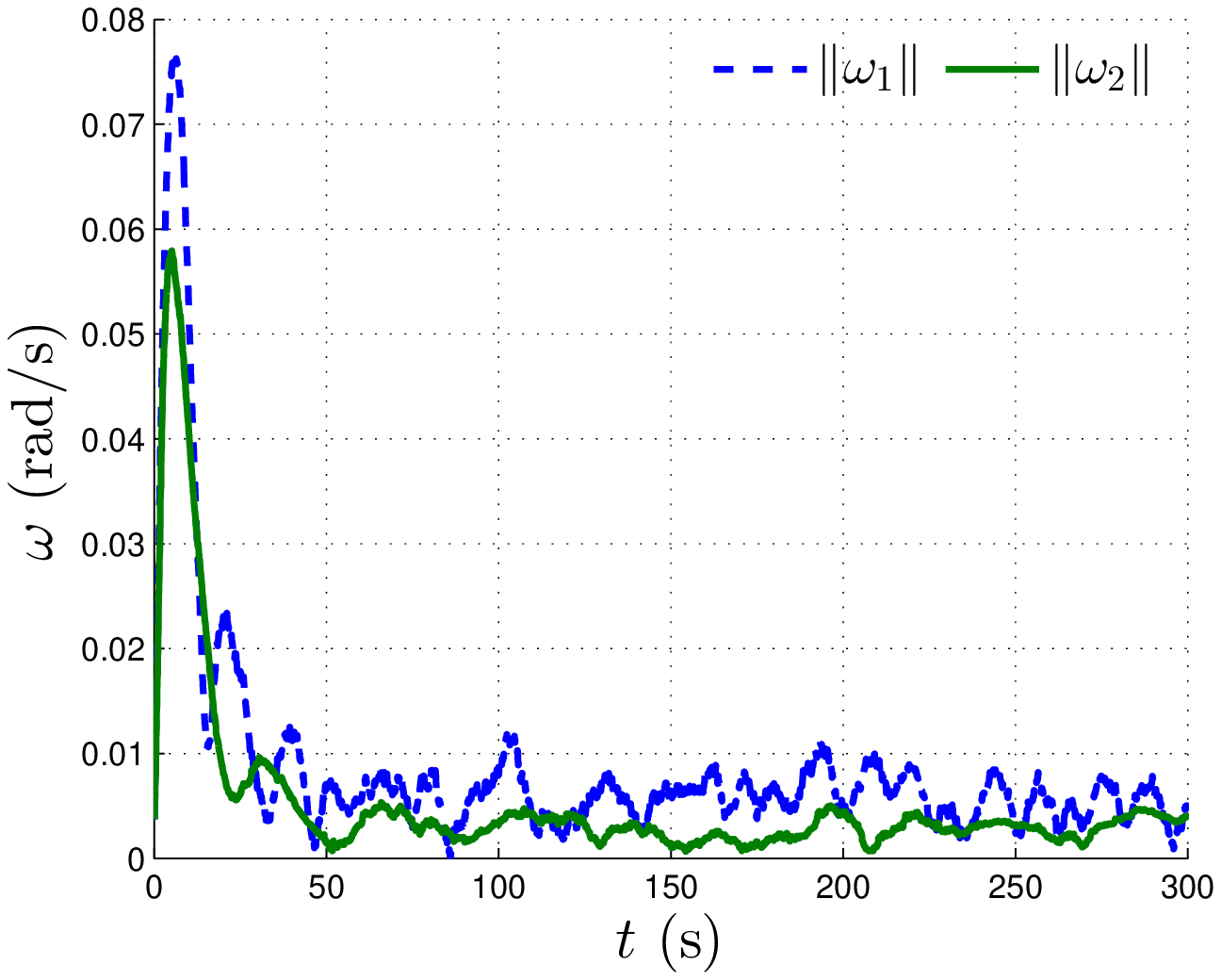}    
\caption{Norms of Angular Velocity Estimate Errors for $h=0.05$s}  
\label{fig9}                                 
\end{center}                                 
\end{figure}


\subsection{Conclusion}

This work obtains an attitude and angular velocity estimation scheme on the Lie group 
of rigid body rotational motion, assuming that measurements of inertial vectors and angular 
velocity are available in continuous-time or at a high sample rate in discrete-time. 
It is shown that Wahba's cost function for attitude determination from vector measurements 
can be generalized and cast as a Morse function on the Lie group of rigid body rotations. 
This Morse function can also be considered as an artificial potential function. A kinetic 
energy-like term, quadratic in the angular velocity estimation errors, can 
be used along with this artificial potential to construct a Lagrangian dependent on 
state estimation errors. The estimator is obtained by applying the Lagrange-d'Alembert 
principle and its discretization to this Lagrangian and a dissipation term dependent on 
the angular velocity estimation error. This estimation scheme is shown 
to be almost globally asymptotically stable, with estimates converging to actual states 
in a domain of attraction that is open and dense in the state space. In the presence of 
bounded measurement noise, the numerical results show that state estimates converge 
to a bounded neighborhood of the actual states. An implicit first order discrete-time version of 
the continuous-time estimation algorithm is obtained by applying the discrete 
Lagrange-d'Alembert principle. An explicit filter is obtained as the adjoint method corresponding to this implicit filter. A symmetric 
second order version of this estimation algorithm is constructed by composing these two 
filters. Using a realistic set of data for a rigid body, numerical 
simulations show that the estimation errors in attitude and angular velocities converge to 
a bounded neighborhood of $(I,0)$ in the presence of a bounded measurement 
noise. Some numerical comparison 
results are presented to show the performances of these filters.

\newpage

\section{COMPARISON OF AN ATTITUDE ESTIMATOR BASED ON THE LAGRANGE-D'ALEMBERT PRINCIPLE WITH SOME STATE-OF-THE-ART FILTERS} 

\label{CH05_ICRA2015} 

\hspace{\parindent}
\textit{This chapter is adapted from a paper published in Proceedings of the 2015 IEEE International Conference on Robotocs and Automation \cite{ICRA2015}. The author gratefully acknowledges Dr. Amit K. Sanyal, Dr. Vijay Kumar and Ehsan Samiei for their participation.}
\\
\\
{\bf{Abstract}}~
Discrete-time estimation of rigid body attitude and angular velocity without any knowledge of 
the attitude dynamics model, is treated using the discrete Lagrange-d'Alembert principle. 
Using body-fixed sensor measurements of direction vectors and angular velocity, a 
Lagrangian is obtained as the difference between a kinetic energy-like term that is 
quadratic in the angular velocity estimation error and an artificial potential obtained 
from Wahba's function. An additional dissipation term that depends on the angular velocity 
estimation error is introduced, and the discrete Lagrange-d'Alembert principle is applied to the 
Lagrangian with this dissipation. An implicit and an explicit first-order version of 
this discrete-time filtering scheme is presented. A comparison of this estimator is made with certain 
state-of-the-art attitude filters in the absence of bias in sensors readings. Numerical simulations show that the presented observer is robust and unlike the extended Kalman filter based schemes, its convergence does not depend on the gains values. Ultimately, the variational estimator is shown to be the most computationally efficient attitude observer.

\subsection{Measurement Model}\label{C5Sec3}
The vectors notation and discretization definitions are introduced in Sections \ref{C4S3.1}. Besides, the measurement model for direction sensors is of the form
\begin{align}
u_{j_i}=R\T_i e_{j_i}+\mathcal{D}_j\nu_{j_i},\label{DirMeasMod}
\end{align}
where the coefficient matrix $\mathcal{D}_j\in\bR^{3\times3}$ allows for different weightings of the components of the output measurement error $\nu_{j_i}$. A common assumption is that the matrix $\mathcal{D}_j$ is full rank and $\sD_j^{-1}=\mathcal{D}_j\mathcal{D}_j\T$ is positive definite. Let $\Omega_i\in\mathbb{R}^3$ be the angular velocity of the rigid body at time $t_i$ expressed in the 
body-fixed frame. The attitude kinematics is given by Poisson's equation:
\begin{align}
\dot{R_i}=R_i\Omega_i^\times.
\label{kinematics}
\end{align}\par
The measurement model for angular velocity is also as follows
\begin{align}
\Omega^m_i=\Omega_i+B\mpz{w}_i,\label{AngMeasMod}
\end{align}
where $\mpz{w}_i\in\bR^3$ is the measurement error in angular velocity and $B\in\bR^{3\times3}$ allows for different weightings for the components of the unknown input measurement $\mpz{w}_i$.\par

\subsection{Other State-of-the-Art Filters on $\SO$}\label{C5Sec5}
Some other observers are available in the literature which can estimate the attitude of the rigid body using the same measurement as explained in Chapter \ref{CH04_VE}. Three estimation schemes are used in  
comparisons with the variational filter: the geometric approximate minimum-energy (GAME), the 
multiplicative extended Kalman filter (MEKF) and a constant gain observer (CGO).

\subsubsection{GAME Filter}\label{GAME}
Generalizing Mortensen's maximum-likelihood filtering scheme, a near-optimal filter is proposed in \cite{zatruma11}. This geometric approximate minimum-energy (GAME) filter in continuous form is given by
\begin{align}
\dot{\hat{R}}&=\hat{R}(\Omega^m-P\ell)^\times,\ell=\sum_{j=1}^k\Big(\sD_j(\hat{u}_j-u_j)\Big)\times\hat{u}_j, \label{RGAME} \\
\dot{P}&=\mathcal{Q}+\dsP_s\bigg(P(2\Omega^m-P\ell)^\times\bigg)+P\Bigg(\tr\Big(\sum_{j=1}^k\dsP_s\big(\sD_j(\hat{u}_j-u_j)\hat{u}_j\T  \big)\Big)I_{3\times3}\nn\\
&~~~~~~~~~~~~~~~~~~~~~~~~~~~~~~~~~~~-\sum_{j=1}^k\dsP_s\big(\sD_j(\hat{u}_j-u_j)\hat{u}_j\T\big)+\sum_{j=1}^k\hat{u}_j^\times\sD_j\hat{u}_j^\times\Bigg)P,
\label{PGAME}
\end{align}
where $\hat{u}_j=\hat{R}\T e_j$, $\mathcal{Q}=BB\T$ with $B$ defined in \eqref{AngMeasMod}, $\dsP_s(\mpz{X})=\frac12(\mpz{X}+\mpz{X}\T)$ for $\mpz{X}\in\bR^{3\times3}$, $\sD_j=(\mathcal{D}_j\mathcal{D}_j\T)^{-1}$ with $\mathcal{D}_j$ defined in \eqref{DirMeasMod}, $\hat{R}(0)=I_{3\times3}$, and $P(0)=\frac{1}{\varphi^2}I_{3\times 3}$ where $\varphi^2$ is the variance 
of the principal angle corresponding to the initial attitude estimate.

\subsubsection{The Multiplicative Extended Kalman Filter}
The Multiplicative Extended Kalman Filter (MEKF) presented in \cite{ekf2,Seba1,Seba2} in continuous form is as follows
\begin{align}
\dot{\hat{R}}&=\hat{R}(\Omega^m-P\ell)^\times,\;\ell=\sum_{j=1}^k\Big(\sD_j(\hat{u}_j-u_j)\Big)\times\hat{u}_j, \label{RMEKF} \\
\dot{P}&=\mathcal{Q}+\dsP_s\Big(P(2\Omega^m)^\times\Big)-P\Bigg(\sum_{j=1}^k\hat{u}_j^\times\sD_j\hat{u}_j^\times\Bigg)P,
\label{PMEKF}
\end{align}
where $\mathcal{Q}$, $\dsP_s(\mpz{X})$ and $\sD_j$ are as defined in Subsection \ref{GAME}, 
and $\hat{R}(0)$ and $P(0)$ are set to the same values as in the GAME filter.

\subsubsection{The Constant Gain Observer}
The Constant Gain Observer (CGO) presented in \cite{mahapf08} in continuous form is also represented as
\begin{align}
\dot{\hat{R}}&=\hat{R}\Big(\Omega^m-K_P\bar\ell \Big)^\times,\;\bar\ell=\sum_{j=1}^k \big(u_j\times\hat{u}_j\big),
\label{CGO}
\end{align}
where $K_P$ is a constant gain and $\hat{R}(0)=I_{3\times3}$.\par
Note that all the filters presented here are in continuous setting and a discretized version of them need to be implemented in numerical simulations, using the measurement model defined in Section \ref{C5Sec3}. The discrete-time versions of these estimators presented in \cite{ZamPhD} use the 
unit quaternion representation.


\subsection{Numerical Simulations and Discussion}\label{C5Sec6}
The performance of the discrete-time variational estimator is compared against that of the estimation 
schemes presented in Section \ref{C5Sec5}, under identical conditions. This means that all the 
estimation schemes work on the  same rigid-body dynamics, have the same initial estimate errors, 
equal time steps, and identical measurement noise. The sampling period and the total simulation 
time are $h=0.01$s and $T=20$s, respectively. A rigid body with prescribed dynamics and 
inputs is considered. 
Three inertially known directions are measured by the sensors and form the matrix $E=I_{3\times3}$ in inertial frame. These measurements contain known levels of noise, however, all sensors 
are assumed to be unbiased. The initial rotation matrix is selected randomly with zero mean and a standard 
deviation of $std_{R0}=60^{\circ}$. The rigid body also has the following angular velocity profile:
\begin{align}
\Omega=\bbm \sin(\frac{2\pi}{15}t)\\-\sin(\frac{2\pi}{18}t+\frac{\pi}{20})\\\cos(\frac{2\pi}{17}t) \ebm
\end{align}

All the estimators start from the same initial attitude estimate, which is $\hat{R}_0=I_{3\times3}$. 
The initial angular velocity estimates are also set to be identical, as follows. According to 
eqs. \eqref{RGAME} and \eqref{RMEKF}, the initial angular velocity estimate errors are given 
by $P(0)\times\ell(0)$ for GAME and MEKF. For the variational estimator, choosing 
$\omega_0=P(0)\times\ell(0)$ and for the CGO, choosing $K_p=P(0)$ satisfies this condition. 
The corresponding initial 
value for covariance matrices in equations \eqref{PGAME} and \eqref{PMEKF} are chosen as $P(0)=\frac{1}{std_{q0}^2}I_{3\times3}=\frac{9}{\pi^2}I_{3\times3}$.\par
The inertia scalar gain for the variational observer is $m=0.5\mbox{ kg}$ and its dissipation matrix is selected as the following positive definite matrix:
\[D=\diag\big([1.8\;\;\; 1.95\;\;\; 2.1]\T\big)\mbox{ N.s}.\]\par
According to \cite{Automatica}, $\Phi(\cdot)$ could be any $C^2$ function, and here it is set to 
$\Phi(x)=x$. The weight matrix for the three directions is also
\[W=\diag\big([1.67\;\;\; 1.11\;\;\; 0.56]\T\big)\mbox{ N.s}.\]\par
As discussed in \cite{ZamPhD}, GAME filter is designed based on rotation matrices, but the numerical 
implementation utilizes unit quaternions. The sensors readings are assumed to be bias-free in all cases. The constant gain for CGO is also set as $k_P=P(0)$, to make its attitude and angular 
velocity estimates identical to the other filters. We compare the performance of these 
filters for two different cases. 

\subsubsection{CASE 1: High Noise Levels}
Both directions measurement error $\nu_j$ and angular velocity measurement error $\mpz{w}$ are random zero mean signals whose probability distribution follow a bump function with unit maximum. The coefficient matrices $\mathcal{D}_j$ in \eqref{DirMeasMod} are chosen equal to $30^\circ$. The coefficient matrix $B$ in \eqref{AngMeasMod} is also set as $25^\circ$/s. Since different estimation techniques have different (initial) filter gains, the equivalence in these 
parameters need to be defined reasonably. In this comparison, the filter gains are set such that 
all filters have the same {\em initial rate} of the attitude estimate, or equivalently they have the same 
initial angular velocity estimate $\hat\Omega_0$. The principal 
angle profiles of the attitude estimate error are compared in Fig. \ref{fig1}.\par

\begin{figure}[h!]
\begin{center}
\includegraphics[height=3.2in]{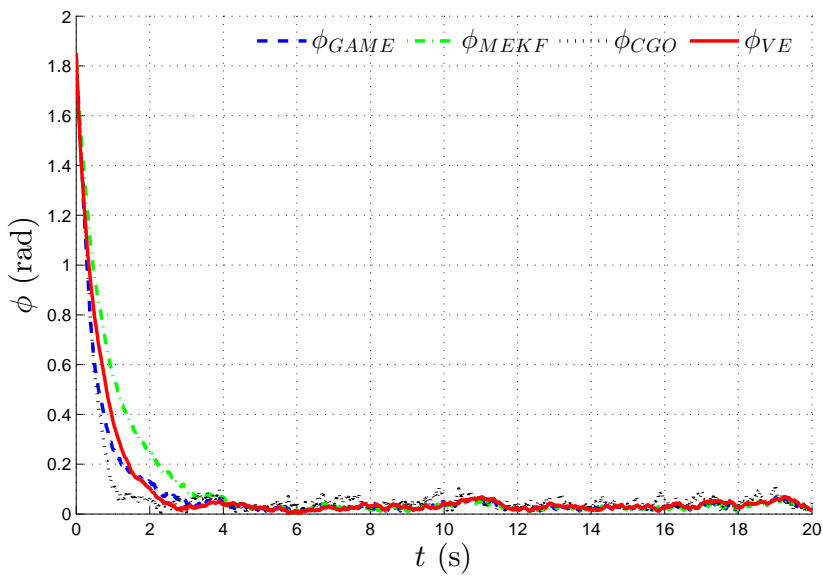}    
\caption{Principal Angles of Attitude Estimation Error For Noise Levels In \cite{ZamPhD}.}  
\label{fig1}                                 
\end{center}                                 
\end{figure}

Based on the behavior of the filters shown in Fig.~\ref{fig1}, the variational estimator converges 
fast enough, if the filter gains are chosen wisely. Although the transient behavior of the LGVI is not the best, the differences of the convergences are not remarkable. The settling time for this filter is as small as the other filters' settling times. Besides, the steady state phase of the attitude estimate error is as smooth as other observers.

\subsubsection{CASE 2: Low Noise Levels, with Filter Gains as Before} 
In this case, the noise signals are considered to be the same type as the previous case (normally-distributed random zero mean bump functions), but with much smaller levels, which are used in simulation implementations of \cite{Automatica}. These levels are close to common coarse attitude and angular velocity sensors for space applications. All the observer gains are kept the same as in case 1, 
to see the filters' performance in the case that these gains are not designed for known noise 
statistics. The coefficient matrices $\mathcal{D}_j$ and $B$ are chosen in such a way that the magnitude of each component of the signals $\mathcal{D}_j\nu_{j_i}$ and $B\mpz{w}_i$ are $2.4^\circ$ and $0.97^\circ$/s, respectively. The principal angle of attitude estimate error for the mentioned filters are plotted in Fig.~\eqref{fig2}.
\begin{figure}[h!]
\begin{center}
\includegraphics[height=3.2in]{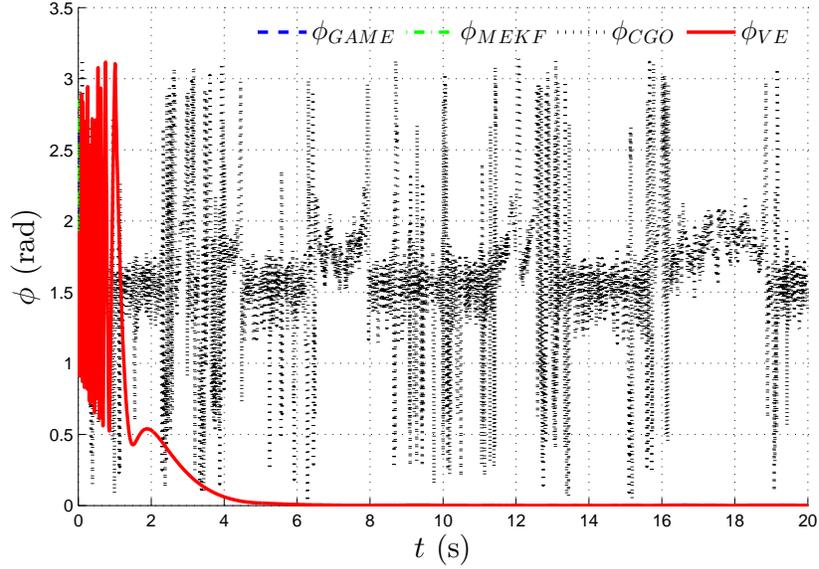}    
\caption{Principal Angles of Attitude Estimation Error For Noise Levels In \cite{Automatica}.}  
\label{fig2}                                 
\end{center}                                 
\end{figure}
\begin{figure}[h!]
\begin{center}
\includegraphics[height=3.2in]{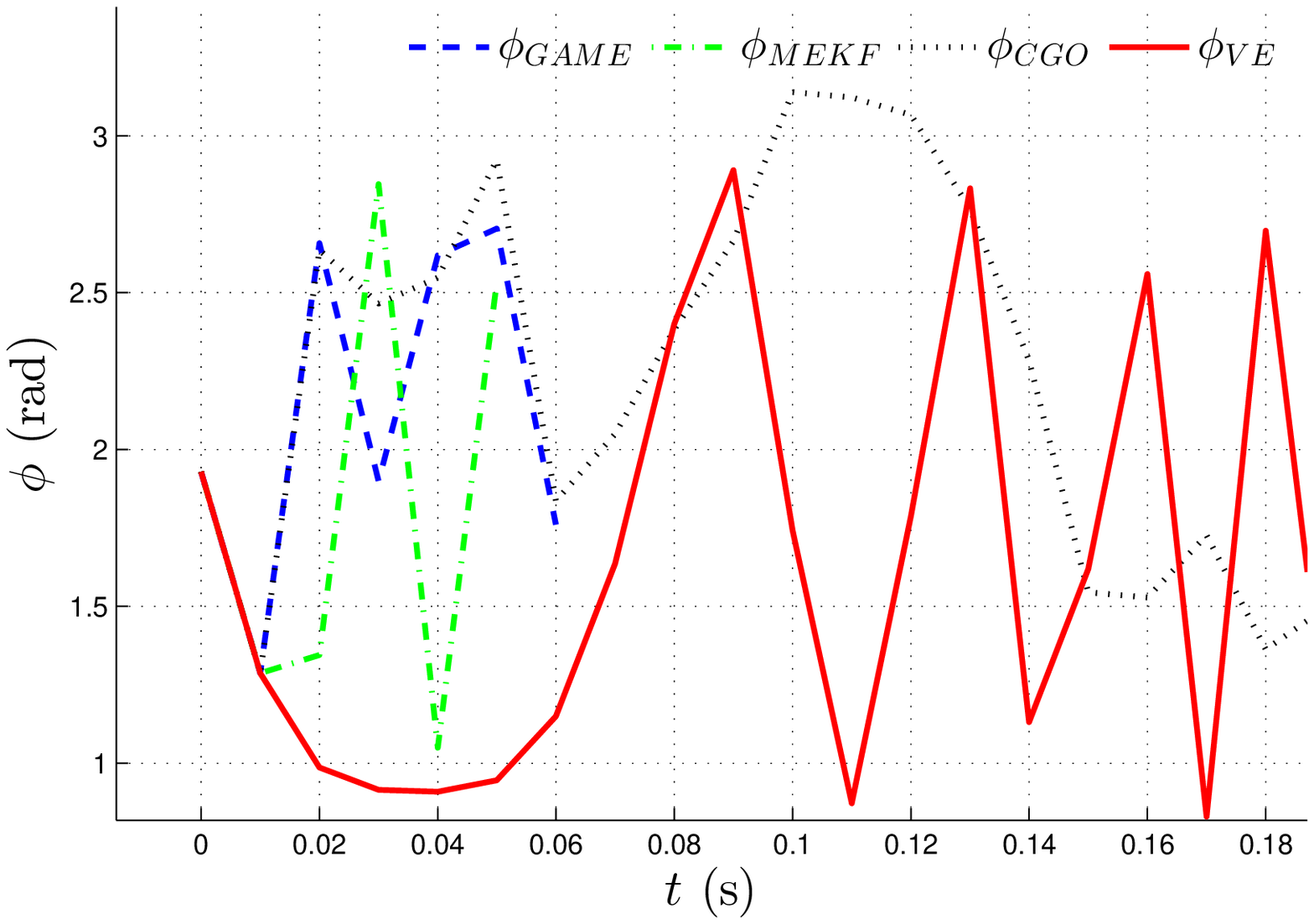}    
\caption{Principal Angles of Attitude Estimation Error For Noise Levels In \cite{Automatica} For Initial Few Steps.}  
\label{fig3}                                 
\end{center}                                 
\end{figure}\par
A magnified behavior of these filters are depicted in Fig. \eqref{fig3}. As can be observed, in this case, the GAME filter and MEKF become singular. On the other hand, the CGO and the variational estimator are stable and filter noise out from the estimates. The settling times are also sufficiently small.\par
Comparing these two cases, one can conclude that although the MEKF and GAME filter perform nicely in the presence of measurement noise with known distribution and level; however, they may not be stable and their initial gains need to be reset, if the noise signal's nature changes. Hence, one major downfall of these filters are their dependence on the value of the initial estimator gain. On the contrary, the variational estimator is a robust filter with proven stability regardless of the statistics 
of the noise. This is manly because of the almost global asymptotically stable structure of this estimator, as has been shown in \cite{Automatica} using the total energy as a Lyapunov function.\par

\subsubsection{CASE 3: Smaller Initial Estimate Error and Low Noise Levels, with Filter Gains as Before}
\begin{figure}[h!]
\begin{center}
\includegraphics[height=3.2in]{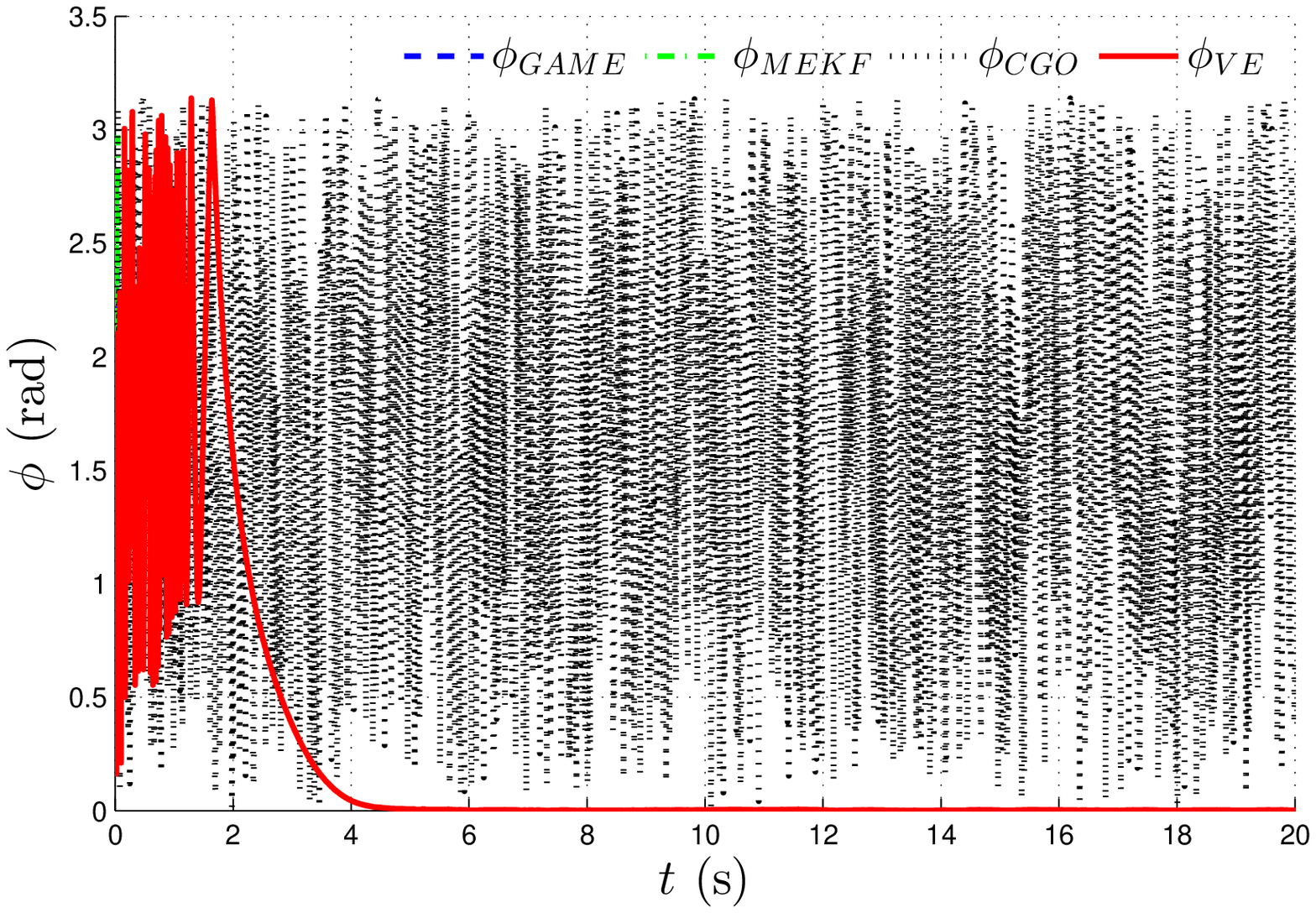}    
\caption{Principal Angles of Attitude Estimation Error For Noise Levels In \cite{Automatica}.}  
\label{fig4}                                 
\end{center}                                 
\end{figure}
\begin{figure}[h!]
\begin{center}
\includegraphics[height=3.2in]{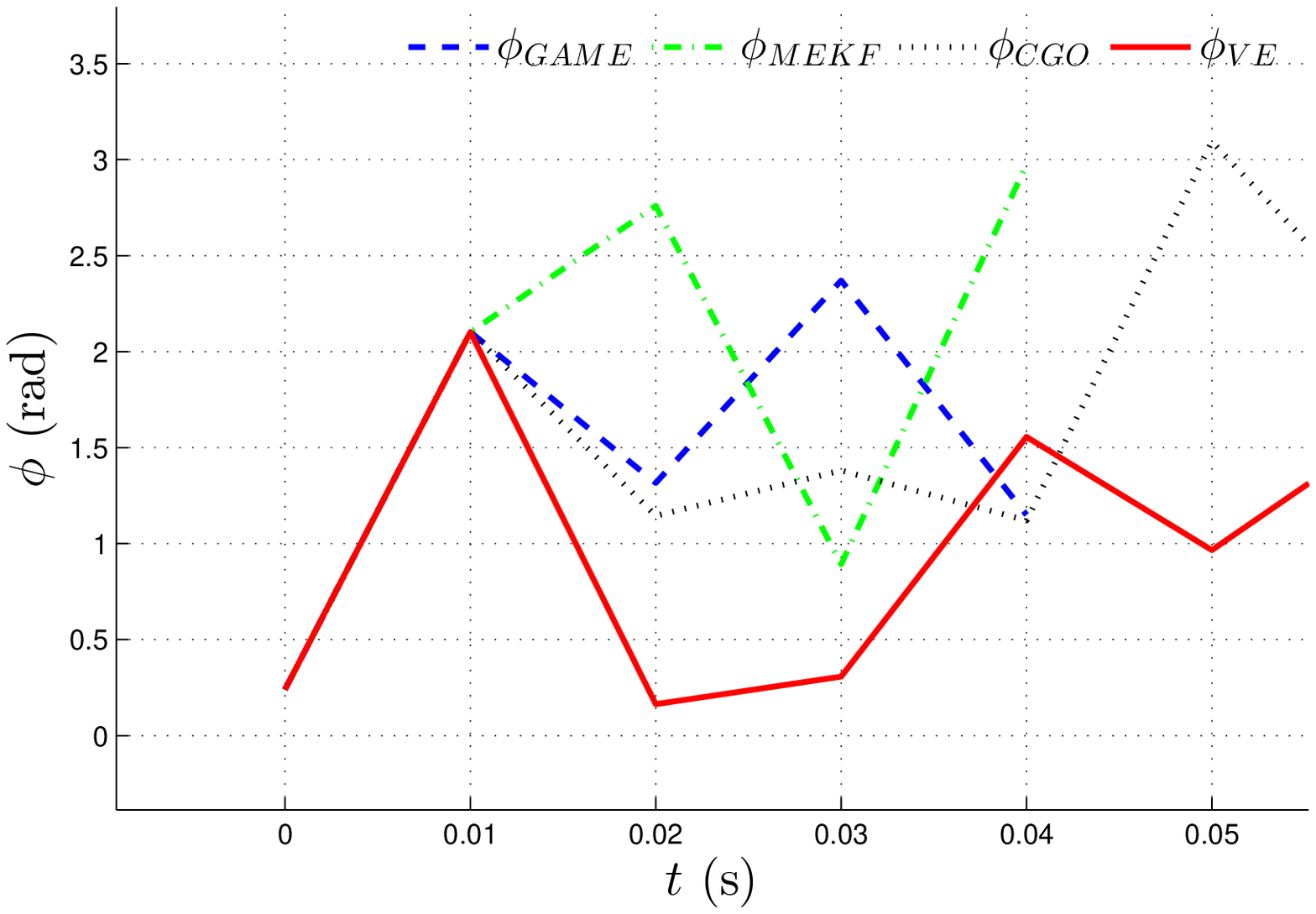}    
\caption{Principal Angles of Attitude Estimation Error For Noise Levels In \cite{Automatica} For Initial Few Steps.}  
\label{fig5}                                 
\end{center}                                 
\end{figure}\par

\subsubsection{Discussion}
Besides, considering the run-times of these filters, one can notice that the explicit version of the variational observer (which has the same behavior as the implicit version) is considerably faster than the others. Using the aforementioned initial conditions and filter gains, the run-times are depicted for a simulation duration of $20$s for these four filters in Table \ref{Tab1}.

\begin{table}[h]
\caption{Run-time of different Filters}\label{Tab1}
\hspace{.75in}\resizebox{4.4in}{!} {
\begin{tabular}{|c||c|c|c|c|}
  \hline
\raisebox{1.5\height}{\color{white} a} Filter & GAME&MEKF&CGO&Var. Est.\raisebox{-1.5\height}{\color{white} a}\\
  \hline
\raisebox{1.5\height}{\color{white} a}   Run-Time & 0.6864 s&0.6240 s&0.4304 s&0.1716 s\raisebox{-1.5\height}{\color{white} a} \\
  \hline
\end{tabular}}
\end{table}

It can be seen from this table that the explicit variational estimator is computationally faster than 
the other filters used in this comparison. This advantage make this LGVI the best choice in real-time experiments, where the computation time is a bottleneck \cite{ICC2015}.


\subsection{Conclusion}\label{C5Sec7}
This work presents both an implicit and an explicit discrete-time attitude and angular velocity 
estimation scheme on the Lie group of rigid body rotational motion, assuming that measurements 
of inertial vectors and angular velocity are available at a high sample rate in discrete-time. 
A discrete-time filter is obtained, in the form of an implicit first order Lie group 
variational integrator, by applying the discrete Lagrange-d'Alembert principle to the discrete 
Lagrangian and a dissipation term dependent on the angular velocity estimation error. An 
explicit filter is also derived as the adjoint method corresponding to this implicit filter. The behavior of this estimation scheme is compared with three state-of-the-art observers for attitude estimation. Using a realistic set of data for a rigid body, numerical simulations show that the variational estimator performs as good as other filters, taking less computational budget. Furthermore, unlike the GAME filter and MEKF, it always is stable and its convergence is not dependent on the type and level of measurement noise.

\newpage

\section{THE VARIATIONAL ATTITUDE ESTIMATOR IN THE PRESENCE OF BIAS IN ANGULAR VELOCITY MEASUREMENTS} 

\label{CH06_Bias} 

\hspace{\parindent}
\textit{This chapter is adapted from a paper submitted to the 2016 American Control Conference. The author gratefully acknowledges Dr. Amit K. Sanyal for his participation.}
\\
\\
{\bf{Abstract}}~
In this work, the variational attitude estimator 
is generalized to include angular velocity measurements that have a constant bias in addition to 
measurement noise. It is shown that the state estimates converge to true states almost 
globally over the state space if the measurements are perfect. Further, it is shown that the 
bias estimates converge to the true bias once the state estimates converge to the true states.

\subsection{Measurement Model}\label{C6Sec3}

For rigid body attitude estimation, assume that some inertially-fixed vectors are measured 
in a body-fixed frame, along with body angular velocity measurements having a constant bias.
Let $k\in\mathbb{N}$ known inertial vectors be measured in the coordinate frame fixed to the 
rigid body as introduced in \ref{C4S1}. Moreover, the direction vector measurements are given by
\begin{align}
u_j^m=R\T e_j+ \nu_j\, \mbox{ or }\, U^m = R\T E+ N, \label{C6DirMeasMod}
\end{align}
where $\nu_j\in\bR^3$ is an additive measurement noise vector and $N\in\bR^{3\times k}$ 
is the matrix with $\nu_j$ as its $j^{\mbox{th}}$ column vector.

The attitude kinematics for a rigid body is given by Poisson's equation \eqref{Poisson} and the measurement model for angular velocity is 
\begin{align}
\Omega^m=\Omega+ w +\beta,\label{C6AngMeasMod}
\end{align}
where $w\in\bR^3$ is the measurement error in angular velocity and $\beta\in\bR^3$ is 
a vector of bias in angular velocity component measurements, which we consider to be a 
constant vector. 

\subsection{Attitude State and Bias Estimation Based on the Lagrange-d'Alembert 
Principle}\label{C6Sec4}
In order to obtain attitude state estimation schemes from continuous-time vector and angular 
velocity measurements, we apply the Lagrange-d'Alembert principle to an action functional of 
a Lagrangian of the state estimate errors, with a dissipation term linear 
in the angular velocity estimate error. This section presents an estimation scheme obtained 
using this approach, as well as stability and convergence properties of this estimator. 

\subsubsection{Lagrangian Constructed from Measurement Residuals}
The ``energy" contained in the errors between the estimated and the measured 
inertial vectors is given by $\cU (\hat R,U^m)$, where $\cU:\SO\times\bR^{3\times k}\to\bR$ is 
defined by \eqref{attindex} and depends on the attitude estimate. Let $\hat\Omega\in\bR^3$ 
and $\hat\beta\in\bR^3$ denote the estimated angular velocity and bias vectors, 
respectively. The ``energy" contained in the vector error between the estimated and the 
measured angular velocity is then given by
\be \cT (\hat\Omega,\Omega^m,\hat\beta)=\frac{m}{2}(\Omega^m-\hat\Omega-\hat\beta) \T 
(\Omega^m-\hat\Omega-\hat\beta). \label{C6angvelindex} \ee
where $m$ is a positive scalar. One can consider the Lagrangian composed of these 
``energy" quantities, as follows:
\begin{align} 
\cL (\hat R,U^m,\hat\Omega,&\Omega^m,\hat\beta) = \cT(\hat\Omega,\Omega^m,\hat\beta)-
\cU (\hat R,U^m) \label{C6cLag}\\
=&\frac{m}{2}(\Omega^m-\hat\Omega-\hat\beta) \T (\Omega^m-\hat\Omega-\hat\beta)
-\Phi\Big( \frac12\lan E-\hat{R}U^m,(E-\hat{R}U^m)W\ran\Big). \nn
\end{align}
If the estimation process is started at time $t_0$, then the action functional 
of the Lagrangian \eqref{C6cLag} over the time duration $[t_0,T]$ is expressed as
\begin{align}
\cS (\cL(\hat R&,U^m,\hat\Omega,\Omega^m))= \int_{t_0}^T \big(\cT (\hat\Omega,
\Omega^m,\hat\beta)- \cU (\hat R,U^m)\big)\di s\label{C6eq:J6}\\
= \int_{t_0}^T& \bigg\{ \frac{m}{2}(\Omega^m-\hat\Omega-\hat\beta) \T (\Omega^m-
\hat\Omega-\hat\beta)
- \Phi\Big(\frac{1}{2}\lan E-\hat{R}U^m,(E-\hat{R}U^m)W\ran\Big) \bigg\} \di s.
 \nn
\end{align}

\subsubsection{Variational Filtering Scheme}
Define the angular velocity measurement residual and the dissipation term:
\be \omega := \Omega^m-\hat\Omega- \hat\beta, \;\ \tau_D= D\omega, \label{C6angvelres} \ee
where $D\in\bR^{3\times 3}$ is a positive definite filter gain matrix. Consider attitude state estimation in continuous time in the presence of measurement noise 
and initial state estimate errors. Applying the Lagrange-d'Alembert principle to the action 
functional $\cS (\cL(\hat R,U^m,\hat\Omega,\Omega^m))$ given by \eqref{C6eq:J6}, in the 
presence of a dissipation term linear in $\omega$, leads to the 
following attitude and angular velocity filtering scheme. 
\begin{theorem} \label{C6filterN}
The filter equations for a rigid body with the attitude kinematics \eqref{Poisson} and with 
measurements of vectors and angular velocity in a body-fixed frame, are of the form
\begin{align}
\begin{cases}
&\dot{\hat{R}}=\hat{R}\hat{\Omega}^\times=\hat{R}(\Omega^m-\omega-\hat\beta)^\times,
\vspace{3mm}\\
&m\dot{\omega}= -m\hat{\Omega}\times \omega+\Phi'\big(\cU^0(\hat{R},U^m)\big)S_L(\hat{R})-D\omega,\vspace{3mm}\\
&\hat\Omega=\Omega^m-\omega-\hat\beta,
\end{cases}
\label{C6eq:filterNoise}
\end{align}
where $\hat{R}(t_0)=\hat{R}_0$, $\omega(t_0)=\omega_0
=\Omega^m_0-\hat\Omega_0$, $S_L(\hat R)= \mrm{vex}\big(L\T \hat R - \hat R \T L\big)\in\bR^3$, 
$L=EW(U^m)\T$ and $W$ is chosen to satisfy the conditions in Lemma 2.1 of \cite{Automatica}.
\label{C6filter1}
\end{theorem}
{\em Proof}: In order to find an estimation scheme that filters the measurement noise in the 
estimated attitude, take the first variation of the action functional \eqref{C6eq:J6} with respect to 
$\hat R$ and $\hat\Omega$ and apply the Lagrange-d'Alembert principle with the dissipative 
term in \eqref{C6angvelres}. Consider the potential term $\cU^0(\hat R,U^m)$ as defined by \eqref{U0def}.
Taking the first variation of this function with respect to $\hat{R}$ gives
\begin{align}
\delta\cU^0&=\lan -\delta\hat RU^m,(E-\hat RU^m)W \ran\nn\\
&=\frac{1}{2}\lan \Sigma^\times,U^m WE\T\hat R-\hat R\T EW(U^m)\T \ran,\nn\\
&=\frac{1}{2}\lan \Sigma^\times, L\T\hat R-\hat R\T L \ran=S\T_L(\hat R)\Sigma.
\end{align}
Now consider $\cU(\hat R,U^m)=\Phi\big(\cU^0(\hat R,U^m)\big)$. Then,
\begin{align}
\delta\cU=\Phi'\big(\cU^0(\hat R,U^m)\big)\delta\cU^0=\Phi'\big(\cU^0(\hat R,U^m)\big)S\T_L(\hat R)\Sigma.
\end{align}
Taking the first variation of the kinematic energy term associated with the artificial system \eqref{C6angvelindex} with respect to $\hat\Omega$ yields
\begin{align}
\delta\cT&=-m(\Omega^m-\hat\Omega-\hat\beta)\T\delta\hat\Omega=-m(\Omega^m-\hat\Omega
-\hat\beta)\T(\dot\Sigma+\hat\Omega\times\Sigma)\nn\\
&=-m\omega\T(\dot\Sigma+\hat\Omega\times\Sigma),
\end{align}
where $\omega$ is as given by \eqref{C6angvelres}. 
Applying Lagrange-d'Alembert principle leads to
\begin{align}
&~~~~~\delta\cS+\int_{t_0}^T\tau_D\T\Sigma\di t=0\\
&\Rightarrow \int_{t_0}^T\Big\{-m\omega\T(\dot\Sigma+\hat\Omega\times\Sigma)-\Phi'\big(\cU^0(\hat R,U^m)\big)S\T_L(\hat R)\Sigma+\tau_D\T\Sigma\Big\}\di t=0 \Rightarrow\nn\\
& -m\omega\T\Sigma\big|_{t_0}^T+\int_{t_0}^T m\dot{\omega}\T\Sigma\di t
=\int_{t_0}^T\Big\{m\omega\T\hat\Omega^\times+\Phi'\big(\cU^0(\hat R,U^m)\big)S\T_L(\hat R)-\tau_D\T\Big\}\Sigma\di t,\nn
\end{align}
where the first term in the left hand side vanishes since $\Sigma(t_0)=\Sigma(T)=0$. After 
substituting the dissipation term $\tau_D=D\omega$, one obtains the second equation in 
\eqref{C6eq:filterNoise}. \hfill\ensuremath{\square}

\subsection{Stability and Convergence of Variational Attitude Estimator}\label{C6Sec5}
The variational attitude estimator given by Theorem \ref{C6filterN} can be used for constant 
or time-varying bias in the angular velocity measurements given by the measurement model 
\eqref{C6AngMeasMod}. The following analysis gives the stability and convergence properties of 
this estimator for the case that $\beta$ in equation \ref{C6AngMeasMod} is constant.

\subsubsection{Stability of Variational Attitude Estimator} 
Prior to analyzing the stability of this attitude estimator, it is useful and instructive to interpret 
the energy-like terms used to define the Lagrangian in equation \eqref{C6cLag} in terms of state 
estimation errors. The following result shows that the measurement residuals, and therefore 
these energy-like terms, can be expressed in terms of state estimation errors in the case of 
perfect measurements. 
\begin{proposition}
Define the state estimation errors
\begin{align} 
&Q= R \hat R\T \, \mbox{ and }\, \omega= \Omega-\hat\Omega-\tilde\beta, \label{C6esterrs} \\
&\mbox{where }\, \tilde\beta = \beta -\hat\beta. \label{C6biaserr} 
\end{align}
In the absence of measurement noise, the energy-like terms \eqref{attindex} and 
\eqref{C6angvelindex} can be expressed in terms of these state estimation errors as follows:
\begin{align}
&\cU (Q) =\Phi \Big(\lan I-Q, K\ran \Big)\, \mbox{ where }\, K=EWE\T, \label{C6attindperf} \\
&\mbox{and } \cT (\omega)= \frac{m}{2} \omega\T\omega. \label{C6angvindperf} 
\end{align} 
\end{proposition}
{\em Proof}: The proof of the above statement is obtained by first substituting $N=0$ and $w=0$ in 
equations \eqref{C6DirMeasMod} and \eqref{C6AngMeasMod}, respectively. The resulting expressions 
for $U^m$ and $\Omega^m$ are then substituted back into equations \eqref{attindex} and 
\eqref{C6angvelindex} respectively. Note that the same variable $\omega$ is used to represent 
the angular velocity estimation error for both cases: with and without measurement noise. 
Expression \eqref{C6attindperf} is also derived in \cite{Automatica}. 
\hfill\ensuremath{\square}
 
The stability of this estimator, for the case of constant rate gyro bias vector $\beta$, is given by  
the following result.
\begin{theorem}\label{C6stabproof}
Let $\beta$ in equation \eqref{C6AngMeasMod} be a constant vector. Then the variational attitude 
estimator given by equations \eqref{C6eq:filterNoise}, in addition to the following equation for 
update of the bias estimate:
\be \dot{\hat\beta}=  \Phi' \big(\cU^0(\hat{R},U^m)\big) P^{-1} S_L(\hat{R}), \label{C6biasest} \ee
is Lyapunov stable for $P\in\bR^{3\times 3}$ positive definite.
\end{theorem}
{\em Proof}: To show Lyapunov stability, the following Morse-Lyapunov function is considered:
\begin{align} 
\begin{split}
V (U^m,\Omega^m,\hat{R},\hat\Omega,\hat\beta) &= \frac{m}{2} (\Omega^m-\hat\Omega-\hat\beta)\T
(\Omega^m-\hat\Omega-\hat\beta) \\
&+ \Phi \big(\cU^0(\hat{R},U^m)\big) + \frac12 (\beta-\hat\beta)\T P (\beta-\hat\beta). 
\end{split} \label{C6Lyapf} 
\end{align}  
Now consider the case that there is no measurement noise, i.e., $N=0$ and $w=0$ in 
equations \eqref{C6DirMeasMod} and \eqref{C6AngMeasMod}, respectively. In this case, 
the Lyapunov function \eqref{C6Lyapf} can be re-expressed in terms of the errors $\omega$, $Q$ and 
$\tilde\beta$ defined by equations \eqref{C6esterrs}-\eqref{C6biaserr}, as follows:
\be V(Q,\omega,\tilde\beta)= \frac{m}{2}\omega\T\omega + \Phi\big(\lan I-Q, K\ran\big) +
\frac12\tilde\beta\T P\tilde\beta. \label{C6errLyap} \ee
The time derivative of the attitude estimation error, $Q\in\SO$, is obtained as:
\be \dot Q= R(\Omega-\hat\Omega)^\times\hat{R}\T= Q\big(\hat R (\omega-\tilde\beta)\big)^\times, 
\label{C6dotQ} \ee
after substituting for $\hat\Omega$ from the third of equations \eqref{C6eq:filterNoise} in the 
case of zero angular velocity measurement noise (when $\Omega^m= \Omega+ \beta$). The 
time derivative of the Morse-Lyapunov function expressed as in \eqref{C6errLyap} can now be obtained 
as follows: 
\begin{align}
\dot V(Q,\omega,\tilde\beta)= m \omega\T\dot\omega -\Phi'\big(\lan I-Q, K\ran\big) S_L\T (\hat R) 
(\omega-\tilde\beta) -\tilde\beta\T P\dot{\hat\beta}. \label{C6dotV} 
\end{align}
After substituting equation \eqref{C6biasest} and the second of equations \eqref{C6eq:filterNoise} in the 
above expression, one can simplify the time derivative of this Lyapunov function along the 
dynamics of the estimator as
\be \dot V(Q,\omega,\tilde\beta)= -\omega\T D\omega \le 0. \label{C6Vdotsimp} \ee
The time derivative \eqref{C6Vdotsimp} is negative semi-definite in the states $(Q,\omega,\tilde\beta)
\in\Ta\SO\times\bR^3$ of this estimator. This proves the result. \hfill\ensuremath{\square}

Lemma \ref{lem1} and \ref{lem2} are required to show the convergence of estimation errors. Lemma \ref{lem1} provides guidelines on how to choose the weight matrix $W$ in Wahba's 
cost function such that it is a Morse function on $\SO$. Choosing $W$ according to these 
guidelines, which depend on the set of inertial vector observed (denoted $E$ here), ensures 
that this is a smooth function on $\SO$ that has the minimum possible number of critical points 
as dictated by the Morse lemma~\cite{bo:miln}. Note that the estimated attitude coincides with the 
true attitude when $Q=I$, which is the minimum of this Morse function according to Lemma 
\ref{lem2}.

\subsubsection{Domain of Convergence of Variational Attitude Estimator}
The domain of convergence of this estimator is given by the following result. 
\begin{theorem}\label{C6convproof}
In the absence of measurement noise, the variational attitude estimator with biased velocity 
measurements, given by eqs. \eqref{C6eq:filterNoise} and \eqref{C6biasest}, converges asymptotically 
to $(Q,\omega,\tilde\beta)=(I,0,0)\in\Ta\SO\times\bR^3$. Further, the domain of attraction is 
a dense open subset of $\Ta\SO\times\bR^3$.
\end{theorem}
{\em Proof}: Note that the error dynamics for the attitude estimate error is given by 
\be \dot Q = Q \psi^\times \, \mbox{ where }\, \psi=\hat R(\omega-\tilde\beta), \label{C6Qdot} \ee
while the error dynamics for the angular velocity estimate error $\omega$ is given by the 
second of equations \eqref{C6eq:filterNoise} and the bias estimate error dynamics is obtained 
from \eqref{C6biaserr} and \eqref{C6biasest} as  
\be \dot{\tilde\beta}= -\Phi' \big(\cU^0(\hat{R},U^m)\big) P^{-1} S_L(\hat{R}). \label{C6betildot} \ee
Therefore, the error dynamics for $(Q,\omega,\tilde\beta)$ is non-autonomous, since they depend 
explicitly on $(\hat R,\hat\Omega)$. In the absence of measurement noise, considering 
\eqref{C6errLyap} and \eqref{C6Vdotsimp} and applying Theorem 8.4 in \cite{khal}, one can conclude 
that $\omega\T D\omega\rightarrow 0$ as $t\rightarrow \infty$, which consequently implies 
$\omega\rightarrow0$. Thus, the positive limit set for this system is contained in
\begin{align}
\cE = \dot{V}^{-1}(0)=\big\{(Q,\omega,\tilde\beta)\in\Ta\SO\times\bR^3:\omega\equiv0\big\}. 
\label{C6dotV0}
\end{align}
Substituting $\omega\equiv 0$ in the filter equations \eqref{C6eq:filterNoise} in the absence of 
measurement noise, we obtain the positive limit set where $\dot V\equiv 0$ (or $\omega\equiv 0$) 
as the set
\begin{align}
\begin{split}
\mathscr{I} &= \big\{(Q,\omega,\tilde\beta)\in\Ta\SO\times\bR^3: S_K(Q)\equiv 0, 
\omega\equiv0, \psi= 0 \big\} \\
&= \big\{(Q,\omega,\tilde\beta)\in\SO\times\bR^3: Q \in C_Q,\, \omega\equiv0,\, \tilde\beta=0 \big\},
\end{split} \label{C6invset}
\end{align}
where $\psi\in\bR^3$ is defined by \eqref{C6Qdot} and $C_Q$ is the set of critical points of 
$\lan I-Q, K\ran$, given by equation (12) or (20) in~\cite{Automatica}. Therefore, in the absence 
of measurement errors, all the solutions of this filter converge asymptotically to the set 
$\mathscr{I}$. Thus, the attitude estimate error converges to the set of critical points of 
$\lan I-Q, K\ran$ in this intersection. The unique global minimum of the Morse-Lyapunov 
function $V$ is at $(Q,\omega,\tilde\beta)=(I,0,0)$ according to Lemma \ref{lem2}; therefore, 
this state estimation error is asymptotically stable.

Now consider the set
\be \mathscr{C}= \mathscr{I}\setminus (I,0,0),  \label{C6othereqb} \ee
which consists of all stationary states that the estimation errors may converge to, besides 
the desired estimation error state $(I,0,0)$. Note that all states in the stable manifold of a 
stationary state in $\mathscr{C}$ will converge to this stationary state. 
From the properties of the critical points $Q_i\in C_Q\setminus (I)$ of $\Phi(\lan K, I- Q\ran)$ 
given in Lemma \ref{lem2}, we see that the stationary points in $\mathscr{I}\setminus (I,0,0)=
\big\{ (Q_i, 0, 0) : Q_i\in C_Q\setminus (I)\big\}$ have stable manifolds whose dimensions
depend on the index of $Q_i$ in $\SO$. Since the angular velocity estimate error $\omega$ 
and the bias estimate error $\tilde\beta$ converge globally to the zero vector according to 
\eqref{C6invset}, the dimension of the stable manifold $\cM^S_i$ of $(Q_i, 0, 0)\in\SO\times\bR^3$ is 
\be \dim (\cM^S_i) = 6+(3-\,\mbox{index of } Q_i)= 9- \,\mbox{index of } Q_i. \label{C6dimStabM} \ee
Therefore, the stable manifolds of $(Q,\omega,\tilde\beta)=(Q_i,0,0)$ are six-dimensional, 
seven-dimensional, or eight-dimensional, depending on the index of $Q_i\in C_Q\setminus (I)$ 
according to \eqref{C6dimStabM}. Moreover, the value of the Lyapunov function $V(Q,\omega,
\tilde\beta)$ is non-decreasing (increasing when $(Q,\omega,\tilde\beta)\notin\mathscr{I}$) 
for trajectories on these manifolds going backwards in time. This implies that the metric distance 
between error states $(Q,\omega,\tilde\beta)$ along trajectories on the stable manifolds 
$\cM^S_i$ grows with the time separation between these states, and this property does not 
depend on the choice of the metric on $\Ta\SO\times\bR^3$. Therefore, these stable manifolds 
are embedded (closed) submanifolds of $\Ta\SO\times\bR^3$ and so is their union. Clearly, all 
states starting in the complement of this union, converge to the stable equilibrium $(Q,\omega,
\tilde\beta)=(I,0,0)$; therefore the domain of attraction of this equilibrium is
\[ \mbox{DOA}\{(I,0,0)\} = \Ta\SO\times\bR^3\setminus\big\{\cup_{i=1}^3 \cM^S_i\big\}, \]
which is a dense open subset of $\Ta\SO\times\bR^3$. 
\hfill\ensuremath{\square}  

\subsection{Discrete-Time Estimator Based on the Lagrange-d'Alembert Principle}\label{C6Sec6}

\subsubsection{Discrete-Time Lagrangian}

The ``energy" in the measurement residual for attitude is discretized as:
\begin{align}
\cU (\hat R_i,U^m_i)&=\Phi\Big(\cU^0(\hat R_i,U^m_i)\Big)=\Phi \Big(\frac{1}{2}\lan E_i-\hat R_i U^m_i,(E_i-\hat R_i U^m_i)W_i\ran \Big),\label{C6dattindex}
\end{align}
where 
$\Phi: [0,\infty)\mapsto[0,\infty)$ is as defined in Section \ref{C6Sec4}. 
The ``energy" in the angular velocity measurement residual is discretized as
\be \cT (\hat\Omega_i,\Omega^m_i)=\frac{m}{2}(\Omega^m_i-\hat\Omega_i-\hat\beta_i) \T
(\Omega^m_i-\hat\Omega_i-\hat\beta_i), \label{C6dangvelindex} \ee 
where $m$ is a positive scalar. 

Similar to the continuous-time attitude estimator in \cite{Automatica}, one can express 
these ``energy" terms for the case that perfect measurements (with no measurement noise) are 
available. In this case, these ``energy" terms can be expressed in terms of the state estimate 
errors $Q_i= R_i \hat R_i\T$ and $\omega_i= \Omega_i-\hat\Omega_i-\hat\beta_i$ as follows:
\begin{align}
& \cU (Q_i)= \Phi \Big(\frac{1}{2}\lan E_i - Q_i\T E_i,(E_i - Q_i\T E_i)W_i\ran \Big)=\Phi \big( \lan I-Q_i, K_i\ran \big)\,\nn \\
& \mbox{ where }\, K_i= E_i W_i E_i\T, \mbox{ and }\,\cT (\omega_i)= \frac{m}{2} \omega_i\T\omega_i\, \mbox{ where }
m>0. \label{C6discUandT} 
\end{align}
The weights in $W_i$ can be chosen such that $K_i$ is always positive definite with distinct 
(perhaps constant) eigenvalues, as in the continuous-time estimator of \cite{Automatica}. Using 
these ``energy" terms in the state estimate errors, the discrete-time Lagrangian is expressed as:
\begin{align}
\cL (Q_i,\omega_i)= \cT (\omega_i)- \cU ( Q_i) =
\frac{m}{2} \omega_i\T\omega_i- \Phi\big( \lan I-Q_i, K_i\ran \big).
\label{C6discLag}
\end{align}

\subsubsection{First-Order Discrete-Time Attitude State Estimation Based on the Discrete 
Lagrange-d'Alembert Principle}
The following statement gives a first-order discretization, in the form of a Lie 
group variational integrator, for the continuous-time estimator of \cite{Automatica}.
\begin{proposition} \label{C6discfilter}
Let discrete-time measurements for two or more inertial vectors along with angular velocity be 
available at a sampling period of $h$. Further, let the weight 
matrix $W_i$ for the set of vector measurements $E_i$ be chosen such that $K_i=E_i 
W_i E_i\T$ satisfies Lemma 2.1 in \cite{Automatica}. A discrete-time estimator obtained by 
applying the discrete Lagrange-d'Alembert principle to the Lagrangian \eqref{C6discLag} is:
\begin{align}
&\hat R_{i+1}=\hat{R_i}\exp\big(h(\Omega_i^m-\omega_i
-\hat\beta_i)^\times\big),\label{C61stDisFil_Rhat}\\
&\hat\beta_{i+1}= \hat\beta_i+ h\Phi' \big(\cU^0(\hat{R}_i,U^m_i)\big) P^{-1} S_{L_i}(\hat{R}_i), 
\label{C61stDisFil_betahat}\\
&\hat\Omega_i=\Omega_i^m-\omega_i-\hat\beta_i, \label{C61stDisFil_Omegahat}\\
&m\omega_{i+1}=\exp(-h \hat\Omega_{i+1}^\times)\Big\{(m I_{3\times3}-hD)
\omega_i+h\Phi'\big(\cU^0(\hat R_{i+1},U^m_{i+1})\big)S_{L_{i+1}}(\hat R_{i+1})\Big\},\label{C61stDisFil_omega}
\end{align}
where $S_{L_i}(\hat R_i)=\mrm{vex}(L_i\T\hat R_i-\hat R_i\T L_i)\in\bR^3$, $\mrm{vex}(\cdot): 
\so\to\bR^3$ is the inverse of the $(\cdot)^\times$ map, $L_i=E_i W_i(U^m_i)\T\in
\mathbb{R}^{3\times3}$, $\cU^0(\hat R_i,U^m_i)$ is defined in \eqref{C6dattindex} and $(\hat R_0,\hat\Omega_0)\in\SO\times\bR^3$ are initial estimated states.
\end{proposition}\par

{\em Proof}: Equation \eqref{C6angvelres} is discretize as
\begin{align}
\omega_i := \Omega^m_i-\hat\Omega_i- \hat\beta_i, \;\ \tau_{D_i}= D\omega_i, \label{C6Discangvelres}
\end{align}
and \eqref{C6biasest} can be rewritten in discrete-time as
\begin{align}
\dot{\hat\beta}_i=\frac{\hat\beta_{i+1}-\hat\beta_i}{h}=  \Phi' \big(\cU^0(\hat{R}_i,U^m_i)\big) P^{-1} S_{L_i}(\hat{R}_i), \label{C6Discbiasest}
\end{align}
 The action functional in expression \eqref{C6eq:J6} is replaced by the 
discrete-time action sum as follows:
\begin{align}
&\cS_d (\cL (\hat R_i,U^m_i,\hat\Omega_i,\Omega^m_i))=  \\ &h \sum_{i=0}^N \Big\{ 
\frac{m}{2}(\Omega^m_i-\hat\Omega_i-\hat\beta_i)\T(\Omega^m_i-\hat\Omega_i-\hat\beta_i)
- \Phi\big(\cU^0(\hat R_i,U^m_i)\big) \Big\}. \nn
\end{align}
Discretize the kinematics of the attitude estimate as
\begin{align}
\hat R_{i+1}=\hat R_i \exp( h\hat\Omega_i^\times), \label{C6diskin}
\end{align}
and consider a first variation in the discrete attitude estimate, $R_i$, of the form
\be \delta \hat R_i=\hat R_i\Sigma_i^\times,  \label{C6disVar} \ee
where $\Sigma_i\in\bR^3$ gives a variation vector for the discrete attitude estimate.
For fixed end-point variations, we have $\Sigma_0=\Sigma_N=0$. 
Further, a first order approximation is to assume that $\hat\Omega_i^\times$ and 
$\delta\hat\Omega_i^\times$ commute. With this assumption, taking the first variation of the 
discrete kinematics \eqref{C6diskin} and substituting from \eqref{C6disVar} gives:
\begin{align}
\begin{split}
\delta \hat R_{i+1} =&\delta \hat R_i \exp(h\hat\Omega_i^\times)+
\hat R_i\delta\big(\exp(h\hat\Omega_i^\times)\big) \\
=&\hat R_i\Sigma_i^\times\exp(h\hat\Omega_i^\times)+ h \hat R_i \exp(h\hat\Omega_i^\times)
\delta\hat\Omega_i^\times \\
 &=\hat R_{i+1}\Sigma_{i+1}^\times. 
\end{split}\label{C6disLagdAlm}
\end{align}
Equation \eqref{C6disLagdAlm} can be re-arranged to obtain: 
\begin{align}
h\delta\hat\Omega_i^\times&=\exp(-h\hat\Omega_i^\times)\hat R_i\T\big[\delta
\hat R_{i+1}-\hat R_i\Sigma_i^\times\exp(h\hat\Omega_i^\times)\big] \nn \\
&=\exp(-h\hat\Omega_i^\times)\hat R_i\T \hat R_{i+1}\Sigma_{i+1}^\times-
\Ad{\exp(-h\hat\Omega_i^\times)}\Sigma_i^\times \nn \\
&=\Sigma_{i+1}^\times-\Ad{\exp(-h\hat\Omega_i^\times)}\Sigma_i^\times.  \label{C6delR}
\end{align}
This in turn can be expressed as an equation on $\bR^3$ as follows:
\be h\delta\hat\Omega_i= \Sigma_{i+1} - \exp(-h\hat\Omega_i^\times)\Sigma_i, 
\label{C6delomega} \ee
since $\Ad{R}\Omega^\times=R\Omega^\times R\T=(R\Omega)^\times$.

Applying the discrete Lagrange-d'Alembert principle~\cite{marswest}, one obtains
\begin{align}
&\delta\cS_d+ h\sum_{i=0}^{N-1}\tau_{D_i}\T\Sigma_i =0\nn\\
\Rightarrow &h\sum_{i=0}^{N-1} m(\hat\Omega_i-\Omega^m_i+\hat\beta_i)\T\delta\hat\Omega_i-\Big\{\Phi'\big(\cU^0(\hat R_i,U^m_i)\big)S_{L_i}\T(\hat R_i)
-\tau_{D_i}\T\Big\}\Sigma_i=0. \label{C6disSd1}
\end{align} 
Substituting  \eqref{C6disVar} and \eqref{C6delomega} into equation \eqref{C6disSd1}, 
one obtains
\begin{align}
&\sum_{i=0}^{N-1}\Big\{m(\hat\Omega_i-\Omega^m_i+\hat\beta_i)\T\big(\Sigma_{i+1}-\exp(-h\hat\Omega_i^\times)\Sigma_i\big)\nn\\
&-h\Phi'\big(\cU^0(\hat R_i,U^m_i)\big)S_{L_i}\T(\hat R_i)\Sigma_i+h\tau_{D_i}\T\Sigma_i\Big\}=0.
\label{C6disLagdAlObs}
\end{align}
For $0\le i<N$, the expression \eqref{C6disLagdAlObs} leads to the following one-step 
first-order LGVI for the discrete-time filter:
\begin{align}
&m(\Omega^m_{i+1}-\hat\Omega_{i+1}-\hat\beta_{i+1})\T\exp(-h\hat\Omega_{i+1}^\times)+h\tau_{D_{i+1}}\T
\nn\\ &-h\Phi'\big(\cU^0(\hat R_{i+1},U^m_{i+1})\big)S_{L_{i+1}}\T(\hat R_{i+1})
+m(\hat\Omega_i-\Omega^m_i+\hat\beta_i)\T=0\nn\\
&\Rightarrow m\exp(h\hat\Omega_{i+1}^\times) (\Omega^m_{i+1}-\hat\Omega_{i+1}-\hat\beta_{i+1})=m(\Omega^m_i-\hat\Omega_i-\hat\beta_i)\nn\\
&+h\Big(\Phi'\big(\cU^0(\hat R_{i+1},U^m_{i+1})\big)S_{L_{i+1}}(\hat R_{i+1})-\tau_{D_{i+1}}\Big),
\end{align}
which after substituting $\omega_i=\Omega_i^m-\hat\Omega_i-\hat\beta_i$ and $\tau_{D_{i+1}}=D\omega_i$ gives the discrete-time equation \eqref{C61stDisFil_omega}. This equation along with \eqref{C6Discangvelres}, \eqref{C6Discbiasest} and \eqref{C6diskin} form the estimator equations \eqref{C61stDisFil_Rhat}-\eqref{C61stDisFil_Omegahat}.
\hfill\ensuremath{\square}

 Note that the estimator equations 
\eqref{C61stDisFil_Rhat}-\eqref{C61stDisFil_Omegahat} given by the LGVI scheme are in the form of an 
implicit numerical integration scheme. The discrete kinematics \eqref{C61stDisFil_Rhat} is solved 
first. Then the angular velocity estimate error is updated by solving the implicit discrete dynamics, 
equation \eqref{C61stDisFil_omega}. The stability and convergence properties of this discrete-time 
estimator are not shown here. This estimator is a first-order (in $h$) discretization of the 
continuous-time estimator given by eqs. \eqref{C6eq:filterNoise} and \eqref{C6biasest}, which is almost 
globally asymptotically stable.

\subsection{Numerical Simulation}\label{C6Sec7}
This section presents numerical simulation results of the discrete estimator 
presented in Section \ref{C6Sec6}, in the presence of constant bias in angular velocity measurements. In order to validate the performance of this estimator, a rigid body's states are artificially generated using the kinematics and dynamics equations. The rigid body moment of inertia is selected to be $J_v=\diag([2.56\;\;3.01\;\;2.98]\T)$ kg.m$^2$. Moreover, a sinusoidal function is applied to it as the only external torque, which is expressed in body fixed frame as
\begin{align}
\varphi(t)=[0\;\;\; 0.028\sin(2.7t-\frac{\pi}{7})\;\;\; 0]\T\mbox{ N.m}.
\end{align}
The rigid body is assumed to have an initial attitude and angular velocity given by,
\begin{align}
\begin{split}
R_0=&\expm_{\SO}\bigg(\Big(\frac{\pi}{4}\times[\frac{3}{7}\;\;\;\; \frac{6}{7}\;\;\;\; \frac{2}{7}]\T\Big)^\times\bigg)\\
\mbox{and }& \Omega_0=\frac{\pi}{60}\times[-2.1\;\;\;\; 1.2\;\; -1.1]\T\mbox{ rad/s}.
\end{split}
\end{align}
A set of five inertial sensors and three gyros perpendicular to each other are assumed to be onboard the rigid body. The actual states generated from kinematics and dynamics of this rigid body are used to simulate the observed directions in the body fixed frame, as well as the comparison between true and estimated states. We assume that there are five inertially known directions which are being measured 
by the five inertial sensors fixed to the rigid body at a constant sample rate. These unit vectors for the constant inertially known matrix $E$ as follows:
\begin{align}
E=\bbm&-0.6543\;\;&-0.6338\;\; &-0.5978\;\;&-0.5559\;\;&-0.5138\\
            &-0.5407   &-0.4559   &-0.4202   &-0.4253   &-0.3845\\
            &0.5287    &0.6248    &0.6827    &0.7142    &0.7669\ebm.
\end{align}
Bounded random zero mean signals whose
probability distributions are normalized bump functions are added to the true direction vectors $U$ to generate each measured 
direction $U^m$. The maximum
error (width of bump function) in each component of a
direction vector measurement is $2.4^\circ$ based on coarse attitude sensors like sun 
sensors and magnetometers. Similarly, random zero mean bump functions are added to each element of $\Omega$ to form the measured $\Omega^m$. The width of these bump functions is $0.97^\circ/s$, which corresponds to a coarse 
rate gyro. Besides, the gyro readings are assumed to contain a constant bias in three directions, as follows:
\begin{align}
\beta=[-0.01\;\;\;-0.005\;\;\;\;\;0.02]\T\mbox{ rad/s}.
\end{align}
The estimator is simulated over a time interval of $T=20$s, with a time stepsize of $h=0.01$s. The estimator's inertia scalar gain is $m=5$ and the dissipation matrix is selected as the following positive definite matrix:
\begin{align}
D=\diag\big([17.04\;\;\; 18.46\;\;\; 19.88]\T\big).
\end{align}
As in \cite{Automatica}, $\Phi(x)=x$. The weight matrix $W$ is also calculated using the conditions in 
\cite{Automatica}. This matrix is given by:
\begin{align}
W=\bbm&296.5458\;\; &-296.8526\;\; &-293.3936\;\; &150.4527\;\; &150.2987\\
            &-296.8526  &368.7300 &341.0189 &-197.1644 &-221.0503\\
            &-293.3936  &341.0189 &321.6729 &-179.3406 &-194.9746\\
            &150.4527  &-197.1644  &-179.3406 &107.4149  &123.2687\\
            &150.2987 &-221.0503 &-194.9746  &123.2687  &147.3057\ebm.
\end{align}
The positive definite matrix for bias gain is selected as $P=4\times10^3I$. The initial estimated states are equal to:
\begin{align}
\begin{split}
\hat R_0&=\expm_{\SO}\bigg(\Big(\frac{\pi}{2.5}\times[\frac{3}{7}\;\;\;\; \frac{6}{7}\;\;\;\; \frac{2}{7}]\T\Big)^\times\bigg),\\
\hat\Omega_0&=[-0.26\;\;\;\;\; 0.1725\;\;\; -0.2446]\T\mbox{ rad/s},\\
&\mbox{and }\hat\beta_0=[0\;\;\;-0.01\;\;\;\;\;0.01]\T\mbox{ rad/s}.
\end{split}
\end{align}

In order to integrate the implicit set of equations in \eqref{C61stDisFil_Rhat}-\eqref{C61stDisFil_omega} numerically, the first two equation are solved at each sampling step. Using \eqref{C61stDisFil_Omegahat}, $\hat\Omega_{i+1}$ in \eqref{C61stDisFil_omega} is written in terms of $\omega_{i+1}$ next. The resulting implicit equation is solved with 
respect to $\omega_{i+1}$ iteratively to a set tolerance applying the Newton-Raphson method. The root of this nonlinear equation 
along with $\hat R_{i+1}$ and $\hat\beta_{i+1}$ are used for the next sampling time instant. This process is repeated to the end of the simulated duration.\par
Results from this numerical simulation are shown here. The principal angle corresponding 
to the rigid body's attitude estimation error is depicted in Fig. \ref{C6Fig1}, and estimation errors in the 
rigid body's angular velocity components are shown in Fig. \ref{C6Fig2}. Finally, Fig. \ref{C6Fig3} portrays estimate errors in bias components. All the estimation errors are 
seen to converge to a neighborhood of $(Q,\omega,\tilde\beta)=(I,0,0)$, where the size of this neighborhood 
depends on the bounds of the measurement noise.

\begin{figure}
\begin{center}
\includegraphics[height=3.2in]{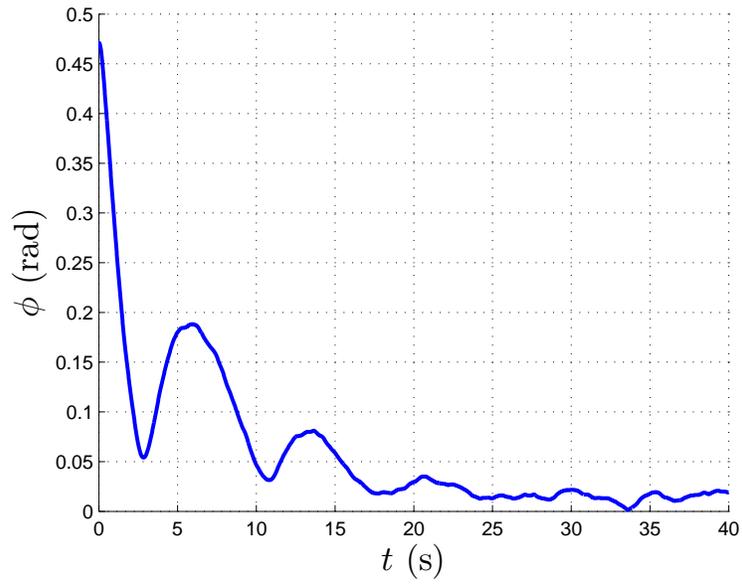}
\caption{Principal angle of the attitude estimate error}  
\label{C6Fig1}                                 
\end{center}                                 
\end{figure}

\begin{figure}
\begin{center}
\includegraphics[height=3.2in]{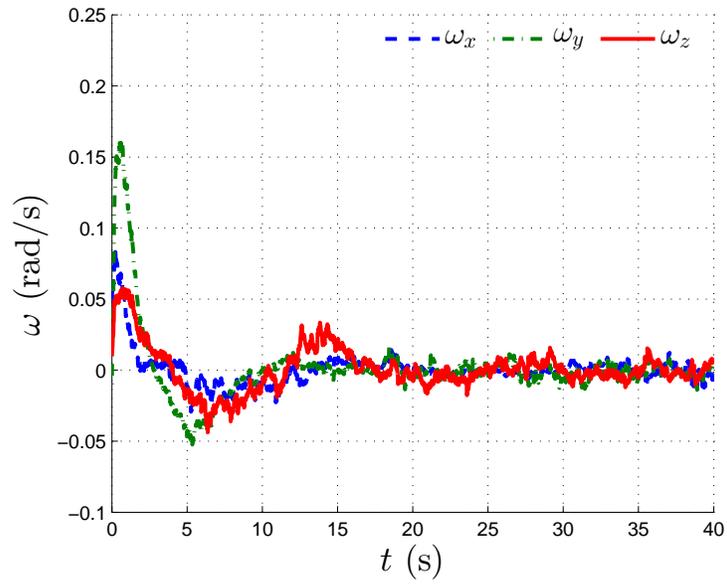}
\caption{Angular velocity estimate error}  
\label{C6Fig2}                                 
\end{center}                                 
\end{figure}

\begin{figure}
\begin{center}
\includegraphics[height=3.2in]{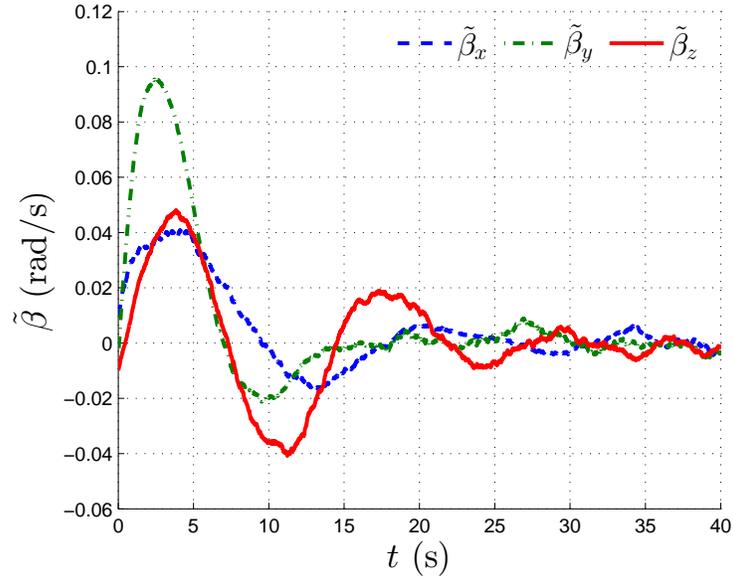}
\caption{Bias estimate error}  
\label{C6Fig3}                                 
\end{center}                                 
\end{figure}

\subsection{Conclusion}\label{C6Sec8}
The formulation of variational attitude estimation is generalized to include bias in angular 
velocity measurements and estimate a constant bias vector. The continuous-time state 
estimator is obtained by applying the Lagrange-d'Alembert principle of variational mechanics 
to a Lagrangian consisting of the energies in the measurement residuals, along with a 
dissipation term linear in angular velocity measurement residual. The update law for the 
bias estimate ensures that the total energy content in the state and bias estimation errors 
is dissipated as in a dissipative mechanical system. The resulting generalization of the 
variational attitude estimator is almost globally asymptotically stable, like the variational attitude 
estimator for the bias-free case reported in~\cite{Automatica}. A discretization of this estimator 
is obtained, in the form of an implicit first order Lie group variational integrator, by applying the 
discrete Lagrange-d'Alembert principle to the discrete Lagrangian with the dissipation term 
linear in the angular velocity estimation error. Using a realistic set of data for rigid body 
rotational motion, numerical simulations show that the estimated states and estimated bias 
converge to a bounded neighborhood of the true states and true bias when the measurement 
noise is bounded. A future extension of this work will be the formulation of an explicit 
discrete-time implementation of this variational attitude estimation in the presence of bias, 
and its real-time implementation with optical and inertial sensors.

\newpage

\section{EXPERIMENTAL VALIDATION OF THE VARIATIONAL ATTITUDE ESTIMATOR} 

\label{CH07_ICC2015} 

\hspace{\parindent}
\textit{This chapter is adapted from a paper published in Proceedings of the 2015 Indian Control Conference \cite{ICC2015}. The author gratefully acknowledges Dr. Amit K. Sanyal and S.P. Viswanathan for their participation.}
\\
\\
{\bf{Abstract}}~
The attitude determination (estimation) scheme presented in Chapter \ref{CH04_VE} is experimentally verified here. Implementing this variational estimation scheme on an Android cellphone and using the data from its ``onboard'' sensors, the cellphone's attitude is determined. This attitude is compared against the attitude derived from solving the Wahba's problem at each time instant to show the performance of the estimator. These results, obtained in the Spacecraft Guidance, Navigation and Control Laboratory 
at NMSU, demonstrate the excellent performance of this estimation scheme with the 
noisy raw data from the smartphone sensors.

\subsection{Definitions}
The raw IMU measurements from the smartphone are fused/filtered through the variational attitude estimation.
In Chapter \ref{CH04_VE}, an estimation of rigid body attitude and angular velocity without any knowledge of the attitude 
dynamics model, is presented using the Lagrange-d'Alembert principle from variational 
mechanics. This variational observer requires at least two body-fixed sensors to measure inertially known and constant direction vectors as well as sensors to read the angular velocity. First- and second-order Lie group variational integrators were introduced for computer implementation using discrete variational mechanics.

In order to determine three-dimensional rigid body attitude instantaneously, three known inertial vectors are needed. This could be satisfied with just two vector measurements. In this case, the cross product 
of the two measured vectors is considered as a third measurement for applying the attitude 
estimation scheme. Let these vectors be denoted as $u_1^m$ and $u_2^m$, in the body-fixed frame. Denote the corresponding known inertial vectors as seen from the rigid body as $e_1$ and $e_2$, and 
let the true vectors in the body frame be denoted $u_i=R \T e_i$ for $i=1,2$, where $R$ is the rotation 
matrix from the body frame to the inertial frame. This rotation matrix provides a coordinate-free,  
global and unique description of the attitude of the rigid body. Define the matrix composed 
of all three measured vectors expressed in the body-fixed frame as column vectors, $U^m= [u_1^m\ u_2^m\ u_1^m\times u_2^m]$ 
and the corresponding matrix of all these vectors expressed in the inertial frame as
$E= [e_1\ e_2\ e_1\times e_2]$.
Note that the matrix of the actual body vectors $u_i$ corresponding to the inertial vectors 
$e_i$, is given by $U= R \T E= [u_1\ u_2\ u_1\times u_2]$. The symmetric second-order filter equations \eqref{2ndDisFil_Rhat}-\eqref{iAV} has been used in order to verify its performance in practice.

\subsection{Experiments}
This estimation scheme is implemented off-board on a remote PC using the sensor measurements acquired and transmitted by the smartphone. The coordinates used for the inertial frame is ENU, which is a right-handed Cartesian frame formed by local east, north and up. The coordinates fixed to the COM of the cellphone with right direction of the screen as $x$, upwards direction as $y$ and the direction out of screen as $z$ is considered to be the body fixed frame. As mentioned in the previous section, at least two inertially known and constant directions are required in order to estimate the rigid body attitude. Using the inertial sensors installed on the smartphone, the accelerometer is used to measure the gravity direction and the magnetometer is used to measure the geomagnetic field direction. The cross product of these two vectors is considered as the third measured vector. In order to find these directions, one could normalize the vector readings from accelerometer and magnetometer in the case that the cellphone is aligned with the true geographical directions and the body fixed frame coincides with the ENU frame. Note that the direction read by the accelerometer shows the local up direction, since an upward acceleration equal to $g$ is applied to the phone in order to cancel the Earth's gravity and keep the phone still. Therefore, the matrix of these three inertially constant directions as expressed in the ENU frame is found to be
$$E=[e_1\ e_2\ e_1\times e_2]=
\begin{bmatrix}
0 & 0.0772 & -0.9921\\
0 & 0.6117 & 0.1251\\
1 & -0.7873 & 0
\end{bmatrix}.
$$
The three axis gyroscope also gives the angular velocity measurements. These three sensors  produce measurement data at different frequencies. The filter's time step is selected according to the fastest sensor, which is the accelerometer here. At those time instants where some of the sensors readings are not available because of the difference in sampling frequencies, the last read value 
from that sensor is used.

$\Phi(\cdot)$ could be any $C^2$ function with the properties described in Section 2 of \cite{Automatica}, but is 
selected to be $\Phi(x)=x$ here. Further, $W$ is selected based on the value of $E$, such that it satisfies the conditions in \cite{Automatica} as below:
$$W=
\begin{bmatrix}
3.19 & 1.51 & 0\\
1.51 & 3.19 & 0\\
0 & 0 & 2
\end{bmatrix}.
$$
The inertia scalar gain is set to $m=0.5$ and the dissipation matrix is selected as the following positive definite matrix:
\[D=\diag\big([12\ 13\ 14]\T\big).\]

Sensors outputs usually contain considerable levels of noise that may harm the behavior of the nonlinear filter. A Butterworth pre-filter is implemented in order to reduce these high-frequency noises. Note that the true quantities would not contain high-frequency signals, since they are related to a rigid body motion. A symmetric discrete-time filter for the first order Butterworth pre-filter is implemented for filtering the measurement data as follows:
\begin{align}
(2+ h)\bar{x}_{k+1}= (2-h)\bar{x}_k + h (x^m_k + x^m_{k+1}),
\end{align}
where $h$ is the time stepsize, $\bar{x}$ and $x^m$ are the filtered and measured quantities, respectively, and the subscript $k$ denotes the $k$th time stamp.
The initial estimated states have the following initial estimation errors:
\begin{align}
Q_0=\expm&_{\SO}\bigg(\Big(2.2\times[0.63\;\;\; 0.62\ -0.48]\T\Big)^\times\bigg), \nn\\
\mbox{and } \omega_0&=[0.001\ 0.002\ -0.003]\T\mbox{ rad/s}.
\end{align}
In order to integrate the implicit set of equations \eqref{2ndDisFil_Rhat}-\eqref{iAV} numerically, the first equation is solved at each sampling step, then the result for $\hat R_{i+1}$ is substituted in the second one. Using the Newton-Raphson method, the resulting equation is solved with respect to $\omega_{i+1}$ iteratively. The root of this nonlinear equation with a specified tolerance along with the $\hat R_{i+1}$ is used for the next sampling time instant. This process is repeated over the simulated time period. The results of this experiment are described next.

\subsection{Results}
Experimental results for the attitude estimation scheme, obtained from the experimental 
setup described in the previous subsection, are presented here. These experiments were 
carried out on the HIL simulator testbed in the Spacecraft Guidance, Navigation and Control 
laboratory at NMSU's MAE department. The principal angle corresponding to the rigid body's attitude estimation error is depicted in Figure
\ref{fig1}. Estimation errors in the rigid body's angular velocity components are shown in 
Figure \ref{fig2}. All the estimation errors are seen to converge to a neighborhood 
of $(Q,\omega)=(I,0)$, where the size of this neighborhood depends on the characteristics of the 
measurement noise.
\begin{figure}
\begin{center}
\includegraphics[height=2.2in]{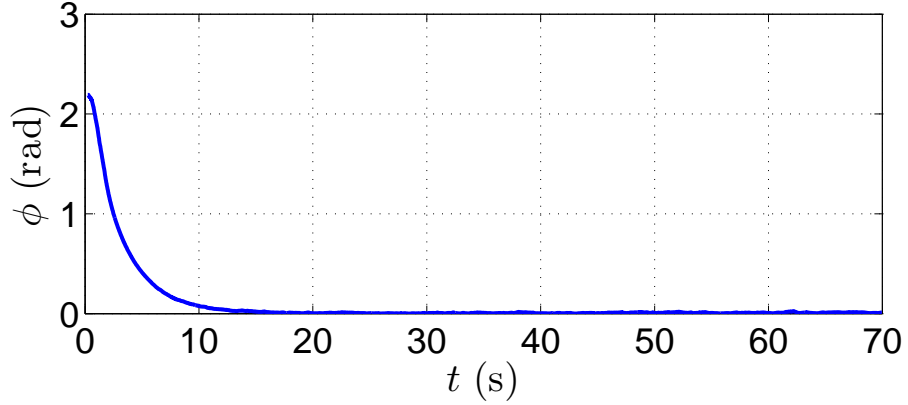}
\caption{Principle Angle of the Attitude Estimation Error}  
\label{fig1}                                 
\end{center}                                 
\end{figure}

\begin{figure}
\begin{center}
\includegraphics[height=2.2in]{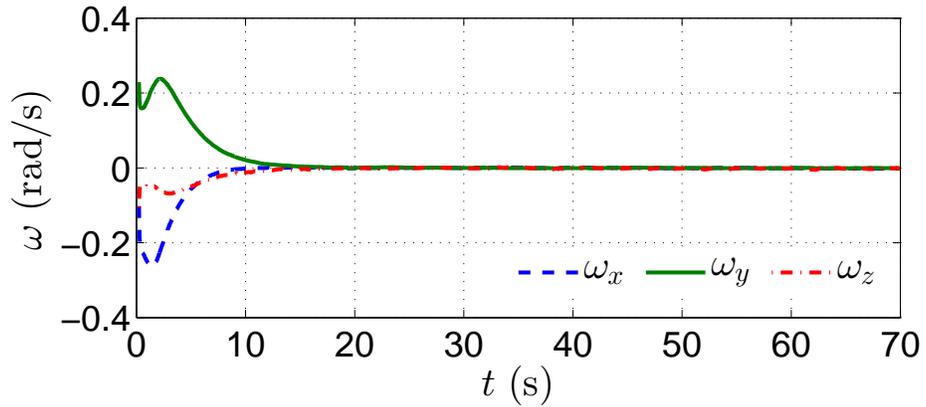}
\caption{Angular Velocity Estimation Error}  
\label{fig2}                                 
\end{center}                                 
\end{figure}

\subsection{Conclusion}
This chapter presents a novel software architecture of a spacecraft 
attitude determination and control subsystem (ADCS), using a smartphone as the onboard 
computer. This architecture is being implemented using a HIL ground simulator for three-axis 
attitude motion simulation in the Spacecraft Guidance, Navigation and Control laboratory at 
NMSU. Theoretical and numerical results for the attitude control and attitude estimation schemes 
that are part of this architecture, have appeared in recent publications. The attitude estimation scheme provides almost global asymptotic stability, and is 
robust to measurement noise and bounded disturbance inputs acting on the spacecraft. 
Experimental verification of the attitude estimation algorithm is presented here, and the 
experimental results show excellent agreement with the theoretical and numerical results 
on this algorithm that have appeared in recent publications.

\newpage

\section{MODEL-FREE RIGID BODY POSE ESTIMATION BASED ON THE LAGRANGE-D'ALEMBERT PRINCIPLE} 

\label{CH08_Automatica2} 

\hspace{\parindent}
\textit{This chapter is adapted from a paper to appear in Automatica \cite{Automatica2}. The author gratefully acknowledges Dr. Amit K. Sanyal for his participation.}
\\
\\
{\bf{Abstract}}~
Stable estimation of rigid body pose and velocities from noisy measurements, without any 
knowledge of the dynamics model, is treated using the Lagrange-d'Alembert principle from 
variational mechanics.
With body-fixed optical and inertial sensor measurements, a Lagrangian is obtained as the 
difference between a kinetic energy-like term that is quadratic in velocity estimation error and 
the sum of two artificial potential functions; one obtained from a generalization of Wahba's 
function for attitude estimation and another which is quadratic in the position estimate error. An 
additional dissipation term that is linear in the velocity estimation error is introduced, and the 
Lagrange-d'Alembert principle is applied to the Lagrangian with this dissipation. A Lyapunov 
analysis shows that the state estimation scheme so obtained provides stable asymptotic 
convergence of state estimates to actual states in the absence of measurement noise, with an 
almost global domain of attraction. This estimation scheme is discretized for computer 
implementation using discrete variational mechanics, as a first order Lie group variational 
integrator. The continuous and discrete pose estimation schemes require optical measurements 
of at least three inertially fixed landmarks or beacons in order to estimate instantaneous pose. 
The discrete estimation scheme can also estimate velocities from such optical measurements.
Moreover, all states can be estimated during time periods when measurements of only two 
inertial vectors, the angular velocity vector, and one feature point position vector are available 
in body frame. In the presence of bounded measurement noise in the vector measurements, 
numerical simulations show that the estimated states converge to a bounded neighborhood 
of the actual states.

\subsection{Navigation using Optical and Inertial Sensors}\label{Sec2}

Consider a vehicle in spatial (rotational and translational) motion. 
\begin{figure}[htb!]
	\centering
	\includegraphics[width=0.7\textwidth]{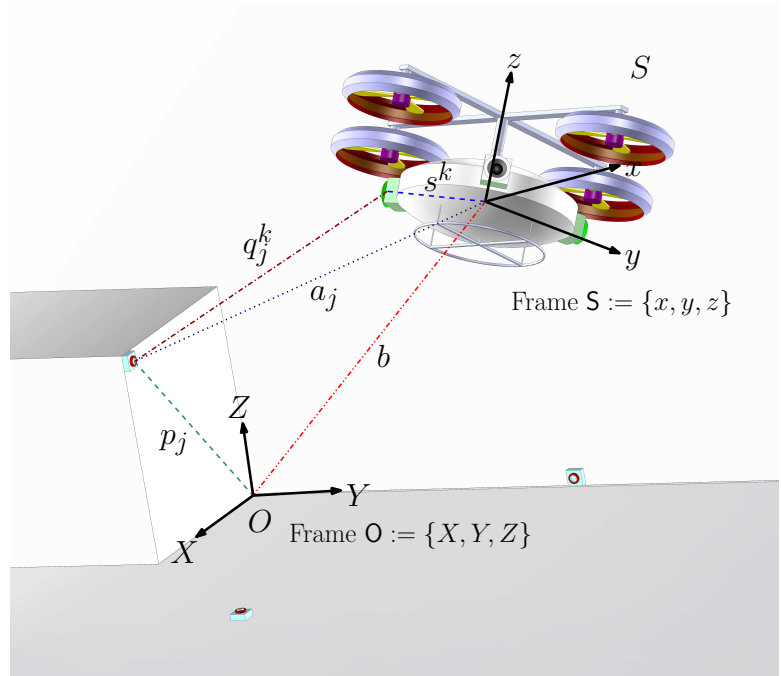}
    	\caption{Inertial landmarks on $O$ as observed from vehicle $S$ with optical  
	measurements.}
    	\label{Frames}
\end{figure}
Onboard estimation of the pose of the vehicle involves assigning a coordinate frame fixed to 
the vehicle body, and another coordinate frame fixed in the environment which takes the role 
of the inertial frame. Let $O$ denote the observed environment and $S$ denote the vehicle. 
Let $\mathsf{S}$ denote a coordinate frame fixed to $S$ and $\sO$ be a coordinate frame 
fixed to $O$, as shown in Fig. \ref{Frames}. Let $R\in\SO$ denote the rotation matrix from frame 
$\sS$ to frame $\sO$ and $b$ denote the position of origin of $\sS$ expressed in frame $\sO$.
The pose (transformation) from body fixed frame $\sS$ to inertial frame $\sO$ is then given by \eqref{gDef}. Consider vectors known in inertial frame $\sO$ measured by inertial sensors 
in the vehicle-fixed frame $\sS$; let $\beta$ be the number of such vectors. In addition, consider 
position vectors of a few stationary points in the inertial frame $\sO$ measured by optical (vision 
or lidar) sensors in the vehicle-fixed frame $\sS$. 
Velocities of the vehicle may be directly measured or can be estimated by linear 
filtering of the optical position vector measurements \cite{ACC2015}. 
Assume that these optical measurements are available for 
$\mpz j$ points at time $t$, whose positions are known in frame $\sO$ as $p_j$, $j\in\cI (t)$, 
where $\cI (t)$ denotes the index set of beacons observed at time $t$. Note that the observed 
stationary beacons or landmarks may vary over time due to the vehicle's motion. These points 
generate ${\mpz j\choose 2}$ unique relative position vectors, which are the vectors connecting 
any two of these landmarks. When two or more position vectors are optically measured, the number of vector measurements that
can be used to estimate attitude is ${\mpz j\choose 2}+\beta$. This number needs to be at least two 
(i.e., ${\mpz j\choose 2}+\beta\geq2$) at an instant, for the attitude to be uniquely determined at 
that instant. In other words, if at least two inertial vectors are measured at all instants (i.e., $\beta\geq2$), 
then beacon position measurements are not required for estimating attitude. However, at least one 
beacon or feature point position measurement is still required to estimate the position of the vehicle. 
Note that the use of two vector measurements for attitude determination was first proposed by the 
TRIAD algorithm in the 1960s \cite{TRIAD}.

\subsubsection{Pose Measurement Model} 

Denote the position of an optical sensor and the unit vector from that sensor to an observed  
beacon in frame $\sS$ as $s^k\in\bR^3$ and $u^k\in\bS^2$, $k=1,\ldots,\mpz k$, respectively. 
Denote the relative position of the $j^{th}$ stationary beacon observed by the $k^{th}$ 
sensor expressed in frame $\sS$ as $q^k_j$. Thus, in the absence of measurement noise
\begin{align}
p_j=R(q^k_j+s^k)+b=Ra_j+b,\; j\in\mathcal I(t),
\label{FrameTrans}
\end{align}
where $a_j=q^k_j+s^k$, are positions of these points expressed in $\sS$. In practice, the $a_j$ 
are obtained from range measurements that have additive noise; we denote as $a_j^m$ the 
measured vectors. In the case of lidar range measurements, these are given by
\be a_j^m=(q^k_j)^m+s^k=(\varrho_j^k)^m u^k+s^k,\; j\in\mathcal I(t), \label{pim} \ee
where $(\varrho_j^k)^m$ is the measured range to the point by the $k^{th}$ sensor. 
The mean of the vectors $p_j$ and $a_j^m$ are denoted as $\bar{p}$ and $\bar{a}^m$ respectively, 
and satisfy
\begin{align}
\bar{a}^m= R\T (\bar p- b)+ \bar\varsigma, \label{barp}
\end{align}
where $\bar{p}=\frac{1}{\mpz j}\sum\limits^\mpz j_{j=1}p_j$, $\bar{a}^m=\frac{1}{\mpz j}\sum
\limits^\mpz j_{j=1}a_j^m$ and $\bar\varsigma$ is the additive measurement noise 
obtained by averaging the measurement noise vectors for each of the $a_j$. Consider the 
${\mpz j\choose 2}$ relative position vectors from optical measurements, denoted as $d_j=
p_\lambda-p_\ell$ in frame $\sO$ and the corresponding vectors in frame $\sS$ as $l_j=
a_\lambda-a_\ell$, for $\lambda,\ell\in\cI (t)$, $\lambda\ne \ell$. The $\beta$ measured inertial 
vectors are included in the set of $d_j$, and their corresponding measured values expressed 
in frame $\sS$ are included in the set of $l_j$. If the total number of measured vectors (both 
optical and inertial), ${\mpz j\choose 2}+\beta=2$, then $l_3=l_1\times l_2$ is considered 
a third measured direction in frame $\sS$ with corresponding vector $d_3=d_1\times d_2$ in 
frame $\sO$. Therefore,
\begin{align}
d_j=Rl_j\Rightarrow D=RL,
\end{align}
where $D=[d_1\;\, \cdots\;\, d_n]$, $L=[l_1\;\, \cdots\;\, l_n]\in\bR^{3\times n}$ with $n=3$ if 
${\mpz j\choose 2}+\beta=2$ and $n={\mpz j\choose 2}+\beta$ if ${\mpz j\choose 2}+\beta>2$. 
Note that the matrix $D$ consists of vectors known in frame $\sO$. Denote the measured value 
of matrix $L$ in the presence of measurement noise as $L^m$. Then,
\begin{align}
L^m=R\T D+\mathscr{L},
\label{VecMeasMod}
\end{align}
where $\mathscr{L}\in\bR^{3\times n}$ consists of the additive noise in the vector measurements made
in the body frame $\sS$. 

\subsubsection{Velocities Measurement Model} 
Denote the angular and translational velocity of the rigid body expressed in body 
fixed frame $\sS$ by $\Omega$ and $\nu$, respectively. Therefore, one can write the kinematics 
of the rigid body as
\begin{align}
\dot{\Omega}=R\Omega^\times,\dot{b}=R\nu\Rightarrow\dot{\msg}= \msg\xi^\vee,
\label{Kinematics}
\end{align}
where $\xi= \bbm \Omega\\ \nu\ebm\in\bR^6$ and $\xi^\vee=\bbm \Omega^\times &\; \nu\\ 0 \;\;& 0\ebm$.
For the general development of the motion estimation scheme, it is assumed that the velocities 
are directly measured. The estimator is then extended to cover the cases where: (i) only angular velocity is 
directly measured; and (ii) none of the velocities are directly measured.

\subsection{Dynamic Estimation of Motion from Proximity Measurements}\label{Sec3}
In order to obtain state estimation schemes from measurements as outlined in Section 
\ref{Sec2} in continuous time, the Lagrange-d'Alembert principle is applied to an action functional of 
a Lagrangian of the state estimate errors, with a dissipation term linear 
in the velocities estimate error. This section presents the estimation scheme obtained 
using this approach. Denote the estimated pose and its kinematics as
\begin{align}
\hat\msg=\bbm \hat R & \;\;\; \hat{b}\\ 0 & \;\;\; 1\ebm\in\SE, \;\;\dot{\hat\msg}=\hat\msg\hat\xi^\vee,
\end{align}
where $\hat\xi$ is rigid body velocities estimate, with $\hat\msg_0$ as the initial pose estimate and the pose estimation 
error as
\begin{align}
\msh=\msg\hat\msg^{-1}=\bbm Q &\;\;\;\; b-Q\hat{b}\\ 0 & \;\;1\ebm=\bbm Q &\;\;\; x\\ 0 &\;\;\;1\ebm\in\SE,
\end{align}
where $Q=R\hat{R}\T$ is the attitude estimation error and $x=b-Q\hat{b}$. Then one obtains, in the case of perfect measurements,
\begin{align}\begin{split}
&\dot\msh = \msh\varphi^\vee,\, \mbox{ where }\, \varphi(\hat\msg,\xi^m,\hat\xi)=\bbm\omega\\\upsilon\ebm= \Ad{\hat\msg}\big(\xi^m- \hat\xi),
\end{split}\label{hdot}
\end{align}
where $\Ad{\mpz{g}}=\bbm \mpz{R}~~&~~0\\ \mpz{b}^\times\mpz{R}~~&~~\mpz{R}\ebm$ for $\mpz{g}=\bbm \mpz{R} &\;\;\; \mpz{b}\\ 0 & \;\;\;1\ebm$. The attitude and position estimation error dynamics are also in the form
\begin{align}
\dot{Q}=Q\omega^\times,\;\;\dot{x}=Q\upsilon.
\label{Qdot}
\end{align}

\subsubsection{Lagrangian from Measurement Residuals}
Consider the sum of rotational and translational measurement residuals between the 
measurements and estimated pose as a potential energy-like function. The rotational potential function (Wahba's cost function~\cite{jo:wahba}) is expressed as
\begin{align}
\cU^0_r (\hat{\msg},L^m,D) &= \frac12\lan D -\hat R L^m , (D -\hat R L^m)W\ran,
\label{U0r}
\end{align}
where $W=\diag(w_j)\in\bR^{n\times n}$ is a positive diagonal matrix of weight factors for the 
measured $l_j^m$. Consider the translational potential function
\begin{align}
\cU_t (\hat \msg,\bar a^m,\bar p) &= \frac12 \kappa y\T y= \frac12\kappa \|\bar{p}-\hat{R}\bar{a}^m-\hat{b}\|^2,
\label{U0t} 
\end{align}
where $\bar{p}$ is defined by \eqref{barp}, $y\equiv y(\hat\msg,\bar{a}^m,\bar p)=\bar{p}-
\hat{R}\bar{a}^m-\hat{b}$ and $\kappa$ is a positive scalar. Therefore, the total potential function is 
defined as the sum of the generalization of \eqref{U0r} defined in~\cite{Automatica,ast_acc14}  
for attitude determination on $\SO$, and the translational energy \eqref{U0t} as
\begin{align}
\cU(\hat{\msg},L^m,D,\bar a^m,\bar p)&=\cU_r (\hat{\msg},L^m,D)+\cU_t(\hat{\msg},\bar a^m,\bar p)\label{costU}\\
&= \Phi \big( \cU^0_r (\hat{\msg},L^m,D) \big)+\cU_t(\hat{\msg},\bar a^m,\bar p)\nn\\
&=\Phi \big(\frac12\lan D -\hat R L^m , (D -\hat R L^m)W\ran\big)+\frac12\kappa \|\bar{p}-\hat{R}\bar{a}^m-\hat{b}\|^2,\nn
\end{align}
where $W$ is positive definite (not necessarily diagonal), and 
$\Phi: [0,\infty)\mapsto[0,\infty)$ is a $\mC^2$ function that satisfies $\Phi(0)=0$ and 
$\Phi'(\mpz x)>0$ for all $\mpz x\in[0,\infty)$. Furthermore, $\Phi'(\cdot)\leq\alpha(\cdot)$ where 
$\alpha(\cdot)$ is a Class-$\mathcal{K}$ function~\cite{khal} and $\Phi'(\cdot)$ denotes the 
derivative of $\Phi(\cdot)$ with respect to its argument. Because of these properties 
of the function $\Phi$, the critical points and their indices coincide for $\cU^0_r$ and 
$\cU_r$~\cite{Automatica}. Define the kinetic energy-like function: 
\be
\cT \Big(\varphi(\hat\msg,\xi^m,\hat\xi)\Big)= \frac12 \varphi(\hat\msg,\xi^m,\hat\xi)\T \bJ\varphi(\hat\msg,\xi^m,\hat\xi), 
\label{costT} \ee
where $\bJ\in\bR^{6\times 6}>0$ is an artificial inertia-like kernel matrix. Note that in contrast to 
rigid body inertia matrix, $\bJ$ is not subject to intrinsic physical constraints like the triangle 
inequality, which dictates that the sum of any two eigenvalues of the inertia matrix has to be 
larger than the third. Instead, $\bJ$ is a gain matrix that can be used to tune the estimator. For 
notational convenience, $\varphi(\hat\msg,\xi^m,\hat\xi)$ is denoted as $\varphi$ from 
now on; this quantity is the velocities estimation error in the absence of measurement 
noise. Now define the Lagrangian 
\be
\cL (\hat{\msg},L^m,D,\bar a^m,\bar p,\varphi)= \cT(\varphi) -\cU(\hat{\msg},L^m,D,\bar a^m,\bar p),
\label{contLag}\ee
and the corresponding action functional over an arbitrary time interval $[t_0,T]$ for $T>0$,
\be \cS \big(\cL (\hat{\msg},L^m,D,\bar a^m,\bar p,\varphi)\big)= \int_{t_0}^T \cL 
(\hat{\msg},L^m,D,\bar a^m,\bar p,\varphi) \di t, \,\label{action} \ee
such that  $\dot{\hat \msg}= \hat \msg(\hat\xi)^\vee$. The following statement gives the form of 
the Lagrangian when perfect (noise-free) measurements are available, and derives the 
variational estimator for rigid body pose and velocities.
\begin{lemma}\label{NoiseFree}
In the absence of measurement noise, the Lagrangian is of the form
\begin{align}
\cL (\msh,D,\bar p,\varphi)=&\frac12 \varphi\T \bJ\varphi -\Phi \big(\lan I-Q, K\ran\big)-\frac12\kappa y\T y,\label{NFContLag}
\end{align}
where $K=DWD\T$ and $y\equiv y(\msh,\bar p)=Q\T x+(I-Q\T)\bar{p}$.
\end{lemma}
{\em Proof}: Suppose that all the measured states are noise free. Therefore, one can replace 
$L^m=L$, $\bar a^m=\bar a$ and $\xi^m=\xi$. The rotational potential function \eqref{U0r} can be 
replaced by
\begin{align}
\cU^0_r (\msh,D) &=\frac12\lan D-\hat{R}L^m, (D-\hat{R}L^m)W\ran=\frac12\lan D-Q\T D, (D-Q\T D)W\ran,\nn\\
&=\frac12\lan I-Q\T, (I-Q\T)DWD\T\ran=\lan I-Q, K\ran\label{NFU0r}
\end{align}
since $\hat{R}E=Q\T D$ for the noise-free case. In addition,
\begin{align}
y(\msh, \bar p)&=\bar{p}-\hat{R}\bar{a}^m-\hat{b}=\bar{p}-\hat{R}\bar{a}-\hat{b}\\
&=\bar{p}-Q\T R\bar{a}-Q\T(b-x)=Q\T x+(I-Q\T)\bar{p}. \nn
\end{align}
The translational potential function in the absence of measurement noise can be expressed as
\begin{align}
\cU_t (\msh,\bar p) =\frac12\kappa y\T y.
\label{NFU0t}
\end{align}
Therefore, the total potential energy function is
\begin{align}
\cU(\msh,D,\bar p)&=\cU_r (\msh,D)+\cU_t(\msh,\bar p)= \Phi \big( \cU^0_r (\msh,D) \big)+\cU_t(\msh,\bar p) \nn\\
&=\Phi \big(\lan I-Q, K\ran\big)+\frac12\kappa y\T y,
\label{NFcostU}
\end{align}
and the kinetic energy function is
\begin{align}
\cT(\varphi)= \frac12 \varphi\T \bJ\varphi.
\label{NFcostT}
\end{align}
Substituting \eqref{NFcostU} and \eqref{NFcostT} into:
\begin{align}
\cL (\msh,D,\bar{p},\varphi)&= \cT (\varphi) -\cU(\msh,D,\bar p)=\cT (\varphi) -\Phi\big(\cU^0_r(\msh,D)\big)-\cU_t(\msh,\bar p),
\label{NFLag}
\end{align}
gives the Lagrangian \eqref{NFContLag} for the noise-free case.
\hfill\ensuremath{\square}

As in \cite{Automatica}, the positive definite weight matrix $W$ can be selected according to the 
following lemma:
\begin{lemma}\label{LemmaW}
Let $\mbox{rank}(D)=3$. Let the singular value decomposition of $D$ be given by
\begin{align}
D :&= U_D \Sigma_D V_D\T\, \mbox{ where }\,U_D\in\mathrm{O}(3),\ V_D\in\mathrm{O}(n),\nn\\
&\Sigma_D\in\mathrm{Diag}^+(3,n),
\end{align}
and $\mathrm{Diag}^+(n_1,n_2)$ is the vector space of $n_1\times n_2$ matrices with 
positive entries along the main diagonal and all other components zero. Let $\sigma_1, 
\sigma_2, \sigma_3$ denote the main diagonal entries of $\Sigma_D$. Further, let the positive 
definite weight matrix $W$ be given by 
\be W= V_D W_0 V_D\T\, \mbox{ where }\, W_0\in\mathrm{Diag}^+(n,n) \label{Wdec} \ee
and the first three diagonal entries of $W_0$ are given by
\be w_1= \frac{\varsigma_1}{\sigma_1^2},\; w_2=\frac{\varsigma_2}{\sigma_2^2},\; w_3=\frac{\varsigma_3}
{\sigma_3^2}\, \mbox{ where }\, \varsigma_1,\varsigma_2,\varsigma_3>0. \label{w123} \ee
Then, $K=DWD\T$ is positive definite and 
\be K= U_D\Delta U_D\T\, \mbox{ where }\, \Delta=
\diag(\varsigma_1,\varsigma_2,\varsigma_3), \ee
is its eigendecomposition. Moreover, if $\varsigma_\imath\ne \varsigma_\jmath\mbox{ for } \imath\ne \jmath$ and $\imath,\jmath\in\{1,2,3\}$, then $\lan I- Q,K\ran$ is a Morse function whose critical points are
\begin{align}
Q\in C_Q=\big\{ I, Q_1, Q_2, Q_3\big\} \mbox{ where } Q_\imath= 2 U_D I_\imath I_\imath\T U_D\T - I,
\label{C_Qdef}
\end{align}
and $I_\imath$ is the $\imath^{th}$ column vector of the identity $I\in\SO$.
\end{lemma}
The proof is presented in \cite{Automatica}.


\subsubsection{Variational Estimator for Pose and Velocities}
The nonlinear variational estimator obtained by 
applying the Lagrange-d'Alembert principle to the Lagrangian \eqref{contLag} with a dissipation 
term linear in the velocities estimation error, is given by the following statement.

\begin{theorem}\label{filterTHM}
The nonlinear variational  estimator for pose and velocities is given by
\begin{align}
\begin{cases}
\bJ\dot{\varphi}&=\adast{\varphi}\bJ\varphi-Z(\hat{\msg},L^m,D,\bar{a}^m,\bar{p})-\bD\varphi,\vspace{.05in}\\
\hat{\xi}&=\xi^m-\Ad{\hat\msg^{-1}}\varphi,
\vspace{.05in}\\
\dot{\hat{\msg}}&=\hat{\msg}(\hat{\xi})^\vee,\label{ContFil}
\end{cases}
\end{align}
where $\adast{\zeta}=(\ad{\zeta})\T$ with $\ad{\zeta}$ defined by \eqref{ad_def}, and 
$Z(\hat{\msg},L^m,D,\bar{a}^m,\bar{p})$ is defined by
\begin{align}
\begin{split}
Z(\hat{\msg},L^m,D,&\bar{a}^m,\bar{p})=\bbm \Phi'\Big(\cU^0_r (\hat{\msg},L^m,D)\Big)S_\Gamma(\hat{R})+\kappa\bar{p}^\times y\\ \kappa y\ebm,\end{split}\label{Z}
\end{align}
where $\cU^0_r (\hat{\msg},L^m,D)$ is defined as \eqref{U0r}, $y\equiv y(\hat{\msg},\bar{a}^m,\bar{p})=\bar{p}-\hat{R}\bar{a}^m-\hat{b}$ and
\begin{align} S_\Gamma(\hat{R})&=\mrm{vex}\big(\Gamma\hat{R}\T-\hat{R}\Gamma\T\big)=\mrm{vex}\big(DW(L^m)\T\hat{R}\T-\hat{R}L^mWD\T\big), \label{SLdef} \end{align}
$\Gamma= DW(L^m)\T$ and $\mrm{vex}(\cdot): \so\to\bR^3$ is the inverse of the $(\cdot)^\times$ map.
\end{theorem}
{\em Proof}: A Rayleigh dissipation term linear in the velocities of the form $\bD\varphi$ 
where $\bD\in\bR^{6\times 6} >0$ is used in addition to the Lagrangian \eqref{NFContLag}, and 
the Lagrange-d'Alembert principle from variational mechanics is applied to obtain the estimator on 
$\Ta\SE$.  {\em Reduced variations} with respect to $\msh$ and $\varphi$~\cite{blbook,marat} are 
applied, given by
\begin{align}
\delta\msh&= \msh\eta^\vee,\; \delta\varphi= \dot\eta+ \ad{\varphi}\eta,\label{variations}\\
\mbox{where }&\eta^\vee=\bbm \Sigma^\times & \rho\\ 0\;\; & 0\ebm\mbox{ and }\ad{\mpz{\zeta}}=\bbm \mpz w^\times~~ & 0\\ \mpz v^\times\;\; & \mpz{w}^\times\ebm,\label{ad_def}
\end{align}
for $\eta=\bbm \Sigma\\ \rho\ebm\in\bR^6$ and
$\zeta=\bbm \mpz w\\ \mpz v\ebm\in\bR^6$, with $\eta(t_0)=\eta(T)=0$. This leads to the expression:
\begin{align}
\delta_{\msh,\varphi} \cS \big(\cL (\msh,D,\bar{p},\varphi)\big)=\int_{t_0}^T \eta\T
\bD\varphi \di t.\label{LagdAlem}
\end{align}
Note that the variations of the attitude and position estimation errors are of the form
\begin{align}
\delta Q=Q\Sigma^\times, \;\delta x=Q\rho,
\label{deltaQ}
\end{align}
respectively. Applying reduced variations to the rotational potential energy term 
\eqref{NFU0r}, one obtains
\begin{align}
\delta_{Q}\cU^0_r(\msh,D)&=\lan -Q\Sigma^\times, K\ran=\frac12\lan \Sigma^\times,KQ-Q\T K\ran
=S_K\T(Q)\Sigma,
\label{deltaU0r}
\end{align}
where
\be S_K(Q)=\mrm{vex}\big(K Q-Q\T K\big). \label{SBdef} \ee
Taking first variation of the translational potential energy term \eqref{NFU0t} with respect to $Q$ 
and $x$ yields:
\begin{align}
\delta_{\msh}\cU_t(\msh,\bar{p})&=\kappa(\delta x+\delta Q\bar{p})\T\big\{x+(Q-I)\bar{p}\big\}=\kappa\big(\rho\T y+\Sigma\T\bar{p}^\times y\big).
\label{deltaU0t}
\end{align}

Therefore, the first variation of the total potential energy \eqref{NFcostU} with respect to estimation errors is
\begin{align}
\delta_{\msh}\cU(\msh,D,\bar{p})=Z\T(\msh,D,\bar{p})\eta,
\end{align}
where $Z(\msh,D,\bar{p})$ is defined by
\begin{align}
Z(\msh,D,\bar{p})=\bbm \Phi'\Big(\lan I-Q, K\ran\Big)S_K(Q)+\kappa\bar{p}^\times\big\{Q\T x+(I-Q\T)\bar{p}\big\} \\
\kappa \{Q\T x+(I-Q\T)\bar{p}\}\ebm.\label{NFZ}
\end{align}
Taking the first variation of the kinetic energy term \eqref{NFcostT} with respect to $\varphi$ 
results in:
\begin{align}
\delta_{\varphi}\cT(\varphi)=\varphi\T\bJ\delta\varphi=\varphi\T\bJ(\dot{\eta}+\ad{\varphi}\eta),
\end{align}
applying the reduced variation for $\delta\varphi$ as given in \eqref{variations}. Therefore, the first 
variation of the action functional \eqref{action} is obtained as
\begin{align}
\delta_{\msh,\varphi}\cS\big(\cL(\msh,D,\bar{p},\varphi)\big)&=\int_{t_0}^T\big\{\varphi\T\bJ(\dot{\eta}+\ad{\varphi}\eta)-\eta\T Z(\msh,D,\bar{p})\big\}\di t\nn\\
&=\int_{t_0}^T\eta\T\Big(\adast{\varphi}\bJ\varphi-Z(\msh,D,\bar{p})-\bJ\dot{\varphi}\Big)\di t+\varphi\T\bJ\eta|_{t_0}^T  \nn\\
&=\int_{t_0}^T\eta\T\Big(\adast{\varphi}\bJ\varphi-Z(\msh,D,\bar{p})-\bJ\dot{\varphi}\Big)\di t,
\label{actionVariation}
\end{align}
applying fixed endpoint variations with $\eta(t_0)=\eta(T)=0$. Substituting \eqref{actionVariation} in expression \eqref{LagdAlem} one obtains
\be \bJ\dot{\varphi}=\adast{\varphi}\bJ\varphi-Z(\msh,D,\bar{p})-\bD\varphi,\label{NFdotvarphi}\ee
where $Z(\msh,D,\bar{p})$ is defined by \eqref{NFZ}. In order to implement this estimator using the 
aforementioned measurements, substitute $Q\T D=\hat{R}L^m$. This changes the rotational 
potential energy formed by the estimation errors in attitude \eqref{NFU0r} to \eqref{U0r}.
Equation \eqref{SBdef} is also reformulated as
\begin{align}
 S_K(Q)&=\mrm{vex}(DWD\T Q-Q\T DWD\T)\label{SBSL}\\
&=\mrm{vex}(DW(L^m)\T\hat{R}\T-\hat{R}(L^m)WD\T)=S_\Gamma(\hat{R}).\nn
\end{align}
Finally, the second row in the matrix $Z(\msh,D,\bar{p})$ is replaced by
\begin{align}
\kappa \{Q\T x+(I-Q\T)\bar{p}\}&=\kappa \{Q\T b-\hat{b}+\bar{p}-Q\T\bar{p}\}\nn\\
&=\kappa \{\hat{R}R\T(b-\bar{p})-\hat{b}+\bar{p}\}\nn\\
&=\kappa \{-\hat{R}\bar{a}^m-\hat{b}+\bar{p}\}.
\end{align}
Taking these changes into account, one could obtain the first of equations \eqref{ContFil} with 
$Z(\hat{\msg},L^m,D,\bar{a}^m,\bar{p})$ and $S_\Gamma(\hat{R})$ defined by \eqref{Z} and \eqref{SLdef}, respectively. 
Thus, the complete nonlinear estimator equations are given by \eqref{ContFil}.
\hfill\ensuremath{\square}

This is a fundamentally new idea of applying a principle from variational mechanics 
to obtain a state estimator, recently applied to rigid body attitude estimation 
in~\cite{Automatica}. This approach differs from the ``minimum-energy" approach to 
nonlinear estimation due to Mortensen~\cite{Mortensen} in some important ways. The 
minimum-energy approach applies Hamilton-Jacobi-Bellman (HJB) theory~\cite{kirk}, 
which can only be ``approximately solved." This approach was recently applied to state 
estimation of rigid body attitude motion in \cite{ZamPhD}. This HJB formulation can only be
approximately solved in practice, using a Riccati-like equation, to obtain a near-optimal filter 
that has no guarantees on stability. In the proposed approach, the time 
evolution of $(\hat \msg,\hat\xi)$ has the form of the dynamics of a rigid body with Rayleigh 
dissipation. This results in an estimator for the motion states $(\msg,\xi)$ that 
dissipates the ``energy" content in the estimation errors $(\mathsf{h},\varphi)= (\msg \hat\msg^{-1}, 
\Ad{\hat\msg}(\xi- \hat\xi))$ to provide guaranteed asymptotic stability in the case of perfect 
measurements~\cite{Automatica}. The differences between these two approaches were 
detailed in~\cite{ICRA2015}, for rigid body attitude estimation.

The proposed estimator combines certain desirable features of stochastic estimation and 
observer design approaches to state estimation for unmanned vehicles, when simultaneous 
inertial vector measurements and optical measurements of fixed beacons or landmarks are available. 
This nonlinear estimator is robust to measurement noise and does not require a dynamics model for 
the vehicle; instead, it estimates the dynamics of the vehicle given the measurement model 
in Section \ref{Sec2}. The variational pose estimator can also be interpreted as a low-pass stable filter (cf. \cite{Tayebi2011}). 
Indeed, one can connect the low-pass filter interpretation to the simple example of the natural dynamics 
of a mass-spring-damper system. This is a consequence of the fact that the mass-spring-damper system 
is a mechanical system with passive dissipation, evolving on a configuration space that is the vector 
space of real numbers, $\bR$. In fact, the equation of motion of this system can be obtained by 
application of the Lagrange-d'Alembert principle on the configuration space $\bR$. 
If this analogy or interpretation is extended to a system evolving on a Lie group as a 
configuration space, then the generalization of the mass-spring-damper system is a ``forced 
Euler-Poincar\'{e} system'' \cite{blbook,marat} with passive dissipation, as is obtained here. Explicit expressions for the vector of velocities $\xi^m$ can be obtained for two common cases 
when these velocities are not directly measured. These two cases are dealt with in the next subsection.

\subsubsection{Variational Estimator Implemented without Direct Velocity Measurements}\label{ContButterworth}

The velocity measurements in \eqref{ContFil} can be replaced by filtered velocity estimates 
obtained by linear filtering of optical and inertial measurements using, e.g., a second-order 
Butterworth filter. This is both useful and necessary when velocities are not directly measured. The 
filtered values $\xi^f$ are then used in place of $\xi^m$ to enhance the nonlinear estimator given by 
Theorem \ref{filterTHM}. Denote the measured vector quantity at time 
$t$ by $\mpz{z}^m$. A linear second-order filter of the form:
\begin{align}
\ddot{\mpz{z}}^f+ 2\mu\omega_n\dot{\mpz{z}}^f=\omega_n^2\big(\mpz{z}^m-\mpz{z}^f\big),
\label{ContBW}
\end{align}
is used, where $\omega_n$ is the natural (cutoff) frequency, $\mu$ is the damping ratio, and 
$\mpz{z}^f$ is the filtered value of $\mpz{z}^m$. Thereafter, 
$\mpz{z}^f$ is used in place of $\mpz{z}^m$ in equations \eqref{ContFil}. 

\paragraph{Angular velocity is measured using rate gyros}
For the case that rate gyro measurements of angular velocities are available besides the 
$\mpz{j}$ feature point (or beacon) position measurements, the linear velocities of the rigid body can 
be calculated using each single position measurement by rewriting \eqref{pjdot} as
\begin{align}
\nu^f=(a_j^f)^\times\Omega^f-v_j^f.
\end{align}
for the $j^{th}$ point. Averaging the values of $\nu$ derived from all feature points gives a more reliable result. Therefore, the rigid body's filtered velocities are expressed in this case as
\begin{align}
\xi^f=\bbm \Omega^f\\ \frac{1}{\mpz{j}}\sum\limits_{j=1}^{\mpz{j}}(a_j^f)^\times\Omega^f-v_j^f \ebm.
\end{align}

\paragraph{Translational and angular velocity measurements are not available}
In the case that both angular and translational velocity measurements are not available or accurate, 
rigid body velocities can be calculated in terms of the inertial and optical measurements. In order to do 
so, one can differentiate \eqref{FrameTrans} as follows
\begin{align}
&\dot{p}_j=R\Omega^\times a_j+R\dot{a}_j+\dot{b}=R\big(\Omega^\times a_j+\dot{a}_j+\nu\big)=0\nn\\
\Rightarrow&\dot{a}_j-a_j^\times\Omega+\nu=0\nn\\
\Rightarrow&v_j=\dot{a}_j=[a_j^\times\; -I]\xi=G(a_j)\xi,
\label{pjdot}
\end{align}
where $G(a_j)=[a_j^\times\; -I]$ has full row rank. From vision-based or Doppler lidar sensors, one can also 
measure the velocities of the observed points in frame $\sS$, denoted $v_i^m$. Here, velocity 
measurements as would be obtained from vision-based sensors is considered. 
The measurement model for the velocity is of the form
\be v_j^m=G(a_j)\xi+\vartheta_j,\ee
where $\vartheta_j\in\bR^3$ is the additive error in velocity measurement $v_j^m$. Instantaneous 
angular and translational velocity determination from such measurements is treated 
in~\cite{ast_acc14}. 
Note that $v_j=\dot{a}_j$, for $j\in\mathcal I(t)$. As this kinematics indicates, the relative 
velocities of at least three beacons are needed to determine the vehicle's translational and 
angular velocities uniquely at each instant. However, when only one or two landmarks/beacons are 
measured, the estimator can propagate velocity estimates based on a least squares velocity 
determined from the available measurements. The rigid body velocities in both cases are obtained 
using the pseudo-inverse of $\mathds{G}(A^f)$:
\begin{align}
\mathds{G}(A^f)\xi^f&=\mathds{V}(V^f)\Rightarrow\xi^f=\mathds{G}^\ddag(A^f)\mathds{V}(V^f),\label{ximMore2}\\
\mbox{where }\;\mathds{G}(A^f)&=\bbm G(a^f_1)\\\vdots\\G(a^f_\mpz j)\ebm\;\mbox{and }\;\mathds{V}(V^f)=\bbm v^f_1\\\vdots\\v^f_\mpz j \ebm,\label{GVDef}
\end{align}
for $1,...,\mpz j\in\mathcal I(t)$. 
When at least three beacons are measured, $\mathds{G}(A^f)$ is a full column rank matrix, 
and $\mathds{G}^\ddag(A^f)= \Big( \mathds{G}\T(A^f)\mathds{G} (A^f)\Big)^{-1} 
\mathds{G}\T (A^f)$ gives its pseudo-inverse. For the case that only one or two beacons are 
observed, $\mathds{G}(A^f)$ is a full row rank matrix, whose pseudo-inverse is given by 
$\mathds{G}^\ddag (A^f)= \mathds{G}\T (A^f)\Big( \mathds{G}(A^f)\mathds{G}\T 
(A^f)\Big)^{-1}$.

\subsection{Stability and Robustness of Estimator}\label{Sec4}
The stability of the estimator (filter) given by Theorem \ref{filterTHM} is analyzed here.
The following result shows that this scheme is stable, with almost global convergence of 
the estimated states to the real states in the absence of measurement noise.
\begin{theorem} \label{thmFilt0}
Let the observed position vectors from optical measurements be bounded. Then, 
the estimator presented in Theorem \ref{filterTHM} is asymptotically stable at the estimation error 
state $(\msh,\varphi)=(I,0)$ in the absence of measurement noise. Further, the domain of 
attraction of $(\msh,\varphi)=(I,0)$ is a dense open subset of $\SE\times\bR^6$. 
\end{theorem}
{\em Proof}: In the absence of measurement noise, $\hat R E=Q\T D$. Therefore, the function $\Phi\big(\cU^0_r
(\hat\msg,L^m,D)\big)=\Phi\big(\cU^0_r(\msh,D)\big)$ is a Morse function on $\SO$. The stability of this 
estimator can be shown using the following candidate Morse-Lyapunov function, which can be 
interpreted as the total energy function (equal in value to the Hamiltonian) corresponding to the 
Lagrangian \eqref{contLag}:
\begin{align}
V(\msh,&D,\bar{p},\varphi)=\cT (\varphi)+\cU(\msh,D,\bar{p})=\frac{1}{2}\varphi\T\bJ\varphi+\Phi \big(\lan I-Q, K\ran\big)+\frac12\kappa y\T y.\label{lyapf}
\end{align}
Note that $V(\msh,D,\bar{p},\varphi)\geq 0$ and $V(\msh,D,\bar{p},\varphi)=0$ if and only if $(\msh,\varphi)=(I,0)$. Therefore, $V(\msh,D,\bar{p},\varphi)$ is positive definite on 
$\SE\times\bR^6$. Using \eqref{Qdot}, one can derive the time derivative of \eqref{NFcostU} as
\begin{align}
\frac{\di}{\di t}\cU(\msh,D,&\bar{p})=\Phi'(\cU^0_r(\msh,D)\big)\lan-Q\omega^\times, K\ran+\kappa (\dot x+\dot Q\bar{p})\T (Qy)\nn\\
&=\Phi'(\cU^0_r(\msh,D)\big)\lan\omega^\times, -Q\T K\ran+\kappa (Q\upsilon+Q\omega^\times\bar{p})\T(Qy)\nn\\
&=\frac12\Phi'(\cU^0_r(\msh,D)\big)\lan\omega^\times, KQ-Q\T K\ran+\kappa (\upsilon+\omega^\times\bar{p})\T y\nn\\
&=\Phi'(\cU^0_r(\msh,D)\big)S\T_K(Q)\omega+\kappa y\T \upsilon+\kappa(\bar{p}^\times y)\T \omega\nn\\
&=Z\T(\msh,D,\bar{p})\varphi,
\label{cU0dot} 
\end{align}
where $S_K(Q)$ is defined as \eqref{SBdef} and $Z(\msh,D,\bar{p})$ as \eqref{NFZ}. Therefore, the time derivative of the candidate Morse-Lyapunov function is
\begin{align}
\dot{V}(\msh,D,\bar{p}&,\varphi)=\varphi\T\bJ\dot\varphi+\varphi\T Z(\msh,D,\bar{p})\nn\\
&=\varphi\T\Big(\adast{\varphi}\bJ\varphi-Z(\msh,D,\bar{p})-\bD\varphi+Z(\msh,D,\bar{p})\Big)\nn\\
&=-\varphi\T\bD\varphi. \label{dlyapf}
\end{align}
noting that $\varphi\T\adast{\varphi}\bJ\varphi=0$. Hence, the derivative of the Morse-Lyapunov function is negative semi-definite. Note that the error dynamics for the pose estimate error $\msh$ is given by \eqref{hdot}, while the 
error dynamics for the velocities estimate error $\varphi$ is given by \eqref{NFdotvarphi}. Note that $D(t)$, as a function of time, is piecewise continuous and uniformly bounded. The first property 
(piecewise continuity) is naturally satisfied by $D(t)$, which is piecewise constant as the number and 
inertial positions of beacons (or feature points) observed by body-fixed optical sensors is piecewise 
continuous in time. The second property (uniform boundedness) is satisfied by $D(t)$ if the position 
vectors observed are bounded in $\bR^3$, as assumed in the statement. Therefore, the error dynamics 
for $(\msh,\varphi)$ is non-autonomous. Considering \eqref{lyapf} and \eqref{dlyapf}, and 
applying Theorem 8.4 in \cite{khal}, one can conclude that $\varphi\T \bD\varphi
\rightarrow 0$ as $t\rightarrow \infty$, which consequently implies $\varphi\rightarrow0$. 
Thus, the positive limit set for this system is contained in
\begin{align}
\cE = \dot{V}^{-1}(0)=\big\{(\msh,\varphi)\in\SE\times\se:\varphi\equiv0\big\}. 
\label{dotV0}
\end{align}
Substituting $\varphi\equiv 0$ in the first equation of the estimator \eqref{ContFil}, we obtain the 
positive limit set where $\dot V\equiv 0$ (or $\varphi\equiv 0$) as the set
\begin{align}
\mathscr{I} &= \big\{(\msh,\varphi)\in\SE\times\bR^6: Z(\msh,D,\bar{p})\equiv 0, 
\varphi\equiv0\big\} \label{invset}\\
&= \big\{(\msh,\varphi)\in\SE\times\bR^6: Q \in C_Q,\  Q\T x=0,\ \varphi\equiv0\big\},\nn
\end{align}
where $C_Q$ is defined by \eqref{C_Qdef}. Therefore, in the absence of measurement errors, all the 
solutions of this estimator converge 
asymptotically to the set $\mathscr{I}$. Define $\cU_r (Q):= \Phi \big(\lan I-Q, K\ran\big)$, 
which is the attitude measurement residual in the case of perfect measurements. 
Thus, the attitude estimate error converges to the set of critical points of $\cU_r (Q)$ in this 
intersection, and the position estimate error $x$ converges to zero. 
%
The unique global minimum of $\cU_r (Q)$ is at $Q=I$ (Lemma 2.1 in \cite{Automatica}), so 
this estimation error is asymptotically stable.\par
Now consider the set
\be \mathscr{C}= \mathscr{I}\setminus (I,0),  \label{othereqb} \ee
which consists of all stationary states that the estimation errors may converge to, besides 
the desired estimation error state $(I,0)$. Note that all states in the stable manifold of a 
stationary state in $\mathscr{C}$ converge to this stationary state. 
From the properties of the critical points $Q_\iota\in C_Q\setminus (I)$ of $\cU^0_r (Q)$, 
$(\iota=1,2,3)$ given in Lemma 2.1 of \cite{Automatica}, we see that the stationary points in 
$\mathscr{I}\setminus (I,0)=
\big\{ (\bbm Q_\iota\;\;&0\\0\;\;\;&1\ebm ,0) : Q_\iota\in C_Q\setminus (I)\big\}$ have stable manifolds 
whose dimensions depend on the index of $Q_\iota$. Since the velocities estimate error $\varphi$ 
converges globally to the zero vector, the dimension of the stable manifold $\cM^S_\iota$
of the critical points, i.e. $(\bbm Q_\iota\;\;&0\\0\;\;\;&1\ebm,0)\in\SE\times\bR^6$ is 
\be \dim (\cM^S_\iota) = 9+(3-\,\mbox{index of } Q_\iota)= 12- \,\mbox{index of } Q_\iota. \label{dimStabM} \ee
Therefore, the stable manifolds of $(\msh,\varphi)=(\bbm Q_\iota\;\;&0\\0\;\;\;&1\ebm,0)$ are nine-dimensional, 
ten-dimensional, or eleven-dimensional, depending on the index of $Q_\iota\in C_Q\setminus (I)$ 
according to \eqref{dimStabM}. Moreover, the value of the Lyapunov function $V(\msh,D,
\varphi)$ is non-decreasing (increasing when $(\msh,\varphi)\notin\mathscr{I}$) 
for trajectories on these manifolds when going backwards in time. This 
implies that the metric distance between error states $(\msh,\varphi)$ along these 
trajectories on the stable manifolds $\cM^S_\iota$ grows with the time separation between these 
states, and this property does not depend on the choice of the metric on $\SE\times\bR^6$. 
Therefore, these stable manifolds are embedded (closed) submanifolds of $\SE\times\bR^6$ 
and so is their union. Clearly, all states starting in the complement of this union, converge 
to the stable equilibrium $(\bbm Q_\iota\;\;&0\\0\;\;\;&1\ebm,0)=(I,0)$; therefore the domain of 
attraction of this equilibrium is
\[ \mbox{DOA}\{(I,0)\} = \SE\times\bR^6\setminus\big\{\cup_{\iota=1}^3 \cM^S_\iota\big\}, \]
which is a dense open subset of $\SE\times\bR^6$. 
\hfill\ensuremath{\square}

Therefore, the domain of attraction for the variational estimation scheme at $(\msh,\varphi)=(I,0)$ is 
almost global over the state space $\Ta\SE\simeq\SE\times\bR^6$, which is the best possible with 
continuous control and navigation schemes for systems evolving on a non-contractible state 
space~\cite{chaturvedi2011rigid,bo:miln}. In the presence of measurement noise with 
bounded frequencies and amplitudes, one can show that the expected values of the 
state estimates converge to a bounded neighborhood of the true states. The size of this 
neighborhood, which can be considered as a measure of the robustness of this estimation 
scheme, depends on the values of the estimator gains $\bJ$, $W$ and $\bD$. These estimator 
gains can be selected based on balancing the transient and steady-state behavior
of the estimator. 

\begin{remark}
In the special case that the weight matrix $W$ in Wahba's function is chosen as a piecewise time 
constant matrix according to Lemma \ref{LemmaW}, $K=DWD\T$ is a constant matrix for all time. 
Therefore, the RHS of \eqref{NFdotvarphi} is not explicitly dependent on time. This makes 
$(\msh,\varphi)$ an autonomous system and therefore the use of Theorem 8.4 of \cite{khal} is not 
required to prove asymptotic stability. One can apply LaSalle's invariance principle (Theorem 4.4 
in \cite{khal}) to prove the convergence of state estimates to the equilibrium $(I,0)$ in this case.
\end{remark}

\subsection{Discretization for Computer Implementation}\label{Sec5}
For onboard computer implementation, the variational estimation scheme outlined above has to 
be discretized. This discretization is carried out in the framework of discrete geometric 
mechanics, and the resulting discrete-time estimator is in the form of a Lie group variational 
integrator (LGVI), as in~\cite{SCLpaper}. Since the estimation scheme proposed here is obtained from 
a variational principle of mechanics, it can be discretized by applying the discrete 
Lagrange-d'Alembert principle~\cite{marswest}. Consider an interval of time $[t_0, T]\in\bR^+$ 
separated into $N$ equal-length subintervals $[t_i,t_{i+1}]$ for $i=0,1,\ldots,N$, with $t_N=T$ 
and $t_{i+1}-t_i=\Delta t$ is the time step size. Let $(\hat \msg_i,\hat\xi_i)\in\SE\times\bR^6$ 
denote the discrete state estimate at time $t_i$, such that $(\hat \msg_i,\hat\xi_i)\approx 
(\hat \msg(t_i),\hat\xi(t_i))$ where $(\hat \msg(t),\hat\xi(t))$ is the exact solution of the 
continuous-time estimator at time $t\in [t_0, T]$. Let the values of the discrete-time measurements  
$\xi^m$, $\bar a^m$ and $L^m$ at time $t_i$ be denoted as $\xi^m_i$, $\bar a^m_i$ and $L^m_i$, 
respectively. Further, denote the corresponding values for the latter two quantities in inertial frame 
at time $t_i$ by $\bar p_i$ and $D_i$, respectively.
The term representing the energy content of the pose estimation error, given by 
\eqref{costU}, is discretized as
\begin{align}
\cU(\hat{\msg}_i,L^m_i,D_i,&\bar{a}^m_i,\bar{p}_i)=\cU_r (\hat{\msg}_i,L_i^m,D_i)+\cU_t(\hat{\msg}_i,
\bar{a}^m_i,\bar{p}_i)\nn\\
&= \Phi \big( \cU^0_r (\hat{\msg}_i,L_i^m,D_i) \big)+\cU_t(\hat{\msg}_i,\bar{a}^m_i,\bar{p}_i)\label{DiscCostU}\\
&=\Phi \big(\frac12\lan D_i -\hat R_i L_i^m , (D_i -\hat R_i L_i^m)W_i\ran\big)+\frac12\kappa\|\bar{p}_i-\hat{R}_i\bar{a}_i^m-\hat{b}_i\|^2,\nn
\end{align}
where $W_i$ is the matrix of weight factors corresponding to $D_i$ at time $t_i$. The term 
encapsulating the energy in the velocities estimate error \eqref{costT}, is discretized as
\begin{align}
\cT \Big(\varphi(\hat\msg_i,\xi_i^m,\hat\xi_i)\Big)= \frac12 \varphi(\hat\msg_i,\xi_i^m,\hat\xi_i)\T \bJ\varphi(\hat\msg_i,\xi_i^m,\hat\xi_i), 
\end{align}
where $\bJ=\diag(J,M)$ and $M,J$ are positive definite matrices.

\begin{lemma}\label{NoiseFree}
In the absence of measurement noise, the discrete-time Lagrangian is of the form
\begin{align}
\cL (\msh_i,D_i,\bar{p}_i&,\varphi_i)=\frac12\lan\mathcal{J}\omega_i^\times,\omega_i^\times\ran+\frac12\lan M\upsilon_i,\upsilon_i\ran-\Phi \big(\lan I-Q_i,K_i\ran\big)-\frac12 \kappa y_i\T y_i,\label{DisLag}
\end{align}
where $y_i\equiv y(\msh_i,\bar p_i)=Q_i\T x_i+(I-Q_i\T)\bar{p}_i$ and $\mathcal{J}$ is defined in terms of the 
matrix $J$ by $\mathcal{J}=\frac12\tr[J]I-J$.
\end{lemma}
A Lie group variational integrator (LGVI) introduced in \cite{Sanyal2011almost} is applied to the 
discrete-time Lagrangian \eqref{DisLag} to obtain the discrete-time filter.

\begin{theorem} \label{discfilter}
A first-order discretization of the estimator proposed in Theorem \ref{filterTHM} is given by
\begin{align}
(J\omega_i)^\times&=\frac{1}{\Delta t}(F_i\mathcal{J}-\mathcal{J}F_i\T),\label{LGVI_F}\\
(M+\Delta t\bD_t)\upsilon_{i+1}&=F_i\T M\upsilon_i+\Delta t \kappa (\hat{b}_{i+1}+\hat{R}_{i+1}\bar{a}^m_{i+1}-\bar{p}_{i+1}),\label{LGVI_upsilon}\\
(J+\Delta t\bD_r)\omega_{i+1}&=F_i\T J\omega_i+\Delta t M\upsilon_{i+1}\times\upsilon_{i+1}\nn\\
+\Delta t&\kappa \bar{p}_{i+1}^\times (\hat{b}_{i+1}+\hat{R}_{i+1}\bar{a}^m_{i+1})\label{LGVI_omega}\\
-\Delta t&\Phi' \big( \cU^0_r (\hat{\msg}_{i+1},L_{i+1}^m,D_{i+1}) \big)S_{\Gamma_{i+1}}(\hat{R}_{i+1}),\nn\\
\hat\xi_i&=\xi^m_i-\Ad{\hat\msg_i^{-1}}\varphi_i,\label{LGVI_xihat}\\
\hat\msg_{i+1}&=\hat\msg_i\exp(\Delta t\hat\xi_i^\vee),\label{LGVI_ghat}
\end{align}
where $F_i\in\SO$, $\big(\hat\msg(t_0),\hat\xi(t_0)\big)=(\hat\msg_0,\hat\xi_0)$, 
$\varphi_i=[\omega_i\T\;\upsilon_i\T]\T$, and $S_{\Gamma_i}(\hat R_i)$ is the value of  
$S_\Gamma(\hat R)$ at time $t_i$, with $S_\Gamma(\hat R)$ as defined by \eqref{SLdef}. 
\end{theorem}

{\em Proof}: Consider first variations with fixed endpoints for the pose estimation errors in discrete 
time given by:
\begin{align}
\delta Q_i &=Q_i\Sigma_i^\times,\;\ \Sigma_0=\Sigma_N=0,\label{discQVar}\\
\delta x_i &=Q_i\rho_i,\;\;\;\; \rho_0=\rho_N=0, \label{discxVar}
\end{align}
where $\Sigma_i,\rho_i\in\bR^3$ are ``discrete variation vectors". It can be shown that for any 
$\omega\in\bR^3$ we have
\begin{align}
(J\omega)^\times=\omega^\times\mathcal{J}+\mathcal{J}\omega^\times.
\label{JomegaCross}
\end{align}
Discretizing \eqref{Qdot} assuming that the angular velocity estimation error is constant in the time interval $[t_i,t_{i+1}]$ with a constant time step size $\Delta t$, one gets
\begin{align}
Q_{i+1}=Q_iF_i, \;\ i\in\{0,1,2,\ldots,N-1\},
\label{AltDiscQ}
\end{align}
where $F_i\in\SO$ is given by
\begin{align}
F_i=\exp(\Delta t\omega_i^\times)\approx I+\Delta t\omega_i^\times.
\label{AltF}
\end{align}
The variation of $F_i$ can be derived from \eqref{AltDiscQ} and $\delta Q_i=Q_i\Sigma_i^\times$. Thus
\begin{align}
\delta F_i=-\Sigma_i^\times F_i+F_i\Sigma_{i+1}^\times.
\label{deltaF}
\end{align}
Using \eqref{JomegaCross} and \eqref{AltF}, one can enforce the skew-symmetry of 
$(J\omega_i)^\times$ by
\begin{align}
(J\omega_i)^\times&=\omega_i^\times\mathcal{J}+\mathcal{J}\omega_i^\times\approx\frac{1}{\Delta t}\Big((F_i-I)\mathcal{J}-\mathcal{J}(F_i\T-I)\Big)=\frac{1}{\Delta t}(F_i\mathcal{J}-\mathcal{J}F_i\T).
\label{DiscJomegaCross}
\end{align}

From \eqref{hdot}, the continuous rate of change of the attitude estimation error is $\dot{x}=
Q\upsilon$, which can be approximated to first order in discrete-time as
\begin{align}
\frac{x_{i+1}-x_i}{\Delta t}\approx Q_i\upsilon_i\Rightarrow x_{i+1}=\Delta t Q_i\upsilon_i+x_i.
\label{Discx}
\end{align}
The first variation in $\upsilon_i$ is then calculated using \eqref{Discx} as
\begin{align}
\delta\upsilon_i&=\delta\Big(\frac{1}{\Delta t}Q_i\T(x_{i+1}-x_i)\Big)\nn\\
&=-\Sigma_i^\times\upsilon_i+\frac{1}{\Delta t}Q_i\T(\delta x_{i+1}-\delta x_i)\nn\\
&=-\Sigma_i^\times\upsilon_i+\frac{1}{\Delta t}F_i\rho_{i+1}-\frac{1}{\Delta t}\rho_i.
\label{deltanu}
\end{align}
The discrete Lagrangian \eqref{DisLag} can be rewritten as
\begin{align}
\cL(\msh_i,&D_i,\bar{p}_i,F_i,\upsilon_i)=\label{AltLag}\\
\frac{1}{2\Delta t}\lan &\mathcal{J}(F_i-I),(F_i-I)\ran+\frac{\Delta t}{2}\lan M\upsilon_i,\upsilon_i\ran-\Delta t\Phi \big( \cU^0_r (\msh_i,D_i) \big)
-\frac{\Delta t}{2}\kappa(Q_iy_i)\T (Q_iy_i).\nn
\end{align}
The action functional \eqref{action} is replaced by the action sum 
\be \cS_d \big(\cL (\msh_i,D_i,\bar{p}_i,F_i,\upsilon_i)\big)= \Delta t
\sum_{i=0}^{N-1} \cL (\msh_i,D_i,\bar{p}_i,F_i,\upsilon_i). \label{actsum} \ee
Applying the discrete Lagrange-d'Alembert principle with two Rayleigh dissipation terms for 
angular and translational motions gives 
\begin{align}
\delta\cS_d\big(\cL(\msh_i,D_i,\bar{p}_i,F_i,\upsilon_i)\big)&+\Delta t\sum_{i=0}^{N-1} \Big\{\lan\Sigma_i,\tau_i\ran+\lan 
\rho_i, f_i\ran\Big\}=0\label{dLagdAl}\\
\Rightarrow\sum_{i=0}^{N-1}\Bigg\{\frac{1}{\Delta t}\lan\delta F_i,\mathcal{J}(F_i-I)&\ran+\Delta t\lan \delta\upsilon_i,M\upsilon_i\ran-\frac{\Delta t}{2}\Phi' \big( \cU^0_r (\msh_i,D_i) \big)\big\lan\Sigma_i^\times,S_{K_i}^\times(Q_i)\big\ran  \nn\\
-\Delta t\kappa\lan \rho_i,y_i&\ran-\Delta t\kappa\lan \Sigma_i^\times,y_i\bar{p}_i\T\ran+\frac{\Delta t}{2}\lan\Sigma_i^\times,\tau_i^\times\ran+\Delta t\lan \rho_i,f_i\ran
\Bigg\}=0.\nn
\end{align}
As symmetric matrices are orthogonal to skew-symmetric matrices in the trace inner product, 
using \eqref{AltF} we can rewrite the first term in \eqref{AltLag} as
\begin{align}
\lan\delta F_i,\mathcal{J}(F_i-I)\ran&=\lan\Sigma_i^\times,\mathcal{J} F_i\T\ran-\lan\Sigma_{i+1}^\times,F_i\T\mathcal{J}\ran
\label{deltaFJ}\\
&=\frac12\lan\Sigma_i^\times,\mathcal{J} F_i\T\ran-\frac12\lan\Sigma_i^\times,F_i\mathcal{J}\ran-\frac12\lan\Sigma_{i+1}^\times,F_i\T\mathcal{J}\ran+\frac12\lan\Sigma_{i+1}^\times,\mathcal{J}F_i\ran\nn\\
&=-\frac{\Delta t}{2}\lan\Sigma_i^\times,(J\omega_i)^\times\ran+\frac{\Delta t}{2}\lan\Sigma_{i+1}^\times,F_i\T(J\omega_i)^\times F_i\ran.\nn
\end{align}
Hence equation \eqref{dLagdAl} can be re-expressed as
\begin{align}
\sum_{i=0}^{N-1}&\Bigg\{ -\frac12\lan\Sigma_i^\times,(J\omega_i)^\times\ran+\frac12\lan\Sigma_{i+1}^\times,F_i\T(J\omega_i)^\times F_i\ran-\frac{\Delta t}{2}\lan \Sigma_i^\times,(\upsilon_i\times M\upsilon_i)^\times\ran\nn\\
&+\lan F_i\rho_{i+1},M\upsilon_i\ran-\lan \rho_i,M\upsilon_i\ran-\frac{\Delta t}{2}\Phi' \big( \cU^0_r (\msh_i,D_i) \big)\big\lan\Sigma_i^\times,S_{K_i}^\times(Q_i)\big\ran-\kappa\Delta t\big\lan \rho_i, y_i\big\ran\nn\\
&-\frac{\kappa\Delta t}{2}\big\lan\Sigma^\times_i, (\bar{p}_i^\times y_i)^\times \big\ran+\frac{\Delta t}{2}\lan\Sigma_i^\times,\tau_i^\times\ran+\Delta t\lan \rho_i,f_i\ran
\Bigg\}=0.
\label{AltLdA2}
\end{align}
Separating this equation into two (rotational and translational) parts leads to
\begin{align}
(M+\Delta t\bD_t)\upsilon_{i+1}=&F_i\T M\upsilon_i-\Delta t \kappa y_{i+1},\\
(J+\Delta t\bD_r)\omega_{i+1}=&F_i\T J\omega_i+\Delta t M\upsilon_{i+1}\times\upsilon_{i+1}-\Delta t\kappa \bar{p}_{i+1}^\times y_{i+1}\\
&-\Delta t\Phi' \big( \cU^0_r (\msh_{i+1},D_{i+1}) \big)S_{K_{i+1}}(Q_{i+1}),\nn
\end{align}
using the identity $\mpz F\T\mpz w^\times \mpz F=(\mpz F\T\mpz w)^\times$ and by replacing 
$\tau_i=-\bD_r \omega_i$ and $f_i=-\bD_t \upsilon_i$, where $\bD_r$ and $\bD_t$ are positive 
definite matrices such that 
\[ \bD= \bbm \bD_r & 0\\ 0\; &\; \bD_t\ebm. \] 
In the presence of measurement noise, $Q_i\T D_i$ and $y_i$ 
are replaced by $\hat{R}_iL_i^m$ and $\bar{p}_i-\hat{b}_i-\hat{R}_i\bar{a}_i^m$, respectively. These 
give the discrete-time state estimator in the form of the Lie group variational integrator 
\eqref{LGVI_F}-\eqref{LGVI_ghat}. 
\hfill\ensuremath{\square}

Model-based discrete-time rigid body state estimators 
using LGVI schemes for attitude estimation were reported in \cite{SCLpaper,jgcd12}, but 
dynamics model-free state estimators using LGVIs have appeared only recently  
in~\cite{Automatica,ACC2015}.

\begin{remark}
In the absence of any direct velocity measurements or only angular velocity measurements, the expressions provided in Section \ref{ContButterworth} to calculate rigid body velocities are still valid in discrete-time. 
One can use the discrete-time variables introduced in this section in place of their continuous-time counterparts. The second-order Butterworth filter \eqref{ContBW} is discretized using the 
\textit{Newmark-$\beta$ Method} as follows:
\begin{align}
\begin{cases} 
\mpz{z}^f_{i+1}=\mpz{z}^f_i+\Delta t \dot{\mpz{z}}^f_i+\frac{\Delta t^2}{4}(\ddot{\mpz{z}}^f_i+\ddot{\mpz{z}}^f_{i+1})\\
\dot{\mpz{z}}^f_{i+1}=\dot{\mpz{z}}^f_i+\frac{\Delta t}{2}(\ddot{\mpz{z}}^f_i+\ddot{\mpz{z}}^f_{i+1})
\end{cases}.
\label{Newmark}
\end{align}
Choosing $\omega_n=2$ and $\mu=\frac12$, this method gives the filtered positions and velocities as  follows:
\begin{align}
\begin{Bmatrix}
\mpz{z}^f_{i+1}\\\dot{\mpz{z}}^f_{i+1}
\end{Bmatrix}
&=\frac{1}{4+4\mu\omega_n\Delta t+\omega_n^2\Delta t^2}\\&
\begin{bmatrix}
4+4\mu\omega_n\Delta t-\omega_n^2\Delta t^2 \;&\; 4\Delta t\;&\; \omega_n^2\Delta t^2\\
-4\omega_n^2\Delta t \;&\; 4-4\mu\omega_n\Delta t-\omega_n^2\Delta t^2 \;&\; 2\omega_n^2\Delta t
\end{bmatrix}
\begin{Bmatrix}
\mpz{z}^f_i\\
\dot{\mpz{z}}^f_i\\
\mpz{z}_i^m+\mpz{z}_{i+1}^m
\end{Bmatrix}.\nn
\end{align}
where $\mpz{z}^m_i$ and $\mpz{z}^f_i$ are the corresponding value of quantities $\mpz{z}^m$ and $\mpz{z}^f$ at time instant $t_i$, respectively. As with the continuous time version, $\xi^m_i$ can be replaced 
with $\xi^f_i$ in the estimator equations. 
\end{remark}

\subsection{Numerical Simulations}\label{Sec6}
This section presents numerical simulation results for the discrete-time estimator obtained in 
Section \ref{Sec5}. In order to numerically simulate this estimator, simulated true states of an 
aerial vehicle flying in a room are produced using the kinematics and dynamics equations of a 
rigid body. The vehicle mass and moment of inertia are taken to be $m_v=420$ g and $J_v=
[51.2\;\;60.2\;\;59.6]\T$ g.m$^2$, respectively. The resultant external forces and torques applied 
on the vehicle are $\phi_v(t)=10^{-3}[10\cos(0.1t)\;\;2\sin(0.2t)\;\;-2\sin(0.5t)]\T$ N and 
$\tau_v(t)=10^{-6}\phi_v(t)$ N.m, respectively. The room is assumed to be a cubic space of size 
10m$\times$10m$\times$10m with the inertial frame origin at the center of this cube. The initial 
attitude and position of the vehicle are:
\begin{align}
R_0=\expm_{\SO}\bigg(\Big(\frac{\pi}{4}\times[\frac{3}{7}\ -\frac{6}{7}\ \frac{2}{7}]\T\Big)^\times
\bigg),\mbox{and } b_0=[2.5\ 0.5\ -3]\T\mbox{ m}.
\end{align}
This vehicle's initial angular and translational velocity respectively, are:
\begin{align}
\begin{split}
\Omega_0=[0.2\;\; -0.05\;\; 0.1]\T\mbox{ rad/s},\mbox{ and } \nu_0=[-0.05\;\;0.15\;\;0.03]\T\mbox{ m/s}.
\end{split}
\end{align}
The vehicle dynamics is simulated over a time interval of $T=150 \mbox{ s}$, with a time 
stepsize of $\Delta t=0.02 \mbox{ s}$. The trajectory of the vehicle over this time interval is 
depicted in Fig.~\ref{Fig2}.
\begin{figure}
\begin{center}
\includegraphics[height=3.2in]{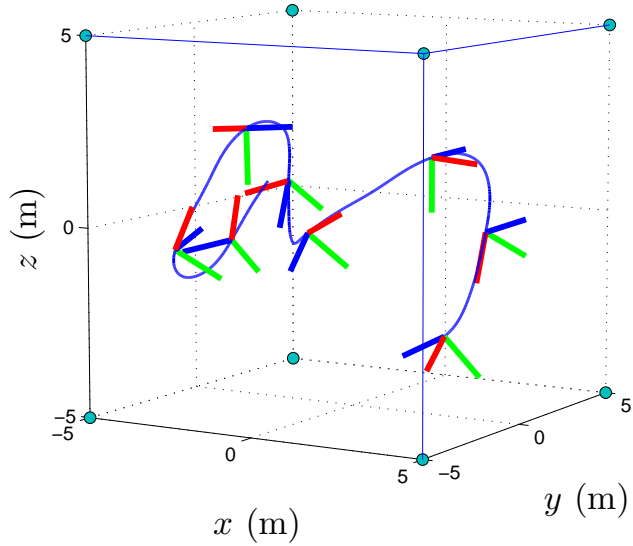}
\caption{Position and attitude trajectory of the simulated vehicle in 3D space.}  
\label{Fig2}                                 
\end{center}                                 
\end{figure}
The following two inertial directions, corresponding to nadir and Earth's magnetic field 
direction, are measured by the inertial sensors on the vehicle:
\begin{align}
d_1=[0\;\;0\;\;-1]\T,\;\;d_2=[0.1\;\;0.975\;\;-0.2]\T.
\end{align}
For optical measurements, eight beacons are located at the eight vertices of the cube, labeled 
1 to 8. The positions of these beacons are known in the inertial frame and their index (label) 
and relative positions are measured by optical sensors onboard the vehicle whenever the 
beacons come into the field of view of the sensors. Three identical cameras (optical sensors) 
and inertial sensors are assumed to be installed on the vehicle. The cameras are fixed to known 
positions on the vehicle, on a hypothetical horizontal plane passing through the vehicle, 120$^\circ$ 
apart from each other, as shown in Fig.~\ref{Frames}. All the camera readings contain random 
zero mean signals whose probability distributions are normalized bump functions with width of 
$0.001$m. The following are selected for the positive definite estimator gain matrices:
\begin{align}
J&=\diag\big([0.9\;\;0.6\;\;0.3]\big),\;\;
M=\diag\big([0.0608\;\;0.0486\;\;0.0365]\big), \\
\bD_r&=\diag\big([2.7\;2.2\;1.5]\big),\;\;\bD_t=\diag\big([0.1\;\;0.12\;\;0.14]\big). \nn
\end{align}
$\Phi(\cdot)$ could be any $C^2$ function with the properties described in Section \ref{Sec3}, but is 
selected to be $\Phi(x)=x$ here. The initial state estimates have the following values:
\begin{align}
\begin{split}
\hat\msg_0=I,\;\;\;\hat\Omega_0=[0.1\;\;0.45\;\;0.05]\T\mbox{ rad/s},
\mbox{ and }\hat\nu_0=[2.05\;\;0.64\;\;1.29]\T\mbox{ m/s}.
\end{split}
\end{align}
The performance of the proposed estimator is presented for two different cases.

\subsubsection{CASE 1: At least three beacons are observed at each time instant}
Having three beacons measured at each time instant guarantees full determination of vehicle's 
translational and angular velocities instantaneously. A conic field of view (FOV) of 
2$\times$40$^\circ$ for cameras can satisfy this condition. The vehicle's velocity is calculated by 
\eqref{ximMore2} in this case. The discrete-time estimator \eqref{LGVI_F}-\eqref{LGVI_ghat} is 
simulated over a time interval of $T=20$ s with sampling interval $\Delta t=0.02$ s. At each time instant, 
\eqref{LGVI_F} is solved using the Newton-Raphson iterative method to find an approximation for 
$F_i$. Following this, the remaining equations (all explicit) are solved to generate the 
estimated states. The principal angle of the attitude estimation error and the position estimation  
error for CASE 1 are plotted in Fig.~\ref{Fig3}. Plots of the angular and translational velocity estimation  
errors are shown in Fig.~\ref{Fig4}.

\begin{figure}
\begin{center}
\includegraphics[height=3.2in]{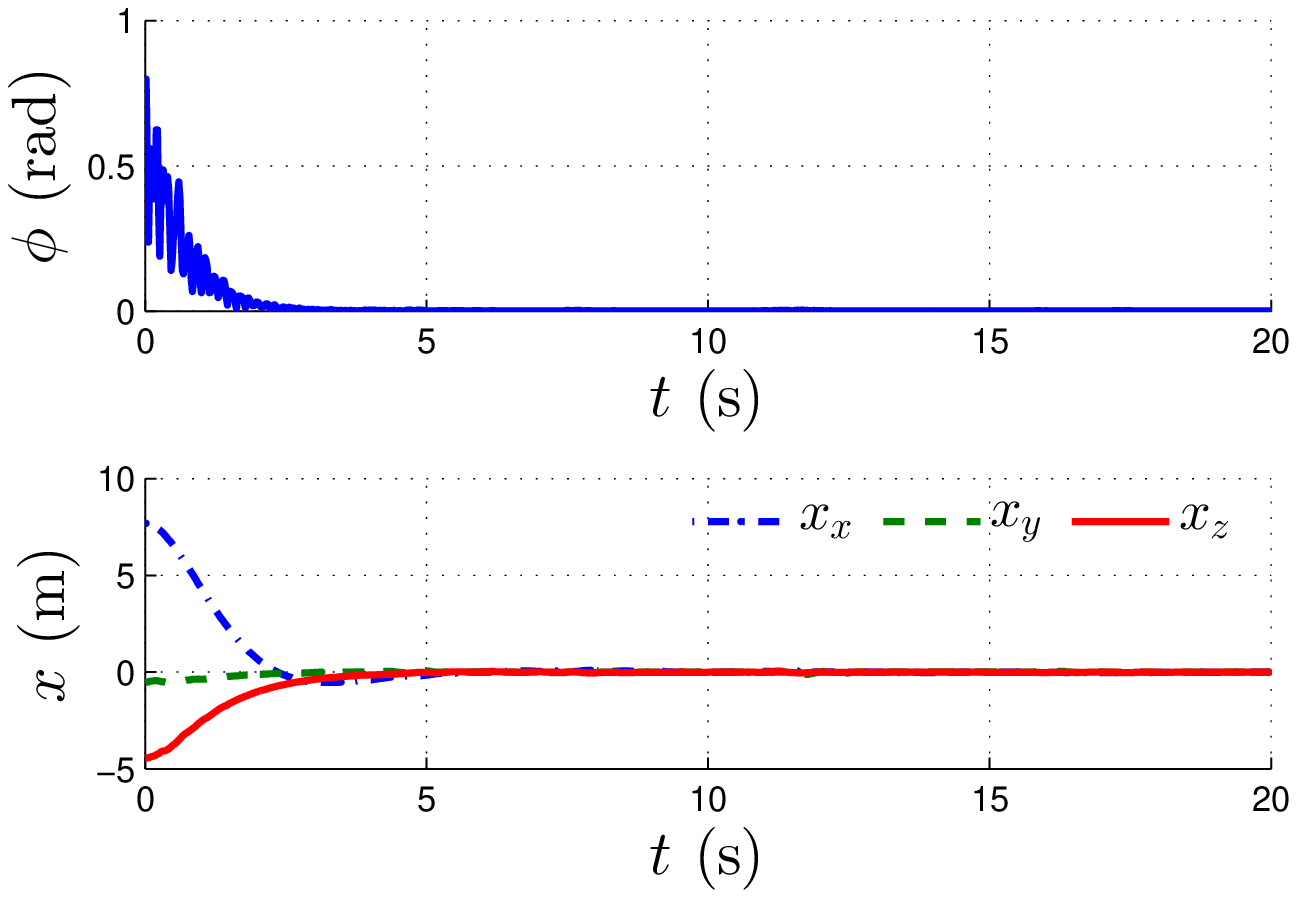}
\caption{Principal angle of the attitude and position estimation error for CASE 1.}  
\label{Fig3}                                 
\end{center}                                 
\end{figure}

\begin{figure}
\begin{center}
\includegraphics[height=3.2in]{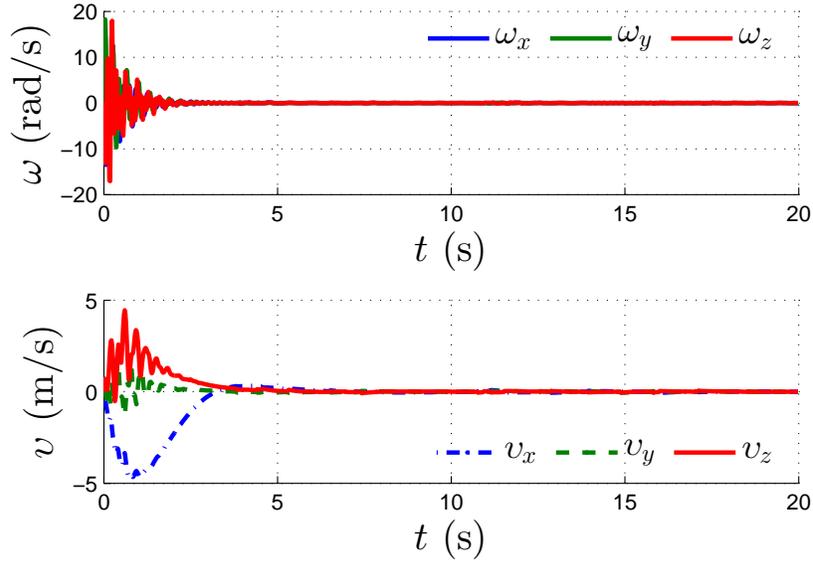}
\caption{Angular and translational velocity estimation error for CASE 1.}  
\label{Fig4}                                 
\end{center}                                 
\end{figure}

\subsubsection{CASE 2: Less than three beacons are measured at some time instants}
To implement the variational estimator for the case that less than three optical 
measurements are available, the field of view of the cameras is decreased to limit the number of 
beacons observed. Assuming the cameras have conical fields of view of 2$\times$25$^\circ$, 
the minimum number of beacons observed instantaneously drops to 1 during 
the simulated time interval. The dynamics model for the aerial vehicle, simulated time duration, 
and sample rate are identical to CASE 1. Fig.~\ref{Fig5} depicts the principal angle of the 
attitude estimation error and the position estimation error for CASE 2, and Fig.~\ref{Fig6} shows 
the angular and translational velocity estimation errors. All estimation errors are shown to 
converge to a neighborhood of $(\msh,\varphi)=(I,0)$ in both cases, where the size of this 
neighborhood depends on the magnitude of measurement noise.

\begin{figure}
\begin{center}
\includegraphics[height=3.2in]{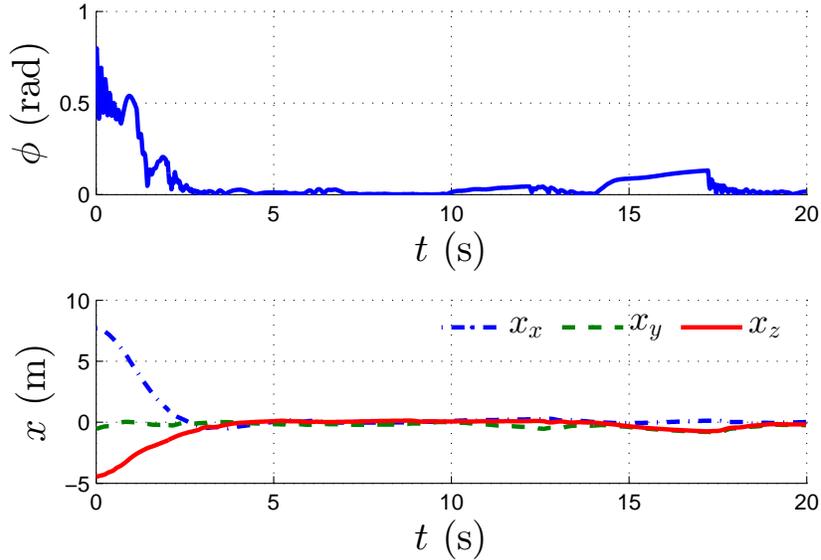}
\caption{Principal angle of the attitude and position estimation error for CASE 2.}  
\label{Fig5}                                 
\end{center}                                 
\end{figure}

\begin{figure}
\begin{center}
\includegraphics[height=3.2in]{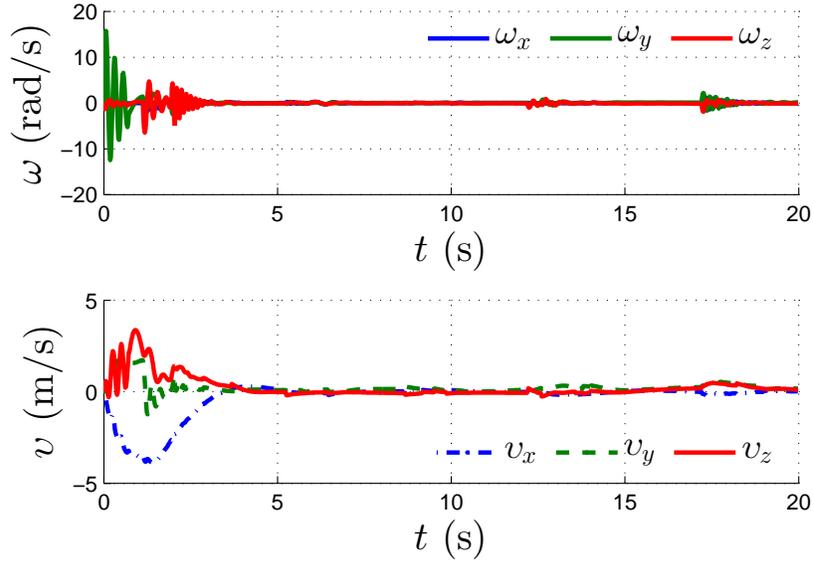}
\caption{Angular and translational velocity estimation error for CASE 2.}  
\label{Fig6}                                 
\end{center}                                 
\end{figure}

\subsection{Conclusion}\label{Sec7}
This chapter proposes an estimator for rigid body pose and velocities, using optical and inertial 
measurements by sensors onboard the rigid body. The sensors are assumed to provide 
measurements in continuous-time or at a sufficiently high frequency, with bounded 
measurement noise. An artificial kinetic energy quadratic in rigid body velocity estimate errors is 
defined, as well as two fictitious potential energies: (1) a generalized Wahba's cost function for 
attitude estimation error in the form of a Morse function, and (2) a quadratic function of  
the vehicle's position estimate error. Applying the Lagrange-d'Alembert principle on a Lagrangian 
consisting of these energy-like terms and a dissipation term linear in velocities estimation error, an 
estimator is designed on the Lie group of rigid body motions. In the absence of measurement 
noise, this estimator is shown to be almost globally asymptotically stable, with estimates converging 
to actual states in a domain of attraction that is open and dense in the state space. The continuous 
estimator is discretized by applying the discrete Lagrange-d'Alembert principle on the discrete 
Lagrangian and dissipation terms linear in rotational and translational velocity estimation errors. 
In the presence of measurement noise, numerical simulations show that state estimates converge 
to a bounded neighborhood of the true states. Future extensions of this work include higher-order 
discretizations of the continuous-time filter given here and obtaining a stochastic 
interpretation of the variational pose estimator.

\newpage

\section{IDEAS FOR FUTURE WORK} 

\label{CH09_Future} 

\hspace{\parindent}
The following are ideas to extend the research presented in this PhD dissertation:
\begin{itemize}

\item{{\bf Combining model-free and model-based estimation}: Implementations of filtering schemes similar to the variational estimator in the presence of measurements of direction vectors and 
angular velocities at different rates, use of a dynamics model for propagation of state 
estimates when measurements are available at low sampling rates, and design of 
state-varying or time-varying filter gains for faster convergence of state estimates.}

\item{{\bf Showing robustness to bounded measurement noise for model-free estimators}: Proving robustness of the variational estimators presented in Chapters \ref{CH04_VE} and \ref{CH08_Automatica2} in the presence of measurement noise. This proof could be obtained by finding a neighborhood of the origin to which the estimates converge, given known upper bounds on the measurement errors. Lyapunov analysis will be exploited to show that outside the derived neighborhood, all the states estimates tend to converge to a region inside that neighborhood.}

\item{{\bf Computationally efficient discretization of model-free filter in $\SE$}: The Lie group variational integrator presented in Chapter \ref{CH08_Automatica2} is implicit and needs to be solved using iterative Newton-Raphson method. This involves a complicated algebra to calculate the rate of change of the implicit equation, in order to approximate its root. Moreover, due to the iterative structure of the solution, the estimator will be computationally slow. A discretization scheme proposed in \cite{Sanyal2011almost} will be utilized in order to obtain computationally more efficient LGVI. Numerical simulations for this variational estimator in $\SE$ will be carried out in future research.}

\item{{\bf Experimental validation of model-free filter in $\SE$}: The performance of the variational estimator presented in Chapter \ref{CH04_VE} has been verified experimentally in \cite{ICC2015}. The comparison of the variational estimator with the state-of-the-art filters can also be verified using the same setup by implementing all the estimators on the Android device. A similar set of experiments could be conducted to show the performance of the $\SE$ version of the filter using experimental equipments. }

\item{{\bf Experimental validation of model-free filter in $\SO$ in the presence of bias in angular velocities}: Gyroscopes are used in practice to provide angular velocities. The output of such sensors usually contain constant or variable drift which harms the performance of the filter. The estimator presented in Chapter \ref{CH06_Bias} is designed in such a way that it could be robust to bias in the sensor readings. The performance of this version of the Variational Attitude Estimator needs to be verified experimentally.}

\end{itemize}

\newpage
\addcontentsline{toc}{section}{REFERENCES}
\setlength{\baselineskip}{\singlespace}
\bibliographystyle{plain}
\bibliography{main}

\begin{thebibliography}{10}

\bibitem{AbeSch}
Abe,~S., Mukai,~T., Hirata,~N., Barnouin-Jha,~O., Cheng,~A., Demura,~H., Gaskell,~R., Hashimoto,~T., Hiraoka,~K., Honda,~T., Kubota,~T., Matsuoka,~M., Mizuno,~T., Nakamura,~R., Scheeres,~D., \& Yoshikawa,~M. (2006).
\newblock Mass and local topography measurements of Itokawa by Hayabusa.
\newblock {\em Science}, 312(5778), 1344--1347.

\bibitem{aguhes06}
Aguiar,~A., \& Hespanha,~J. (2006).
\newblock Minimum-energy state estimation for systems with perspective outputs.
\newblock {\em IEEE Transactions on Automatic Control}, 51(2), 226--241.

\bibitem{amelin2014algorithm}
Amelin,~K.~S., \& Miller,~A.~B. (2014).
\newblock An algorithm for refinement of the position of a light UAV on the
  basis of Kalman filtering of bearing measurements.
\newblock {\em Journal of Communications Technology and Electronics},
  59(6), 622--631.

\bibitem{Barfoot}
Barfoot,~T., Forbes,~J.~R., \& Furgale,~P.~T. (2011).
\newblock Pose estimation using linearized rotations and quaternion algebra.
\newblock {\em Acta Astronautica}, 68(1), 101--112.

\bibitem{Bonnabel2013}
Barrau,~A., \& Bonnabel,~S. (2013).
\newblock Intrinsic filtering on $\SO$ with discrete-time observations.\newblock In {\em Proceedings of the $52^{nd}$ IEEE Conference on Decision and Control} (pp. 3255--3260). Florence, Italy.

\bibitem{Bonnabel2015}
Barrau,~A., \& Bonnabel,~S. (2015).
\newblock Intrinsic filtering on Lie groups with applications to attitude estimation.
\newblock {\em IEEE Transactions on Automatic Control}, 60(2), 436--449.

\bibitem{GES}
Batista,~P., Silvestre,~C., \& Oliveira,~P. (2012).
\newblock A GES attitude observer with single vector observations.
\newblock {\em Automatica}, 48(2), 388--395.

\bibitem{Bayadi2014almost}
Bayadi,~R., \& Banavar,~R.~N. (2014).
\newblock Almost global attitude stabilization of a rigid body for both internal and external actuation schemes.
\newblock {\em European Journal of Control}, 20(1), 45--54.

\bibitem{Lotfi}
Benziane,~L., Benallegue,~A., \& Tayebi,~A. (2014).
\newblock Attitude stabilization without angular velocity measurements.
\newblock In {\em Proceedings of the IEEE International Conference on Robotics and
  Automation} (pp. 3116--3121). Hong Kong.

\bibitem{bhat}
Bhat,~S.~P., \& Bernstein,~D.~S. (2000).
\newblock A topological obstruction to continuous global stabilization of rotational motion and the unwinding phenomenon.
\newblock {\em Systems \& Control Letters}, 39(1), 63--70.

\bibitem{bhatFT}
Bhat,~S.~P., \& Bernstein,~D.~S. (2000).
\newblock Finite-time stability of continuous autonomous systems.
\newblock {\em SIAM Journal on Control and Optimization}, 38(3), 751--766.

\bibitem{TRIAD}
Black,~H. (1964).
\newblock A passive system for determining the attitude of a satellite.
\newblock {\em AIAA Journal}, 2(7), 1350--1351.

\bibitem{blbook}
Bloch,~A.~M. (2003).
\newblock {\em Nonholonomic Mechanics and Control}.
\newblock New York: Springer-Verlag.

\bibitem{bonmaro09}
Bonnabel,~S., Martin,~P., \& Rouchon,~P. (2009).
\newblock Nonlinear symmetry-preserving observers on {L}ie groups.
\newblock {\em IEEE Transactions on Automatic Control}, 54(7), 1709--1713.

\bibitem{Bras2013CDC}
Br{\'a}s,~S., Izadi,~M., Silvestre,~C., Sanyal,~A., \& Oliveira,~P. (2013).
\newblock Nonlinear Observer for 3D Rigid Body Motion.
\newblock In {\em Proceedings of the $52^{nd}$ IEEE Conference on Decision and Control}. Florence, Italy.

\bibitem{Sergio2}
Br{\'a}s,~S., Cunha,~R., Vasconcelos,~J., Silvestre,~C., \& Oliveira,~P. (2010).
\newblock Experimental evaluation of a nonlinear attitude observer based on image and inertial measurements.
\newblock In {\em Proceedings of the American Control Conference} (pp. 4552--4557). Baltimore, MD.

\bibitem{Sergio1}
Br{\'a}s,~S., Silvestre,~C., Oliveira,~P., Vasconcelos,~J., \& Cunha,~R. (2011).
\newblock An Experimentally Evaluated Attitude Observer Based on Range and Inertial Measurements.
\newblock In {\em Proceedings of the $18^{th}$ IFAC World Congress}. Milan, Italy.

\bibitem{strand1}
Bridges,~C., Kenyon,~S., Underwood,~C., \& Sweeting,~M. (2011).
\newblock STRaND: Surrey Training Research and Nanosatellite Demonstrator.
\newblock In {\em Proceedings of the $1^{st}$ IAA Conference on University Satellite Missions and CubeSat Workshop in Europe}. Roma, Italy.

\bibitem{Bullo_ECC95}
Bullo,~F., \& Murray,~R.~M. (1995).
\newblock Proportional derivative ({PD}) control on the {E}uclidean group.
\newblock In {\em Proceedings of the European Control Conference} (pp. 1091--1097). Roma, Italy.

\bibitem{chaturvedi2011rigid}
Chaturvedi,~N.~A., Sanyal,~A.~K., \& McClamroch,~N.~H. (2011).
\newblock Rigid-body attitude control.
\newblock {\em IEEE Control Systems Magazine}, 31(3), 30--51.

\bibitem{dfpu09}
Davila,~J., Fridman,~L., Pisano,~A., \& Usai,~E. (2009).
\newblock Finite-time state observation for nonlinear uncertain systems via higher-order sliding modes.
\newblock {\em International Journal of Control}, 82(8), 1564--1574.

\bibitem{dimassi2012continuously}
Dimassi,~H., Loria,~A., \& Belghith,~S. (2012).
\newblock Continuously-implemented sliding-mode adaptive unknown-input observers under noisy measurements.
\newblock {\em Systems \& Control Letters}, 61(12), 1194--1202.

\bibitem{Dorato2006overview}
Dorato,~P. (2006).
\newblock An overview of finite-time stability.
\newblock {\em Current Trends in Nonlinear Systems and Controln}, 185--194.

\bibitem{jo:solwahba}
Farrell,~J., Stuelpnagel,~J., Wessner,~R., Velman,~J., \& Brock,~J. (1966).
\newblock A least squares Estimate of Satellite Attitude, Solution 65-1.
\newblock {\em SIAM Review}, 8(3), 384--386.

\bibitem{Forbes2014Automatica}
Forbes,~J.~R., de~Ruiter,~A.~H.~J., \& Zlotnik,~D.~E. (2014).
\newblock Continuous-time norm-constrained Kalman filtering.
\newblock {\em Automatica}, 50(10), 2546--2554.

\bibitem{gold}
Goldstein,~H. (1980).
\newblock {\em Classical Mechanics} (2$^{nd}$ edition).
\newblock Pearson Education India.

\bibitem{GoodPhD}
Goodarzi,~F.~A. (2015).
\newblock {\em Geometric Nonlinear Controls for Multiple Cooperative Quadrotor UAVs Transporting a Rigid Body}.
\newblock Ph.D. Dissertation. The George Washington University, Washington, D.C., USA.

\bibitem{Good2015ACC}
Goodarzi,~F.~A., \& Lee,~T. (2015).
\newblock Dynamics and Control of Quadrotor UAVs Transporting a Rigid Body Connected via Flexible Cables.
\newblock In {\em Proceedings of the American Control Conference} (pp. 4677--4682). Chicago, IL, USA.

\bibitem{Good2013ECC}
Goodarzi,~F.~A., Lee,~D., \& Lee,~T. (2013).
\newblock Geometric nonlinear {PID} control of a quadrotor {UAV} on {SE}(3).
\newblock In {\em Proceedings of the European Control Conference} (pp. 3845--3850). Zurich, Switzerland.

\bibitem{Good2014ASME}
Goodarzi,~F.~A., Lee,~D., \& Lee,~T. (2014).
\newblock Geometric Adaptive Tracking Control of a Quadrotor UAV on SE(3) for Agile Maneuvers.
\newblock {\em ASME Journal of Dynamic Systems, Measurement and Control}, 137(9), 091007.

\bibitem{Good2014ACC}
Goodarzi,~F.~A., Lee,~D., \& Lee,~T. (2014).
\newblock Geometric stabilization of a quadrotor UAV with a payload connected by flexible cable.
\newblock In {\em Proceedings of the American Control Conference} (pp. 4925--4930). Portland, OR, USA.

\bibitem{Good2015IJCAS}
Goodarzi,~F.~A., Lee,~D., \& Lee,~T. (2015).
\newblock Geometric Control of a Quadrotor UAV Transporting a Payload Connected via Flexible Cable.
\newblock {\em International Journal of Control, Automation, and Systems}, 13(6), 1--13.

\bibitem{green}
Greenwood,~D.~T. (1977).
\newblock {\em Classical Dynamics}.
\newblock Courier Corporation.

\bibitem{GNSS}
Grip,~H.~F., Fossen,~T.~I., Johansen,~T.~A., \& Saberi,~A. (2012).
\newblock A nonlinear observer for integration of {GNSS} and {IMU} measurements with gyro bias estimation.
\newblock In {\em Proceedings of the American Control Conference} (pp. 4607--4612). Montreal, Canada.

\bibitem{Haichao2015}
Gui,~H., \& Vukovich,~G. (2015).
\newblock Dual-quaternion-based adaptive motion tracking of spacecraft with reduced control effort.
\newblock {\em Nonlinear Dynamics}, 1--18.

\bibitem{Haddad2008finite}
Haddad,~M.~M., Nersesov,~S.~G., \& Du,~L. (2008).
\newblock Finite-time stability for time-varying nonlinear dynamical systems.
\newblock In {\em Proceedings of the American Control Conference} (pp. 4135--4139). Seattle, WA, USA.

\bibitem{haluwa}
Hairer,~E., Lubich,~C., \& Wanner,~G. (2006).
\newblock {\em Geometric Numerical Integration}.
\newblock volume 31 of Springer Series in Computational Mathematics.

\bibitem{Hua}
Hua,~M.D., Zamani,~M., Trumpf,~J., Mahony,~R., \& Hamel,~T. (2011).
\newblock Observer design on the special euclidean group {SE}(3).
\newblock In {\em Proceedings of the 50th IEEE Conference on Decision and
  Control and European Control Conference} (pp. 8169--8175). Orlando, FL, USA.

\bibitem{Izadi2013DSCC}
Izadi,~M., Bohn,~J., Lee,~D., Sanyal,~A., Butcher,~E., \& Scheeres,~D. (2013).
\newblock A Nonlinear Observer Design for a Rigid Body in the Proximity of a Spherical Asteroid.
\newblock In {\em Proceedings of the ASME Dynamic Systems and Control Conference}. Palo Alto, CA, USA.

\bibitem{ICRA2015}
Izadi,~M., Samiei,~E., Sanyal,~A.~K., \& Kumar,~V. (2015).
\newblock Comparison of an attitude estimator based on the
  {L}agrange-d'{A}lembert principle with some state-of-the-art filters.
\newblock In {\em Proceedings of the IEEE International Conference on Robotics and
  Automation} (pp. 2848--2853). Seattle, WA, USA.

\bibitem{Automatica}
Izadi,~M., \& Sanyal,~A.~K. (2014).
\newblock Rigid body attitude estimation based on the {L}agrange-d'{A}lembert
  principle.
\newblock {\em Automatica}, 50(10), 2570--2577.

\bibitem{Automatica2}
Izadi,~M., \& Sanyal,~A.~K. (2015).
\newblock Rigid body pose estimation based on the {L}agrange-d'{A}lembert principle.
\newblock To appear in {\em Automatica}.

\bibitem{ACC2015}
Izadi,~M., Sanyal,~A.~K., Samiei,~E., \& Viswanathan,~S.~P. (2015).
\newblock Discrete-time rigid body attitude state estimation based on the
  discrete {L}agrange-d'{A}lembert principle.
\newblock In {\em Proceedings of the American Control Conference} (pp. 3392--3397). Chicago, IL, USA.

\bibitem{strand2}
Kenyon,~S., Bridges,~C., Liddle,~D., Dyer,~R., Parsons,~J., Feltham,~D., Taylor,~R., Mellor,~D., Schofield,~A., \& Linehan,~R. (2011).
\newblock STRaND-1: Use of a \$500 Smartphone as the Central Avionics of a Nanosatellite.
\newblock In {\em Proceedings of the $62^{nd}$ International Astronautical Congress}. Cape Town, South Africa.

\bibitem{khal}
Khalil,~H.~K. (2001).
\newblock {\em Nonlinear Systems} (3$^{rd}$ edition).
\newblock Prentice Hall, Upper Saddle River, NJ.

\bibitem{Khosravian2015Recursive}
Khosravian,~A., Trumpf,~J., Mahony,~R., \& Hamel,~T. (2015).
\newblock Recursive Attitude Estimation in the Presence of Multi-rate and Multi-delay Vector Measurements.
\newblock In {\em Proceedings of the American Control Conference} (pp. 3199--3205). Chicago,
  IL, USA.

\bibitem{Khosravian2015observers}
Khosravian,~A., Trumpf,~J., Mahony,~R., \& Lageman,~C. (2015).
\newblock Observers for invariant systems on Lie groups with biased input
  measurements and homogeneous outputs.
\newblock {\em Automatica}, 55, 19--26.

\bibitem{kirk}
Kirk,~D.~E. (1971).
\newblock {\em Optimal Control Theory: An Introduction}.
\newblock Prentice Hall, NY.

\bibitem{Lageman}
Lageman,~C., Trumpf,~J., \& Mahony,~R. (2010).
\newblock Gradient-like observers for invariant dynamics on a {L}ie group.
\newblock {\em IEEE Transactions on Automatic Control}, 55, 367--377.

\bibitem{Lee2014AAS}
Lee,~D., Izadi,~M., Sanyal,~A., Butcher,~E., \& Scheeres,~D. (2014).
\newblock Finite-Time Control for Body-Fixed Hovering of Rigid Spacecraft over an Asteroid.
\newblock In {\em Proceedings of the AAS/AIAA Space Flight Mechanics}. Santa Fe, NM, USA.

\bibitem{mclele}
Lee,~T., Leok,~M., \& McClamroch,~N.~H. (2005).
\newblock A Lie Group Variational Integrator for the Attitude Dynamics of a Rigid Body with Applications to the {3D} Pendulum.
\newblock In {\em Proceedings of the IEEE Conference on Control Applications} (pp. 962--967). Toronto, Canada.

\bibitem{leishman2014relative}
Leishman,~R.~C., McLain,~T.~W., \& Beard,~R.~W. (2014).
\newblock Relative navigation approach for vision-based aerial {GPS}-denied
  navigation.
\newblock {\em Journal of Intelligent \& Robotic Systems}, 74(1-2), 97--111.

\bibitem{mahapf08}
Mahony,~R., Hamel,~T., \& Pflimlin,~J.~M. (2008).
\newblock Nonlinear complementary filters on the special orthogonal group.
\newblock {\em IEEE Transactions on Automatic Control}, 53(5), 1203--1218.

\bibitem{mabeda04}
Maithripala,~D.~H., Berg,~J.~M., \& Dayawansa,~W.~P. (2004).
\newblock An intrinsic observer for a class of simple mechanical systems on a
  {L}ie group.
\newblock In {\em Proceedings of the American Control Conference} (pp. 1546--1551). Boston,
  MA, USA.

\bibitem{markley1988attitude}
Markley,~F.~L. (1988).
\newblock Attitude determination using vector observations and the singular value decomposition.
\newblock {\em The Journal of the Astronautical Sciences}, 36(3), 245--258.

\bibitem{ekf2}
Markley,~F.~L. (2003).
\newblock Attitude Error Representations for {K}alman Filtering.
\newblock {\em AIAA Journal of Guidance, Control, and Dynamics}, 26(2), 311--317.

\bibitem{markSO3}
Markley,~F.~L. (2006).
\newblock Attitude filtering on $\SO$.
\newblock {\em The Journal of the Astronautical Sciences}, 54(4), 391--413.

\bibitem{marat}
Marsden,~J.~E., \& Ratiu,~T.~S. (1999).
\newblock {\em Introduction to mechanics and symmetry: a basic exposition of
  classical mechanical systems} (Vol. 17).
\newblock Springer Science \& Business Media.

\bibitem{marswest}
Marsden,~J.~E., \& West,~M. (2001).
\newblock Discrete mechanics and variational integrators.
\newblock {\em Acta Numerica}, 10, 357--514.

\bibitem{Phonesat1}
Marshall,~W., \& Beukelaers,~V. (2011).
\newblock PhoneSat: a smartphone-based spacecraft bus.
\newblock In {\em Proceedings of the $62^{nd}$ International Astronautical Congress}. Cape Town, South Africa.

\bibitem{miller2014tracking}
Miller,~A., \& Miller,~B. (2014).
\newblock Tracking of the UAV trajectory on the basis of bearing-only
  observations.
\newblock In {\em Proceedings of the 53rd Annual Conference on Decision and Control} (pp. 4178--4184). Los Angeles, CA, USA.

\bibitem{MilSch}
Miller,~J., Konopliv,~A., Antreasian,~P., Bordi,~J., Chesley,~S., Helfrich,~C., Owen,~W., Wang,~T., Williams,~B., Yeomans,~D., \& Scheeres,~D. (2002).
\newblock Determination of Shape, Gravity, and Rotational State of Asteroid 433 Eros.
\newblock {\em Icarus}, 155(1), 3--17.

\bibitem{bo:miln}
Milnor,~J. (1963).
\newblock {\em Morse Theory}.
\newblock Princeton University Press, Princteon, NJ.

\bibitem{Gaurav_ASR}
Misra,~G., Izadi,~M., Sanyal,~A.~K., \& Scheeres,~D.~J. (2015).
\newblock Coupled orbit–attitude dynamics and relative state estimation of
  spacecraft near small Solar System bodies.
\newblock {\em Advances in Space Research}.

\bibitem{Mortensen}
Mortensen,~R.~E. (1968).
\newblock Maximum-likelihood recursive nonlinear filtering.
\newblock {\em Journal of Optimization Theory and Applications}, 2(6), 386--394.

\bibitem{rehbinder2003pose}
Rehbinder,~H., \& Ghosh,~B.~K. (2003).
\newblock Pose estimation using line-based dynamic vision and inertial sensors.
\newblock {\em IEEE Transactions on Automatic Control}, 48(2), 186--199.

\bibitem{RodImp}
Rodrigues,~S.~S., Crasta,~N., Aguiar,~A.~P., \& Leite,~F.~S. (2010).
\newblock State Estimation for Systems on {SE}(3) with Implicit Outputs: An Application to Visual Servoing.
\newblock In {\em Proceedings of the 8th IFAC Symposium on Nonlinear Control Systems}. Bologna, Italy.

\bibitem{san06}
Sanyal,~A.~K. (2006).
\newblock Optimal Attitude Estimation and Filtering Without Using Local Coordinates, Part I: Uncontrolled and Deterministic Attitude Dynamics.
\newblock In {\em Proceedings of the American Control Conference} (pp. 5734--5739). Minneapolis, MN, USA.

\bibitem{CDC2013Sanyal}
Sanyal,~A.~K., Bohn,~J., \& Bloch,~A.~M. (2013).
\newblock Almost global finite time stabilization of rigid body attitude dynamics.
\newblock In {\em Proceedings of the $52^{nd}$ IEEE Conference on Decision and Control} (pp. 3261-3266). Florence, Italy.

\bibitem {AKS2008GNC}
Sanyal,~A.~K., \& Chaturvedi,~N.~A. (2008).
\newblock Almost global robust attitude tracking control of spacecraft in gravity.
\newblock In {\em Proceedings of the AIAA Guidance, Navigation, and Control Conference} (pp. AIAA-2008-6979). Honolulu, HI, USA.

\bibitem{sanyal2009inertia}
Sanyal,~A.~K., Fosbury,~A., Chaturvedi,~N.~A., \& Bernstein,~D.~S. (2009).
\newblock Inertia-free spacecraft attitude tracking with disturbance rejection
  and almost global stabilization.
\newblock {\em Journal of Guidance, Control, and Dynamics}, 32(4), 1167--1178.

\bibitem{Izadi2014DSCC}
Sanyal,~A., Izadi,~M., \& Bohn,~J. (2014).
\newblock An Observer for Rigid Body Motion with Almost Global Finite-Time Convergence.
\newblock In {\em Proceedings of the ASME Dynamic Systems and Control Conference}. San Antonio, TX, USA.

\bibitem{ast_acc14}
Sanyal,~A.~K., Izadi,~M., \& Butcher,~E.~A. (2014).
\newblock Determination of relative motion of a space object from simultaneous
  measurements of range and range rate.
\newblock In {\em Proceedings of the American Control Conference} (pp. 1607--1612). Portland, OR, USA.

\bibitem {Space2014}
Sanyal,~A.~K., Izadi,~M., Misra,~G., Samiei,~E., \& Scheeres,~D.~J. (2014).
\newblock Estimation of Dynamics of Space Objects from Visual Feedback During Proximity Operations.
\newblock In {\em Proceedings of the AIAA Space Conference}. San Diego, CA, USA.

\bibitem{SCLpaper}
Sanyal,~A.~K., Lee,~T., Leok,~M., \& McClamroch,~N.~H. (2008).
\newblock Global optimal attitude estimation using uncertainty ellipsoids.
\newblock {\em Systems \& Control Letters}, 57(3), 236--245.

\bibitem{jgcd12}
Sanyal,~A.~K., \& Nordkvist,~N. (2012).
\newblock Attitude state estimation with multi-rate measurements for almost
  global attitude feedback tracking.
\newblock {\em AIAA Journal of Guidance, Control, and Dynamics}, 35(3), 868--880.

\bibitem{Sanyal2011almost}
Sanyal,~A.~K., Nordkvist,~N., \& Chyba,~M. (2011).
\newblock An almost global tracking control scheme for maneuverable autonomous
  vehicles and its discretization.
\newblock {\em IEEE Transactions on Automatic Control}, 56(2), 457--462.

\bibitem{schaub2009analytical}
Schaub,~H., \& Junkins,~J. (2009).
\newblock {\em Analytical Mechanics of Space Systems}.
\newblock AIAA. Reston, VA.

\bibitem{Seba1}
Seba,~A., El Hadri,~A., Benziane,~L., \& Benallegue,~A. (2014).
\newblock Multiplicative Extended Kalman Filter based on Visual Data for Attitude Estimation.
\newblock In {\em the IEEE International Conference on Robotics and
  Automation} Workshop on Modelling, Estimation, Perception, and Control of All Terrain Mobile Robots. Hong Kong.

\bibitem{Seba2}
Seba,~A., El Hadri,~A., Benziane,~L., \& Benallegue,~A. (2014).
\newblock Attitude Estimation Using Line-Based Vision and Multiplicative Extended Kalman Filter.
\newblock In {\em Proceedings of the 13$^{th}$ International Conference on Control, Automation, Robotics, and Vision} (pp. 456--461). Singapore.

\bibitem{shen2013vision}
Shen,~S., Mulgaonkar,~Y., Michael,~N., \& Kumar,~V. (2013).
\newblock Vision-based state estimation and trajectory control towards
  aggressive flight with a quadrotor.
\newblock In {\em Proceedings of the Robotics Science and Systems}.

\bibitem{shen2013rotor}
Shen,~S., Mulgaonkar,~Y., Michael,~N., \& Kumar,~V. (2013).
\newblock Vision-based state estimation for autonomous rotorcraft {MAV}s in
  complex environments.
\newblock In {\em Proceedings of the IEEE International Conference on Robotics and
  Automation} (pp. 1758--1764). Karlsruhe, Germany.

\bibitem{Sommer}
Sommer,~H., Forbes,~J.~R., Siegwart,~R., Furgale,~P. (2015).
\newblock Continuous-Time Estimation of Attitude Using B-Splines on Lie Groups.
\newblock {\em AIAA Journal of Guidance, Control, and Dynamics}.

\bibitem{6315375}
Swensen,~J.~P., \& Cowan,~N.~J. (2012).
\newblock An almost global estimator on $\SO$ with measurement on $\ensuremath{\mathrm{S}}^2$.
\newblock In {\em Proceedings of the American Control Conference} (pp. 1780 -1786). Montreal, Canada.

\bibitem{TakSch}
Takahashi,~Y., \& Scheeres,~D.~J. (2011).
\newblock Small-Body Postrendezvous Characterization via Slow Hyperbolic Flybys.
\newblock {\em Journal of Guidance, Control, and Dynamics}, 34(6), 1815--1827.

\bibitem{Tayebi2011}
Tayebi,~A., Roberts,~A., \& Benallegue,~A. (2011).
\newblock Inertial measurements based dynamic attitude estimation and velocity-free attitude stabilization.
\newblock In {\em Proceedings of the American Control Conference} (pp. 1027--1032). San Francisco, CA, USA.

\bibitem{Vas1}
Vasconcelos,~J.~F., Cunha,~R., Silvestre,~C., \& Oliveira,~P. (2010).
\newblock A nonlinear position and attitude observer on {SE}(3) using landmark
  measurements.
\newblock {\em Systems \& Control Letters}, 59, 155--166.

\bibitem{silvest08}
Vasconcelos,~J.~F., Silvestre,~C., \& Oliveira,~P. (2008).
\newblock A nonlinear {GPS/IMU} based observer for rigid body attitude and
  position estimation.
\newblock In {\em Proceedings of the IEEE Conference on Decision and Control} (pp. 1255--1260). Cancun, Mexico.

\bibitem {ICC2015}
Viswanathan,~S.~P., Sanyal,~A.~K., \& Izadi,~M. (2015).
\newblock Mechatronics Architecture of Smartphone Based Spacecraft ADCS using VSCMG Actuators.
\newblock In {\em Proceedings of the Indian Control Conference}. Chennai, India.

\bibitem{jo:wahba}
Wahba,~G. (1965).
\newblock A least squares estimate of satellite attitude, Problem 65-1.
\newblock {\em SIAM Review}, 7(5), 409.

\bibitem{Yu20051957}
Yu,~S., Yu,~X., Shirinzadeh,~B., \& Man,~Z. (2005).
\newblock Continuous finite-time control for robotic manipulators with terminal sliding mode.
\newblock {\em Automatica}, 41, 1957--1964.

\bibitem{ZamPhD}
Zamani,~M. (2013).
\newblock {\em Deterministic Attitude and Pose Filtering, an Embedded {L}ie
  Groups Approach}.
\newblock Ph.D. Dissertation. Australian National University, Canberra, Australia.

\bibitem{zatruma11}
Zamani,~M., Trumpf,~J., \& Mahony,~R. (2011).
\newblock Near-optimal deterministic filtering on the rotation group.
\newblock {\em IEEE Transactions on Automatic Control}, 56(6), 1411--1414.

\bibitem{Zamani2012second}
Zamani,~M., Trumpf,~J., \& Mahony,~R. (2012).
\newblock A second order minimum-energy filter on the special orthogonal group.
\newblock In {\em Proceedings of the American Control Conference} (pp. 1895--1900). Montreal, Canada.

\bibitem{Zamani2013minimum}
Zamani,~M., Trumpf,~J., \& Mahony,~R. (2013).
\newblock Minimum-energy filtering for attitude estimation.
\newblock {\em IEEE Transactions on Automatic Control}, 58(11), 2917--2921.

\bibitem{Zarrouati}
Zarrouati,~N., Aldea,~E., \& Rouchon,~P. (2012).
\newblock $\SO$-invariant asymptotic observers for dense depth field estimation based on visual data and known camera motion.
\newblock In {\em Proceedings of the American Control Conference} (pp. 4116--4123). Montreal, Canada.

\bibitem{Zlotnik}
Zlotnik,~D.~E., \& Forbes,~J.~R. (2015).
\newblock Differential Geometric SLAM.
\end{thebibliography}
%

\end{document}